\documentclass[a4paper,11pt,twoside]{article}
\usepackage{geometry}
\usepackage{fancyhdr}
\usepackage{graphicx}
\usepackage{amssymb}
\usepackage{amsmath}
\usepackage{amsfonts}
\usepackage{theorem}
\usepackage{empheq}
\usepackage{mathrsfs}
\usepackage{mathtools}
\usepackage{bm}
\usepackage{color}
\usepackage{setspace}
\usepackage{exscale}
\usepackage{relsize}
\usepackage{float}
\usepackage{cite}
\usepackage{nicefrac}
\usepackage{setspace}
\usepackage{algorithm,algorithmic}
\usepackage{tikz}
\usetikzlibrary{positioning}
\usepackage{xcolor}
\usepackage{rotating}
\usepackage{enumerate}
\usepackage{hyperref}
\usepackage{authblk}
\usepackage{textcomp}
\usepackage[title]{appendix}
\usepackage{cleveref}
\usepackage{cancel}

\usepackage{booktabs, multirow, threeparttable}
\usepackage{paralist} 
\usepackage{verbatim} 
\usepackage{subfigure} 

\newtheorem{theorem}{Theorem}[section]

{\theorembodyfont{\rmfamily}\newtheorem{example}{Example}[section]}
{\theorembodyfont{\rmfamily}}
\newtheorem{rmk}{Remark}[section]

\newenvironment{proof}{{\flushleft \bf Proof:}}{\hfill$\square$\par}

\newenvironment{acknowledgment}{{\flushleft \bf Acknowledgment:}}{}

\allowdisplaybreaks[1]

\numberwithin{equation}{section}
\numberwithin{figure}{section}
\numberwithin{table}{section}

\usepackage{caption}
\newcommand{\hf}{\frac{1}{2}}
\newcommand{\xL}{x_{j-\frac{1}{2}}}
\newcommand{\xR}{x_{j+\frac{1}{2}}}
\newcommand{\yL}{y_{k-\frac{1}{2}}}
\newcommand{\yR}{y_{k+\frac{1}{2}}}
\newcommand{\jL}{j-\frac{1}{2}}
\newcommand{\jR}{j+\frac{1}{2}}
\newcommand{\kL}{k-\frac{1}{2}}
\newcommand{\kR}{k+\frac{1}{2}}

\newcommand{\bnabla}{\bm{\nabla}}
\newcommand{\eps}{\varepsilon}
\newcommand{\dt}{\Delta t}
\newcommand{\dx}{\Delta x}
\newcommand{\dy}{\Delta y}
\newcommand{\bx}{\bm{x}}
\newcommand{\bn}{\bm{n}}
\newcommand{\bv}{\bm{v}}
\newcommand{\bU}{\bm{U}}
\newcommand{\bV}{\bm{V}}
\newcommand{\bF}{\bm{F}}
\newcommand{\bG}{\bm{G}}
\newcommand{\bS}{\bm{S}}
\newcommand{\bQ}{\bm{Q}}

\newcommand{\g}{\textsl{g}}
\newcommand{\Ro}{\mathrm{Ro}}
\newcommand{\Bu}{\mathrm{Bu}}
\renewcommand{\d}{\mathrm{d}}

\newcommand\eref[1]{(\ref{#1})}

\newcommand*\xbar[1]{%
	\hbox{%
		\vbox{%
			\hrule height 0.5pt 
			\kern0.3ex
			\hbox{%
				\kern-0.05em
				\ensuremath{#1}%
				\kern-0.02em
			}%
		}%
	}%
}

\graphicspath{{Figures/}}

\begin{document}
	
\title{An Asymptotic-Preserving Dual Formulation Finite-Volume Method for the Thermal Rotating Shallow Water Equations}
\author{Alina Chertock\thanks{Department of Mathematics, North Carolina State University, Raleigh, NC 27695, USA;
{\tt chertock@math.ncsu.edu}}, Alexander Kurganov\thanks{Department of Mathematics and Shenzhen International Center for Mathematics,
Southern University of Science and Technology, Shenzhen, 518055, China; {\tt alexander@sustech.edu.cn}}, Lorenzo
Micalizzi\thanks{Department of Mathematics, North Carolina State University, Raleigh, NC 27695, USA; {\tt lmicali@ncsu.edu}}, and Nan
Zhang\thanks{Department of Mathematics, Southern University of Science and Technology, Shenzhen, 518055, China;
{\tt zhangn@sustech.edu.cn}}}
\date{}
\maketitle

\begin{abstract}
We propose a new second-order asymptotic-preserving (AP) dual formulation finite-volume (DF-FV) method for the thermal rotating shallow
water (TRSW) equations. The TRSW system models geophysical flows characterized by horizontal temperature/density variations, exhibiting
multi-scale dynamics due to the coexistence of fast rotational waves and slower advective processes. To efficiently address challenges
associated with the multiscale nature of the TRSW system, we follow the DF-FV framework and develop a DF-FV method, in which both the
conservative and nonconservative (primitive) forms of the equations are simultaneously solved, allowing the method to exploit the
complementary strengths of each representation across different flow regimes. The primitive formulation is better suited for preserving the
correct asymptotic behavior in nearly thermal quasi-geostrophic (TQG) regimes characterized by a low Rossby number, while the conservative
formulation is essential for robust shock capturing in high-Rossby-number regimes, in which nonconservative discretizations may fail to
converge to physically relevant weak solutions.

In order to achieve an AP property of the primitive system discretization, we introduce a special hyperbolic splitting to separate nonstiff
and stiff parts of the TRSW system. The nonstiff terms are then discretized explicitly using the path-conservative central-upwind approach,
while the stiff terms are approximated in a semi-implicit manner using central differences in space: this way, no restrictive time-step
constraints are to be imposed. The conservative system, on the other hand, is discretized explicitly using the central-upwind scheme, which
is stable in low-Rossby-number regimes only. The AP property and stability of the resulting DF-FV method are enforced by applying a
post-processing coupling that recovers the correct TQG dynamics in the zero-Rossby-number limit while preserving shock speeds and preventing
spurious oscillations in high-Rossby-number regimes. Importantly, the semi-implicit discretization requires only the solution of linear
well-posed elliptic problems, avoiding nonlinear solvers and reducing computational cost. Numerical experiments demonstrate that the
resulting AP DF-FV method is uniformly accurate across a broad range of flow regimes, efficient to implement, and robust in capturing both
asymptotic limits and nonlinear wave phenomena.
\end{abstract}



\section{Introduction}
In geophysical fluid dynamics, particularly for large-scale atmospheric and oceanic flows, rotating shallow water (RSW) equations serve as 
fundamental models. They capture essential features of geophysical flows under the influence of the Coriolis force induced by the Earth's 
rotation; see, e.g., \cite{Pedlosky_1987,Zeitlin_2018}. However, the classical RSW equations neglect horizontal variations in
temperature/density, which can play a crucial role in stratified flows. An improvement in this respect is thermal rotating shallow water 
(TRSW) model, which was multiply reinvented both in the meteorological and oceanographic literature, in the context of the boundary layer in
the atmosphere \cite{Lavoie_1972,Salby_1989}, and of the mixed layer in the ocean \cite{McCreary_1993,Ripa_1995,Young_1994}. It was later
applied to planetary atmospheres \cite{Cho_2008,Warneford_2014}, and was rediscovered again in the context of testing the general
circulation models \cite{Zerroukat_2015}.

The TRSW equations are particularly relevant when horizontal scales exceed vertical scales by orders of magnitude, as in weather systems 
spanning thousands of kilometers horizontally but only tens of kilometers vertically; see, e.g., \cite{Majda_2003,Zeitlin_2018}. The
dynamics of such systems are characterized by several nondimensional parameters. Among them, the Rossby number plays a central role: it
measures the ratio of the rotational to the advective time scales. When the Rossby number is large (order one), rotational effects are weak,
and the dynamics are dominated by compressible, wave-like phenomena. In contrast, when the Rossby number is small, rotation strongly
constrains the flow, resulting in slow, balanced dynamics that are well described by quasi-geostrophic models. In this regime, the TRSW
equations reduce to the thermal quasi-geostrophic (TQG) system, which provides a simplified yet accurate description of large-scale balanced
flows.

This multiscale nature of the TRSW equations creates severe numerical difficulties. Explicit schemes  for such equations
\cite{Eldred_2019,Kurganov_2020,Kurganov_2021,Ricardo_2024,Tambyah_2025}, while effective in capturing shocks and discontinuities in
high-Rossby-number regimes, become prohibitively expensive in low-Rossby-number regimes, as they require extremely fine grids and very small
time steps to resolve the fast inertial oscillations imposed by rotation. Fully implicit methods can, in principle, circumvent these
restrictions, but they lead to large, nonlinear systems that are costly to solve and may lack robustness, leading to convergence to the
wrong limiting solutions. To overcome this difficulty, hybrid approaches such as implicit-explicit (IMEX)
\cite{Ascher_1997,Boscarino_2013,Boscarino_2018,Hu_2025,Kennedy_2003} and semi-implicit (SI)
\cite{Boscarino_2016,Boscarino_2022,Chertock_2015} methods have been developed. By treating stiff terms (semi-)implicitly and nonstiff terms
explicitly, IMEX and SI schemes provide a balance between stability and efficiency. However, in strongly multiscale systems like the TRSW
one, simply stabilizing the stiff terms is not enough: numerical methods must also reproduce the correct asymptotic behavior of the system
as the Rossby number vanishes. This requirement motivates the design of asymptotic-preserving (AP) schemes, which have been developed for a
variety of nonlinear hyperbolic partial differential equations including kinetic models 
\cite{Jin_1999,Filbet_2010,Dimarco_2013,Hu_2017,Abgrall_2020,Klar_1999,Zhang_2023,Anandan_2024}, compressible Euler equations 
\cite{Allegrini_2025,Boscarino_2019,Boscarino_2022,Boscheri_2020,CKKM, Degond_2007,Dimarco_2017,Haack_2012,Huang_2025,Klein_1995,Zeifang_2020},
shallow water equations \cite{Bispen_2014,Busto_2022,Huang_2022,Huang_2024,Liu_2020,Vater_2018,Xie_2024}, multi-layer shallow water
equations \cite{Couderc_2017,Duran_2017}, RSW equations \cite{Kurganov_2022,Liu_2019,Xie_2025,Zakerzadeh_2017,Zhang_2026}, AP methods
ensure that, as parameters, such as the Rossby (Mach, Froude) number, tend to zero, the numerical solution automatically converges to the
reduced asymptotic model without the need to either refine the mesh or reduce the size of time steps. At the same time, AP schemes remain
accurate and non-oscillatory in the opposite limit, where compressibility and shocks dominate.

In this work, we extend the AP method from \cite{CKKM} to the TRSW equations. This AP method is based on a dual formulation of the studied 
system: both the conservative and nonconservative (primitive) forms are used to evolve the numerical solution. The motivation for
considering both formulations arises from their complementary properties. In high-Rossby-number regimes, effective shock capturing requires
the conservative form of the equations, since discretizations of nonconservative forms may fail to converge to the physically relevant weak
solutions; see, e.g., \cite{Abgrall_2010,Hou_1994}. On the other hand, in low-Rossby-number regimes, numerical methods are typically based
on the nonconservative form, which is better suited for enforcing the divergence-free constraint and preserving the correct asymptotic
behavior as the Rossby (Mach, Froude) number tends to zero, where both formulations are expected to converge to the same physically
relevant limit.

The proposed second-order AP method for the TRSW equations, like the AP method from \cite{CKKM}, is constructed using the nonstaggered dual 
formulation finite-volume (DF-FV) method recently developed in \cite{CFKM} (for other recent works on the dual formulation methods, we refer
the reader to \cite{Abgrall_2023,Abgrall_2024,Abgrall_2025}. We consider the conservative and nonconservative (primitive) formulations of 
the TRSW equations, leveraging the advantages of both formulations to design a scheme that remains valid across the full spectrum of Rossby 
numbers. To achieve the AP property of the primitive system discretization, we introduce a special hyperbolic splitting to separate nonstiff
and stiff parts of the TRSW system, and employ a path-conservative central-upwind (PCCU) discretization in space
\cite{CKKM,Chu_2022,Castro_2019}, which handles nonconservative terms and captures shocks robustly, together with an SI Runge-Kutta time
integration that stabilizes the fast oscillatory dynamics with a time-step restriction independent of the Rossby number. The AP property is
enforced by solving both the conservative and primitive forms of the TRSW equations simultaneously and by applying a post-processing
coupling at each time step to recover the correct TQG dynamics in the low-Rossby limit while preserving shock speeds in high-Rossby-number
regimes. It is important to emphasize that the incorporated SI time discretization does not require solving nonlinear systems: only linear
elliptic equations are solved at each time step, which significantly reduces computational cost, while maintaining stability and accuracy.
As a result, the proposed numerical framework offers a numerical method that is uniformly accurate across different-Rossby-number regimes,
efficient to implement, and robust in capturing both asymptotic limits and nonlinear wave phenomena.

The remainder of the paper is organized as follows. In \S\ref{sec2}, we introduce the TRSW equations and revisit the TQG limit of the 
nondimensional TRSW equations. In \S\ref{sec3}, we present a hyperbolic splitting, which is used in conjunction with the SI time
discretization to design an efficient AP scheme for the augmented primitive system of the TRSW equations; this method is expected to be 
efficient and accurate in low-Rossby-number regimes. In \S\ref{sec4}, we use the DF-FV framework in order to design a numerical method
capable of capturing shocks in high-Rossby-number regimes without compromising the AP property. To this end, we develop a central-upwind 
(CU) scheme for the conservative formulation of the TRSW equations and then introduce a post-processing technique that combines the
solutions of the primitive and conservative systems in a way that results in a DF-FV method applicable across all-Rossby-number regimes. In 
\S\ref{sec5}, we present numerical experiments to demonstrate the performance of the proposed AP DF-FV method. Conclusions are provided in 
\S\ref{sec6}. 

\section{TRSW Equations and Thermal Quasi-Geostrophic Limit}\label{sec2}
We consider the two-dimensional (2-D) TRSW equations, which can be written in the form \cite{Warneford_2013,Gouzien_2017}:
\begin{equation}
\begin{aligned}
&\frac{{\rm D}\bv}{{\rm D}t}+f\bv^\perp=-\Theta\nabla h-\hf h\nabla\Theta,\\
&\frac{{\rm D}h}{{\rm D}t}+h\bnabla\!\cdot\!\bv=0,\\
&\frac{{\rm D}\Theta}{{\rm D}t}=0,
\end{aligned}
\label{2.1}
\end{equation}
where $\bx=(x,y)$ denotes the horizontal spatial coordinates, $t$ is time, $\bv(\bx,t)=(u(\bx,t),v(\bx,t))^\top$ represents the horizontal
velocity field, $\bv^\perp:=(-v,u)^\top$, $\nabla:=\big(\frac{\partial}{\partial_x},\frac{\partial}{\partial_y}\big)$, and 
$\frac{{\rm D}}{{\rm D}t}:=\frac{\partial}{\partial t}+\bv\!\cdot\!\nabla$ is the material derivative. In addition, $h(\bx,t)$ represents
the thickness of the fluid layer, $f(y)=f_0+\beta y$ is the Coriolis parameter ($f_0$ and $\beta$ are nonnegative constants; $\beta=0$
corresponds to the simplest $f$-plane approximation), and $\Theta(\bx,t)$ is the buoyancy. In the special case where $\Theta\equiv\g$ (a
constant acceleration due to gravity), the TRSW system \eref{2.1} reduces to the standard RSW equations.

The TRSW system \eref{2.1} can be rewritten in the dimensionless form. To this end, we first introduce several reference quantities. 
Specifically, $L_0$, $V_0$, and $T_0=L_0/V_0$ denote the characteristic horizontal length, velocity, and time scales of the flow,
respectively. For the vertical direction, two characteristic scales are considered: the first scale is the mean depth of the fluid layer,
denoted by $H_0$; the second scale is the mean size of the depth perturbation $\phi:=h-H_0$, denoted by $\Phi_0$. Similarly, we introduce
the buoyancy perturbation $\theta:=\Theta-\Theta_0$, where $\Theta_0$ denotes the mean characteristic buoyancy scale. Note that
$\phi=\phi(\bx,t)$ and $\theta=\theta(\bx,t)$ are functions of the spatial variable $\bx$ and time $t$. We then define the following
nondimensional variables:
\begin{equation}
\bx^{\prime}:=\frac{\bx}{L_0},\quad\bv^{\prime}:=\frac{\bv}{V_0},\quad t^{\prime}:=\frac{t}{T_0},\quad h^{\prime}:=\frac{h}{H_0},
\quad\phi^{\prime}:= \frac{\phi}{\Phi_0},\quad\Theta^{\prime}:=\frac{\Theta}{\Theta_0},\quad\theta^{\prime}:=\frac{\theta}{\theta_0},
\label{2.2}
\end{equation}
where $\theta_0$ denotes the buoyancy perturbation scale. Note that \eref{2.2} yields the following relations: 
\begin{equation}
h^{\prime}=1+\frac{\Phi_0}{H_0}\phi^{\prime},\qquad\Theta^{\prime}=1+\frac{\theta_0}{\Theta_0}\theta^{\prime}.
\label{2.3}
\end{equation}		

We now substitute \eref{2.2} into \eref{2.1}, use \eref{2.3}, omit the superscript $'$ for simplicity, and obtain the following
nondimensional form of the TRSW equations:
\begin{equation}
\begin{aligned}
&\frac{{\rm D}\bv}{{\rm D}t}+\bigg(\frac{1}{\Ro}+\bar{\beta}y\bigg)\bv^\perp=-\frac{\Bu}{\Ro^2}\left(\frac{\Phi_0}{H_0}\nabla\phi
+\frac{\theta_0}{2\Theta_0}\nabla\theta+\frac{\Phi_0\theta_0}{2H_0\Theta_0}\big(\phi\nabla\theta+2\theta\nabla\phi\big)\right),\\
&\frac{{\rm D}\phi}{{\rm D}t}+\bigg(\frac{H_0}{\Phi_0}+\phi\bigg)\bnabla\!\cdot\!\bv=0,\\
&\frac{{\rm D}\theta}{{\rm D}t}=0.
\end{aligned}
\label{2.4}
\end{equation}
Here, we have introduced the dimensionless parameters
\begin{equation}
\Ro:=\frac{V_0}{L_0f_0},\quad\bar\beta:=\beta L_0T_0,\quad\mbox{and}\quad\Bu:=\bigg(\frac{L_R}{L_0}\bigg)^2=\frac{\Theta_0H_0}{L_0^2f_0^2},
\label{2.5}
\end{equation}
where $\Ro$ is the Rossby number, which quantifies the ratio of the characteristic rotation timescale $1/f_0$ to the advective timescale
$T_0$; $\Bu$ is the Burger number; and $L_R:=\sqrt{\Theta_0H_0}/f_0$ denotes the Rossby deformation radius. When the Rossby number is small,
the solution of \eref{2.4} is expected to be in the so-called TQG limit. To remain close to a thermal geostrophic equilibrium, the pressure
forces associated with the gradients of the geostrophic streamfunction $\psi:=\phi+\theta$ must be exactly balanced by the Coriolis effects,
that is, the following relation should be satisfied:
\begin{equation}
\frac{\Bu}{\Ro}\cdot\frac{\Phi_0}{H_0}=\frac{\Bu}{\Ro}\cdot\frac{\theta_0}{2\Theta_0}=1.
\label{2.6}
\end{equation}
In addition, we assume that the Rossby deformation radius $L_R$ is comparable to the characteristic horizontal length $L_0$, that is,
\begin{equation}
\Bu={\cal O}(1).
\label{2.7}
\end{equation}

In summary, in low-Rossby-number regimes, the scaling assumptions \eref{2.6} and \eref{2.7} define the so-called TQG scaling:
\begin{equation}
\Ro=\eps\ll1,\quad\Bu=\nu={\cal O}(1),\quad\frac{\Phi_0}{H_0}=\frac{\theta_0}{2\Theta_0}=\frac{\Ro}{\Bu}={\cal O}(\eps),
\label{2.8}
\end{equation}
in which the nondimensional TRSW equations \eref{2.4} become
\begin{subequations}
\begin{align}
&\frac{{\rm D}\bv}{{\rm D}t}+\Big(\frac{1}{\eps}+\bar\beta y\Big)\bv^\perp=-\frac{1}{\eps}\nabla\psi-
\frac{1}{\nu}\big(\phi\nabla\theta+2\theta\nabla\phi\big),\label{2.9a}\\
&\frac{{\rm D}\phi}{{\rm D}t}+\Big(\phi+\frac{\nu}{\eps}\Big)\bnabla\!\cdot\!\bv=0,\label{2.9b}\\
&\frac{{\rm D}\theta}{{\rm D}t}=0,\label{2.9c}
\end{align}
\label{2.9}
\end{subequations}
which serve as a starting point for the subsequent asymptotic reduction towards the TQG equations.

Next, we take the curl of the velocity equation \eref{2.9a} and obtain
\begin{equation}
\frac{{\rm D}\omega}{{\rm D}t}+\omega\bnabla\!\cdot\!\bv+\bnabla\!\cdot\!\big(\bar\beta y\bv\big)+\frac{1}{\eps}\bnabla\!\cdot\!\bv=
\frac{1}{\nu}[\phi,\theta],
\label{2.10}
\end{equation}
where we have used the following notations and identities: $\bnabla\!\times\!\bv=(0,0,\omega)^\top$, where $\omega:=v_x-u_y$ is the
vorticity, $[\phi,\theta]:=\phi_x\theta_y-\phi_y\theta_x$ is the 2-D Jacobian operator,
\begin{equation*}
\bnabla\!\times\!\bigg(\frac{{\rm D}\bv}{{\rm D}t}\bigg)=\Big(0,0,\frac{{\rm D}\omega}{{\rm D}t}+\omega\bnabla\!\cdot\!\bv\Big)^\top,\quad
\bnabla\!\times\!\big(\bar{\beta}y\bv^\perp\big)=\Big(0,0,\bnabla\!\cdot\!\big(\bar\beta y\bv\big)\Big)^\top,\quad
\bnabla\!\times\!\big(\phi\nabla\theta\big)=\big(0,0,[\phi,\theta]\big)^\top.
\end{equation*}
We then divide the depth perturbation \eref{2.9b} by $\nu$ and subtract it from \eref{2.10} and introduce the potential vorticity (PV) 
$q:=\omega+\bar\beta y-\phi/\nu$ to obtain the PV equation
\begin{equation}
\frac{{\rm D}q}{{\rm D}t}+q\bnabla\!\cdot\!\bv=\frac{1}{\nu}[\phi,\theta].
\label{2.11}
\end{equation}
We stress that the PV variable $q$ neither explicitly depends on the small parameter $\eps$ nor exhibits any stiff dependence on it. This 
structural property renders the PV equation \eref{2.11} particularly suitable for the design of an efficient AP scheme, as will be
demonstrated in \S\ref{sec3}. 

In order to formally derive the limiting TQG equations, we assume that the variables $\bm\pi:=(u,v,\psi,\phi,\theta)^\top$ admit an
asymptotic expansion with respect to $\eps$:
\begin{equation}
\bm\pi(\bx,t)=\bm\pi^{(0)}(\bx,t)+\eps\bm\pi^{(1)}(\bx,t)+\cdots.
\label{2.12}
\end{equation}
Substituting \eref{2.12} into the nondimensional TRSW equations \eref{2.9} together with the PV equation \eref{2.11}, and collecting the
terms of order ${\cal O}(\eps^{-1})$, we first obtain from the velocity equation \eref{2.9a} the geostrophic balance
\begin{equation}
\bv^{(0)}=\nabla^\perp\big(\phi^{(0)}+\theta^{(0)}\big)=\nabla^\perp\psi^{(0)},
\label{2.13}
\end{equation}
where $\psi^{(0)}:=\phi^{(0)}+\theta^{(0)}$ denotes the leading-order geostrophic streamfunction. Recalling that
$\bnabla\!\times\!\bv=(0,0,\omega)^\top$, equation \eref{2.13} directly implies the relation 
\begin{equation}
\omega^{(0)}=\Delta\psi^{(0)}.
\label{2.14}
\end{equation}
From the depth perturbation equation \eref{2.9b}, the leading-order ${\cal O}(\eps^{-1})$ term enforces the incompressibility condition 
$\bnabla\!\cdot\!\bv^{(0)}=0$. Proceeding to the ${\cal O}(1)$ terms, the buoyancy perturbation equation \eref{2.9c} yields
\begin{equation}
\frac{{\rm D}}{{\rm D}t}\theta^{(0)}=0,
\label{2.15}
\end{equation}
indicating that the leading-order buoyancy perturbation is materially conserved. Finally, applying the same asymptotic expansion to the PV 
equation \eref{2.11} and using the incompressibility condition $\bnabla\!\cdot\!\bv^{(0)}=0$ result in
\begin{equation}
\frac{{\rm D}}{{\rm D}t}q^{(0)}-\frac{1}{\nu}\Big[\phi^{(0)},\theta^{(0)}\Big]=-q^{(0)}\bnabla\!\cdot\!\bv^{(0)}=0.
\label{2.16}
\end{equation}

In summary, combining equations \eref{2.13}--\eref{2.16}, we arrive at the TQG equations:
\begin{subequations}
\begin{align}
&\bv^{(0)}=\nabla^\perp\psi^{(0)},\quad\Delta\psi^{(0)}=\omega^{(0)},\label{2.17a}\\
&\bigg(\frac{\partial}{\partial t}+\bv^{(0)}\cdot\nabla\bigg)\theta^{(0)}=0,\label{2.17b}\\
&\bigg(\frac{\partial}{\partial t}+\bv^{(0)}\cdot\nabla\bigg)q^{(0)}=\frac{1}{\nu}\Big[\psi^{(0)},\theta^{(0)}\Big]\label{2.17c}. 
\end{align}
\label{2.17}
\end{subequations}

We would like to point out that, in order to obtain the correct asymptotic limit when transitioning from the nondimensional TRSW equations 
\eref{2.9} to the limiting TQG equations \eref{2.17}, it is essential to assume that the initial conditions are well-prepared, that is,
\begin{equation}
\begin{aligned}
&\bv(\bx,0)=\bv^{(0)}(\bx,0)+{\cal O}(\eps),&&\omega(\bx,0)=\omega^{(0)}(\bx,0)+{\cal O}(\eps),\\
&\psi(\bx,0)=\psi^{(0)}(\bx,0)+{\cal O}(\eps),&&\theta(\bx,0)=\theta^{(0)}(\bx,0)+{\cal O}(\eps),
\end{aligned}
\label{2.18}
\end{equation}
where $\bv^{(0)}(\bx,0)$, $\omega^{(0)}(\bx,0)$, and $\psi^{(0)}(\bx,0)$ satisfy the limiting TQG equations \eref{2.17a}.  

We conclude this section by reformulating the nondimensional TRSW equations \eref{2.9} in conservative form, which needs to be used in 
high-Rossby-number regimes to accurately capture shock discontinuities. Assuming the TQG scaling \eref{2.8}, the depth of the fluid layer
and buoyancy can be expressed as
\begin{equation}
h=1+\frac{\eps}{\nu}\phi,\quad\Theta=1+\frac{2\eps}{\nu}\theta,
\label{2.19}
\end{equation}
and the conservative formulation of the nondimensional TRSW equations can then be written in terms of $h$, $hu$, $hv$, and $h\Theta$ as the 
following the hyperbolic system of balance laws:
\begin{equation}
\begin{aligned}
&h_t+(hu)_x+(hv)_y=0,\\
&(hu)_t+\left(hu^2+\frac{\nu}{2\eps^2}\,\Theta h^2\right)_x+(huv)_y=\frac{1+\eps\bar\beta y}{\eps}\,hv,\\
&(hv)_t+(huv)_x+\left(hv^2+\frac{\nu}{2\eps^2}\,\Theta h^2\right)_y=-\frac{1+\eps\bar\beta y}{\eps}\,hu,\\
&(h\Theta)_t+(hu\Theta)_x+(hv\Theta)_y=0.
\end{aligned}
\label{2.20}
\end{equation}
\begin{rmk}
We emphasize that the primitive \eref{2.9} and conservative \eref{2.20} formulations are equivalent for smooth solutions only. For 
discontinuous solutions that may develop in high-Rossby-number regimes, only the conservative formulation is valid.
\end{rmk}

\section{Semi-Implicit AP Scheme for the Primitive System}\label{sec3}
In this section, we develop an efficient AP scheme for the primitive system \eref{2.9}. To ensure the asymptotic convergence to the TQG 
equations \eref{2.17}, we would need to evolve not only $\bv$, $\phi$, and $\theta$, but also $q$. This idea is inspired by the approach 
proposed in \cite{Abgrall_2023}, which has recently been utilized to develop AP schemes for the magnetohydrodynamics \cite{Chen_2024}, 
compressible Euler \cite{Huang_2025}, and RSW \cite{Zhang_2026} equations.

We add the PV equation \eref{2.11} to \eref{2.9} and consider the following \emph{augmented primitive system}:
\begin{equation}
\left\{\begin{aligned}
&\bv_t+\bv\cdot\nabla\bv+\bigg(\frac{1}{\eps}+\bar\beta y\bigg)\bv^\perp=-\frac{1}{\eps}\nabla\big(\phi+\theta\big)-
\frac{1}{\nu}\big(\phi\nabla\theta+2\theta\nabla\phi\big),\\
&\phi_t+\bv\cdot\nabla\phi+\Big(\phi+\frac{\nu}{\eps}\Big)\bnabla\!\cdot\!\bv=0,\\
&\theta_t+\bv\cdot\nabla\theta=0,\\
&q_t+\bv\cdot\nabla q+q\bnabla\!\cdot\!\bv=\frac{1}{\nu}[\phi,\theta].
\end{aligned}\right.
\label{3.1}
\end{equation}
In the vector form, \eref{3.1} can be written as
\begin{equation}
\bV_t+B(\bV)\bV_x+C(\bV)\bV_y+\bQ(\bV)=\bm0,
\label{3.2}
\end{equation}
where $\bV:=(u,v,\phi,\theta,q)^\top$ is the vector of augmented primitive variables, 
\begin{equation}
B(\bV)=\begin{pmatrix}
u&0&\dfrac{\Theta}{\eps}&\dfrac{h}{\eps}&0\\0&u&0&0&0\\\nu\dfrac{h}{\eps}&0&u&0&0\\0&0&0&u&0\\q&0&0&0&u\\
\end{pmatrix}\!,~~
C(\bV)=\begin{pmatrix}
v&0&0&0&0\\0&v&\dfrac{\Theta}{\eps}&\dfrac{h}{\eps}&0\\0&\nu\dfrac{h}{\eps}&v&0&0\\0&0&0&v&0\\0&q&0&0&v
\end{pmatrix}\!,~~
\bQ(\bV)=\begin{pmatrix}
-\Big(\dfrac{1}{\eps}+\bar{\beta}y\Big)v\\[1.5ex]\Big(\dfrac{1}{\eps}+\bar{\beta}y\Big)u\\0\\0\\-\dfrac{1}{\nu}[\phi,\theta]
\end{pmatrix}.
\label{3.3}
\end{equation}
Note that we have used the identities \eref{2.19} to obtain the matrices $B$ and $C$ in \eref{3.3}.

The augmented primitive system \eref{3.2}--\eref{3.3} is hyperbolic and its eigenvalues in direction $\bn$ are
$\lambda_1=\bv\cdot\bn-c_s/\eps$, $\lambda_{2,3,4}=\bv\cdot\bn$, and $\lambda_5=\bv\cdot\bn+c_s/\eps$ with $c_s=\sqrt{\nu h\Theta}$. Since
the characteristic speeds $\lambda_{1,5}$ scale inversely with the Rossby number $\eps$, an explicit time discretization would impose a
severe time-step restriction. To address this issue, we first decompose the fluxes into two components associated with the slow (nonstiff)
and fast (stiff) dynamics, and then employ an SI Runge-Kutta method to effectively eliminate the stiff time-step limitation.

The remainder of this section is organized as follows. In \S\ref{sec31}, we begin with a hyperbolic splitting to identify the corresponding 
slow and fast dynamics. In \S\ref{sec32}, a first-order SI Runge-Kutta time discretization is applied to the split augmented primitive
system to ensure the AP property. In \S\ref{sec33}, we describe the PCCU spatial discretization for the nonstiff terms. Finally, in
\S\ref{sec34} and \S\ref{sec35}, we present fully discrete first- and second-order SI PCCU schemes, respectively.
		
\subsection{Hyperbolic Splitting}\label{sec31}
The system \eref{3.2}--\eref{3.3} can be split into the nonstiff and stiff parts as follows:
\begin{equation}
\bV_t+\widetilde{B}(\bV)\bV_x+\widetilde{C}(\bV)\bV_y+\widetilde{\bQ}(\bV)+\frac{1}{\eps}\Big(\widehat{B}(\bV)\bV_x
+\widehat{C}(\bV)\bV_y+\widehat{\bQ}(\bV)\Big)=\bm0,
\label{3.4}
\end{equation}
where 
\begin{equation}
\begin{aligned}
&\widetilde{B}(\bV)=\begin{pmatrix}
u&0&\dfrac{\Theta-b(t)}{\eps}&\dfrac{h-b(t)}{\eps}&0\\0&u&0&0&0\\\nu\dfrac{h-a(t)}{\eps}&0&u&0&0\\0&0&0&u&0\\q&0&0&0&u
\end{pmatrix},\\[1ex]
&\widetilde{C}(\bV)=\begin{pmatrix}
v&0&0&0&0\\0&v&\dfrac{\Theta-b(t)}{\eps}&\dfrac{h-b(t)}{\eps}&0\\0&\nu\dfrac{h-a(t)}{\eps}&v&0&0\\0&0&0&v&0\\0&q&0&0&v
\end{pmatrix},
\end{aligned}
\label{3.5}
\end{equation}
and
\begin{equation}
\widetilde{\bQ}(\bV)=
\left(-\Big(\frac{1-b(t)}{\eps}+\bar{\beta}y\Big)v,\Big(\frac{1-b(t)}{\eps}+\bar{\beta}y\Big)u,0,0,-\frac{1}{\nu}[\phi,\theta]\right)^\top,
\label{3.6}
\end{equation}
are the nonstiff components, while the stiff components are represented by
\begin{equation}
\widehat{B}(\bV)=\begin{pmatrix}
0&0&b(t)&b(t)&0\\0&0&0&0&0\\\nu a(t)&0&0&0&0\\0&0&0&0&0\\0&0&0&0&0
\end{pmatrix},\quad
\widehat{C}(\bV)=\begin{pmatrix}
0&0&0&0&0\\0&0&b(t)&b(t)&0\\0&\nu a(t)&0&0&0\\0&0&0&0&0\\0&0&0&0&0
\end{pmatrix},
\label{3.7}
\end{equation}
and
\begin{equation}
\widehat{\bQ}(\bV)=\left(-b(t)v,b(t)u,0,0,0\right)^\top.
\label{3.8}
\end{equation}
In order to ensure that the subsystem $\bV_t+\widetilde B(\bV)\bV_x+\widetilde C(\bV)\bV_y+\widetilde{\bQ}(\bV)=\bm0$ is both nonstiff and
hyperbolic, we need to choose the parameters $a(t)$ and $b(t)$ in an appropriate way. To this end, we first compute the eigenvalues of the
matrices $\widetilde B(\bV)$ and $\widetilde C(\bV)$, which are
\begin{equation}
\left\{u\pm\frac{1}{\eps}\sqrt{\nu\big(h-a(t)\big)\big(\Theta-b(t)\big)},u,u,u\right\}\quad\mbox{and}\quad
\left\{v\pm\frac{1}{\eps}\sqrt{\nu\big(h-a(t)\big)\big(\Theta-b(t)\big)},v,v,v\right\},
\label{3.9}
\end{equation}
respectively. A natural choice for the parameters $a(t)$ and $b(t)$ is then
\begin{equation}
a(t)=(1-\eps)\!\min_{(x,y)\in\Omega}\!h(x,y,t),\quad b(t)=(1-\eps)\!\min_{(x,y)\in\Omega}\!\Theta(x,y,t),
\label{3.10f}
\end{equation}
where $\Omega\in\mathbb R$ is a bounded domain. Indeed, using the asymptotic expansion \eref{2.12} for $\phi$ and $\theta$ and the scaling
relations for $h$ and $\Theta$ in \eref{2.19}, we obtain that $h-a(t)={\cal O}(\eps)$ and $\Theta-b(t)={\cal O}(\eps)$. Therefore, the
eigenvalues in \eref{3.9} remain real and uniformly bounded with respect to $\eps$, ensuring that the nonstiff subsystem can be treated 
explicitly without imposing any restrictive time-step constraints, while the stiff subsystem should be discretized in a (semi-)implicit
manner.  

\subsection{First-Order SI Time Discretization}\label{sec32}
We start with a first-order SI time discretization for \eref{3.4}--\eref{3.8} written in the component form:
\begin{subequations}
\begin{align}
&\frac{\bv^{n+1}-\bv^n}{\dt}+\bm{{\cal R}}^{\bv,n}+\frac{b^n}{\eps}\nabla\psi^{n+1}+\frac{b^n}{\eps}\big(\bv^{n+1}\big)^\perp=\bm0,
\label{3.10a}\\
&\frac{\phi^{n+1}-\phi^n}{\dt}+{\cal R}^{\phi,n}+\nu\,\frac{a^n}{\eps}\bnabla\!\cdot\!\bv^{n+1}=0,\label{3.10b}\\
&\frac{\theta^{n+1}-\theta^n}{\dt}+{\cal R}^{\theta,n}=0,\label{3.10c}\\
&\frac{q^{n+1}-q^n}{\dt}+{\cal R}^{q,n}=0,\label{3.10d}
\end{align}
\label{IMEX1}
\end{subequations}
where the superscripts $n$ and $n+1$ denote the data at time levels $t^n$ and $t^{n+1}=t^n+\dt$, respectively, $a^n:=a(t^n)$ and
$b^n:=b(t^n)$ are defined in \eref{3.10f}, and $\bm{{\cal R}}^n:=((\bm{{\cal R}}^{\bv,n})^\top,{\cal R}^{\phi,n},{\cal R}^{\theta,n},
{\cal R}^{q,n})^\top$ are the following nonstiff terms:
\begin{equation}
\begin{aligned}
\bm{{\cal R}}^{\bv,n}&:=\bv^n\cdot\nabla\bv^n+\frac{\Theta^n-b^n}{\eps}\nabla\phi^n+\frac{h^n-b^n}{\eps}\nabla\theta^n+
\frac{1+\eps\bar{\beta}y-b^n}{\eps}(\bv^n)^\perp,\\
{\cal R}^{\phi,n}&:=\bv^n\cdot\nabla\phi^n+\nu\,\frac{h^n-a^n}{\eps}\bnabla\!\cdot\!\bv^n,\\
{\cal R}^{\theta,n}&:=\bv^n\cdot\nabla\theta^n,\qquad{\cal R}^{q,n}:=\bv^n\cdot\nabla q^n+q^n\bnabla\!\cdot\!\bv^n-
\frac{1}{\nu}[\phi^n,\theta^n],
\end{aligned}
\label{3.11}
\end{equation}
which are discretized in \eref{IMEX1} explicitly. 

In the SI framework, the solution of the scheme \eref{IMEX1} is updated in the following three steps.

\smallskip
\noindent
$\bullet\,$ \textbf{Step 1 (Update $\theta$ and $q$).} We solve the nonstiff equations \eref{3.10c}--\eref{3.10d} to obtain 
$$
\theta^{n+1}=\theta^n-\dt{\cal R}^{\theta,n},\quad q^{n+1}=q^n-\dt{\cal R}^{q,n}.
$$

\smallskip
\noindent
$\bullet\,$ \textbf{Step 2 (Solve the linear elliptic equation for $\psi$ and update $\phi$).} We take the divergence of \eref{3.10a} to
obtain
\begin{equation*}
\bnabla\!\cdot\!\bv^{n+1}=\bnabla\!\cdot\!\bv^n-\dt\bnabla\!\cdot\!\bm{{\cal R}}^{\bv,n}-\frac{b^n\dt}{\eps}\Delta\psi^{n+1}
+\frac{b^n\dt}{\eps}\omega^{n+1},
\end{equation*}
use the updated values $\theta^{n+1}$ and $q^{n+1}$ to compute
$$
\omega^{n+1}=q^{n+1}-\bar{\beta}y+\frac{\phi^{n+1}}{\nu}=q^{n+1}-\bar{\beta}y+\frac{\psi^{n+1}-\theta^{n+1}}{\nu},
$$
and substitute it into the last equation, which can be rewritten as
\begin{equation}
\bnabla\!\cdot\!\bv^{n+1}=\bnabla\!\cdot\!\bv^n-\dt\bnabla\!\cdot\!\bm{{\cal R}}^{\bv,n}+\frac{b^n\dt}{\eps}\bigg(q^{n+1}-\bar{\beta}y-
\frac{\theta^{n+1}}{\nu}\bigg)+\frac{b^n\dt}{\eps\nu}\psi^{n+1}-\frac{b^n\dt}{\eps}\Delta\psi^{n+1}.
\label{3.12}
\end{equation}
We then add \eref{3.10b} and \eref{3.10c}, denote by ${\cal R}^{\psi,n}:={\cal R}^{\phi,n}+{\cal R}^{\theta,n}$, and obtain an evolution
equation for the streamfunction:
\begin{equation*}
\frac{\psi^{n+1}-\psi^n}{\dt}+{\cal R}^{\psi,n}+\nu\frac{a^n}{\eps}\bnabla\!\cdot\!\bv^{n+1}=0,
\end{equation*}
into which we substitute \eref{3.12} and arrive at the following linear elliptic equation for $\psi^{\,n+1}$:
\begin{equation}
\big(\eps^2+a^nb^n(\dt)^2\big)\psi^{n+1}-\nu a^nb^n(\dt)^2\Delta\psi^{n+1}={\cal S}^n,
\label{3.13}
\end{equation}
where the source term,
\begin{equation}
\begin{aligned}
{\cal S}^n=&-a^nb^n(\dt)^2\big(\nu q^{n+1}-\nu\bar{\beta}y-\theta^{n+1}\big)\\
&-\eps\nu a^n\dt\big(\bnabla\!\cdot\!\bv^n-\dt\bnabla\!\cdot\!\bm{{\cal R}}^{\bv,n}\big)+\eps^2\big(\psi^n-\dt{\cal R}^{\psi,n}\big),
\end{aligned}
\label{3.14}
\end{equation}
is computed using the quantities already available at the time level $t=t^n$. After solving \eref{3.13}--\eref{3.14} for $\psi^{n+1}$, the
depth perturbation is given by $\phi^{n+1}=\psi^{n+1}-\theta^{n+1}$.

\smallskip
\noindent
$\bullet\,$ \textbf{Step 3 (Update $\bv$).} We rewrite \eref{3.10a} as
\begin{equation}
\eps\bv^{n+1}+b^n\dt\big(\bv^{n+1}\big)^\perp=\eps\bv^n-\eps\dt\bm{{\cal R}}^{\bv,n}-b^n\dt\nabla\psi^{n+1},
\label{3.15}
\end{equation}
and use $\psi^{n+1}$ computed in Step 2 to obtain the linear algebraic system for $\bv^{n+1}$, which is then solved exactly to obtain
\begin{equation}
\begin{aligned}
\bv^{n+1}=\frac{1}{\eps^2+(b^n\dt)^2}\Big(\!&-\eps^2\dt\bm{{\cal R}}^{\bv,n}+\eps b^n(\dt)^2\big(\bm{{\cal R}}^{\bv,n}\big)^\perp+
\eps^2\bv^n\\
&-\eps b^n\dt\big((\bv^n)^\perp+\nabla\psi^{n+1}\big)+\big(b^n\dt\big)^2\nabla^\perp\psi^{n+1}\Big).
\end{aligned}
\label{3.16}
\end{equation}

\smallskip
We now prove that the AP property of the designed first-order SI scheme.
\begin{theorem}
Assume that the discrete initial data satisfy \eref{2.18}, namely, 
\begin{equation}
\bv^0=\bv^{(0),0}+{\cal O}(\eps),\quad\omega^0=\omega^{(0),0}+{\cal O}(\eps),\quad\psi^0=\psi^{(0),0}+{\cal O}(\eps),\quad
\theta^0=\theta^{(0),0}+{\cal O}(\eps),
\label{3.17}
\end{equation}
and, in addition, assume that $\bv^{(0),0}$, $\omega^{(0),0}$, and $\psi^{(0),0}$ satisfy the TQG equations \eref{2.17a}. Then, the SI
scheme, given by \eref{3.10c}, \eref{3.10d}, and \eref{3.13}--\eref{3.15}, is AP in the sense that its solution admits the expansions
\eref{3.17} at all time levels and the scheme provides a consistent discretization of the TQG equations \eref{2.17} as $\eps\to0$.
\end{theorem}
\begin{proof}
We proceed by mathematical induction. Given \eref{2.17a} and \eref{3.17} for the time level $t=t^0=0$, let us assume that at the time level
$t=t^n$, the numerical solution admits the same expansions 
\begin{equation}
\bv^n=\bv^{(0),n}+{\cal O}(\eps),\quad\omega^n=\omega^{(0),n}+{\cal O}(\eps),\quad\psi^n=\psi^{(0),n}+{\cal O}(\eps),\quad 
\theta^n=\theta^{(0),n}+{\cal O}(\eps),
\label{3.18}
\end{equation}
and the leading terms in \eref{3.18} satisfy \eref{2.17a}, namely,
\begin{equation}
\bv^{(0),n}=\nabla^\perp\psi^{(0),n},\quad\Delta\psi^{(0),n}=\omega^{(0),n}.
\label{3.19}
\end{equation}
From \eref{3.10c}, \eref{3.11}, and \eref{3.18}, we obtain
\begin{equation*}
\theta^{n+1}=\theta^n-\dt\bv^n\cdot\nabla\theta^n=\theta^{(0),n}-\dt\bv^{(0),n}\cdot\nabla\theta^{(0),n}+{\cal O}(\eps). 
\end{equation*}
Hence, at the next time level,
\begin{equation}
\theta^{n+1}=\theta^{(0),n+1}+{\cal O}(\eps)
\label{3.21}
\end{equation}
with 
\begin{equation}
\theta^{(0),n+1}=\theta^{(0),n}-\dt\bv^{(0),n}\cdot\nabla\theta^{(0),n}.
\label{3.22}
\end{equation}
Note that equation \eref{3.22} represents a consistent discrete approximation of \eref{2.17b}. 
		
Next, we consider \eref{3.10d} and \eref{3.11}, use $\psi^n=\phi^n+\theta^n$ and $[\psi^n,\theta^n]=[\phi^n,\theta^n]$, and arrive at
\begin{equation*}
\begin{aligned}
q^{n+1}&=q^n-\dt\Big(\bv^n\cdot\nabla q^n+q^n\bnabla\!\cdot\!\bv^n-\frac{1}{\nu}[\psi^n,\theta^n]\Big)\\
&=q^{(0),n}-\dt\bv^{(0),n}\cdot\nabla q^{(0),n}+\frac{\dt}{\nu}\big[\psi^{(0),n},\theta^{(0),n}\big]+{\cal O}(\eps),
\end{aligned}
\end{equation*}
where we have used the divergence-free property $\bnabla\!\cdot\!\bv^{(0),n}=0$, which automatically follows from the first equation in 
\eref{3.19}. Therefore,
\begin{equation}
q^{n+1}=q^{(0),n+1}+{\cal O}(\eps)
\label{3.23}
\end{equation}
with
\begin{equation}
q^{(0),n+1}=q^{(0),n}-\dt\bv^{(0),n}\cdot\nabla q^{(0),n}+\frac{\dt}{\nu}\big[\psi^{(0),n},\theta^{(0),n}\big],
\label{3.24}
\end{equation}
which provides a consistent discrete scheme for the PV equation \eref{2.17c}.  

Then, substituting the expansions \eref{3.21} and \eref{3.23} into equations \eref{3.13}--\eref{3.14}, yields
\begin{equation*}
{\cal S}^n=-a^nb^n(\dt)^2\Big(\nu q^{(0),n+1}-\nu\bar{\beta}y-\theta^{(0),n+1}\Big)+{\cal O}(\eps),
\end{equation*}
and hence,
\begin{equation}
\psi^{n+1}=\psi^{(0),n+1}+{\cal O}(\eps),
\label{3.25f}
\end{equation}
where the leading-order term $\psi^{(0),n+1}$ satisfies the following elliptic equation:
\begin{equation}
a^nb^n(\dt)^2\psi^{(0),n+1}-\nu a^nb^n(\dt)^2\Delta\psi^{(0),n+1}=
-a^nb^n(\dt)^2\left(\nu q^{(0),n+1}-\nu\bar{\beta}y-\theta^{(0),n+1}\right).
\label{3.25}
\end{equation}
Using the definitions of $q$ and $\psi$, together with the expansions \eref{3.21}, \eref{3.23}, and \eref{3.25f}, the vorticity $\omega$ at
the next time level, $t=t^{n+1}$, can be expressed as
\begin{equation*}
\omega^{n+1}=q^{n+1}-\bar{\beta}y+\frac{1}{\nu}\big(\psi^{n+1}-\theta^{n+1}\big)=q^{(0),n+1}-\bar{\beta}y+
\frac{1}{\nu}\big(\psi^{(0),n+1}-\theta^{(0),n+1}\big)+{\cal O}(\eps).
\end{equation*}
Hence, we have
\begin{equation}
\omega^{n+1}=\omega^{(0),n+1}+{\cal O}(\eps),
\label{3.27f}
\end{equation}
where the leading-order vorticity is given by
\begin{equation}
\omega^{(0),n+1}=q^{(0),n+1}-\bar{\beta}y+\frac{1}{\nu}\left(\psi^{(0),n+1}-\theta^{(0),n+1}\right).
\label{3.27}
\end{equation}
Substituting \eref{3.27} into \eref{3.25} yields the following Poisson equation for the leading-order streamfunction:
\begin{equation}
\Delta\psi^{(0),n+1}=\omega^{(0),n+1},
\label{3.29}
\end{equation}
which is the second equation in \eref{2.17a}.
		
Finally, combining \eref{3.15} with the expansion \eref{3.25f}, we obtain
\begin{equation}
\bv^{n+1}=\bv^{(0),n+1}+{\cal O}(\eps)
\label{3.30}
\end{equation}
with
\begin{equation}
\bv^{(0),n+1}=\nabla^\perp\psi^{(0),n+1},
\label{3.31f}
\end{equation}	
which is the first equation in \eref{2.17a}.
		
In summary, by combining equations \eref{3.21}, \eref{3.25f}, \eref{3.27f}, and \eref{3.30}, we conclude that the numerical solution at the 
next time level, $t=t^{n+1}$, can be expanded with respect to $\eps$ as 
\begin{equation*}
\bv^{n+1}=\bv^{(0),n+1}+{\cal O}(\eps),~~\omega^{n+1}=\omega^{(0),n+1}+{\cal O}(\eps),~~\psi^{n+1}=\psi^{(0),n+1}+{\cal O}(\eps),~~
\theta^{n+1}=\theta^{(0),n+1}+{\cal O}(\eps),
\end{equation*}
which, together with \eref{3.22}, \eref{3.24}, \eref{3.29}, and \eref{3.31f}, demonstrates that the scheme provides a consistent discrete
approximation of the TQG equations \eref{2.17} as $\eps\to0$, which completes the proof of this theorem.
\end{proof}

\subsection{Space Discretization}\label{sec33}
The temporal discretization introduced in \S\ref{sec32} guarantees the correct asymptotic limit as $\eps$ tends to zero in the semi-discrete
temporal setting. In this section, we describe a spatial discretization. We first introduce uniform finite-volume cells
$I_{j,k}=[\xL,\xR]\times[\yL,\yR]$ centered at $x_j:=\big(\xL+\xR\big)/2$ and $y_k:=\big(\yL+\yR\big)/2$, and define $\dx:=\xR-\xL$ and
$\dy:=\yR-\yL$. We assume that the cell averages $\xbar{\bV}_{j,k}^{\,n}\approx\frac{1}{\dx\dy}\iint_{I_{j,k}}\bV(x,y,t^n)\,\d x\d y$ are
available at the time level $t=t^n$, and use them to obtain a global piecewise linear reconstruction
\begin{equation*}
\widetilde{\bV}(x,y,t^n)=\xbar{\bV}_{j,k}^{\,n}+(\bV_x)_{j,k}^n(x-x_j)+(\bV_y)_{j,k}^n(y-y_k),\quad(x,y)\in I_{j,k}.
\end{equation*}
The latter allows one to compute one-sided point values of $\bm V$ at the midpoints of the cell interfaces:
\begin{equation}
\begin{aligned}
&\bV_{\jR,k}^{-,n}=\xbar{\bV}_{j,k}^{\,n}+\frac{\dx}{2}(\bV_x)_{j,k}^n,&&
\bV_{\jL,k}^{+,n}=\xbar{\bV}_{j,k}^{\,n}-\frac{\dx}{2}(\bV_x)_{j,k}^n,\\
&\bV_{j,\kR}^{-,n}=\xbar{\bV}_{j,k}^{\,n}+\frac{\dy}{2}(\bV_y)_{j,k}^n,&&
\bV_{j,\kL}^{+,n}=\xbar{\bV}_{j,k}^{\,n}-\frac{\dy}{2}(\bV_y)_{j,k}^n.
\end{aligned}
\label{3.31}
\end{equation}
In order to obtain second-order accuracy in space and prevent oscillations, the slopes in \eref{3.31} are to be computed using a nonlinear
limiter. In the numerical examples reported in \S\ref{sec5}, we have used the generalized minmod limiter (see, e.g., 
\cite{Lie_2003,Nessyahu_1990,Sweby_1984}):
\begin{equation}
\begin{aligned}
(\bV_x)_{j,k}^n&={\rm minmod}\left(\mu\,\frac{\xbar{\bV}_{j,k}^{\,n}-\xbar{\bV}_{j-1,k}^{\,n}}{\dx},\,
\frac{\xbar{\bV}_{j+1,k}^{\,n}-\xbar{\bV}_{j-1,k}^{\,n}}{2\dx},\,\mu\,\frac{\xbar{\bV}_{j+1,k}^{\,n}-\xbar{\bV}_{j,k}^{\,n}}{\dx}\right),\\
(\bV_y)_{j,k}^n&={\rm minmod}\left(\mu\,\frac{\xbar{\bV}_{j,k}^{\,n}-\xbar{\bV}_{j,k-1}^{\,n}}{\dy},\,
\frac{\xbar{\bV}_{j,k+1}^{\,n}-\xbar{\bV}_{j,k-1}^{\,n}}{2\dy},\,\mu\,\frac{\xbar{\bV}_{j,k+1}^{\,n}-\xbar{\bV}_{j,k}^{\,n}}{\dy}\right),
\end{aligned}
\label{3.32}
\end{equation}
where the minmod function, defined by
\begin{equation*}
{\rm minmod}(c_1,c_2,\ldots)=\left\{\begin{aligned}
&\min(c_1,c_2,\ldots)&&\mbox{if}~c_i>0,~\forall i,\\
&\max(c_1,c_2,\ldots)&&\mbox{if}~c_i<0,~\forall i,\\
&\,0&&\mbox{otherwise},
\end{aligned}\right.
\end{equation*}
is applied in a componentwise manner. The parameter $\mu\in[1,2]$ in \eref{3.32} is to be chosen to adjust the amount of numerical
dissipation present in the numerical scheme, with larger values of $\mu$ leading to sharper but, in general, more oscillatory
reconstructions. 
	
\subsubsection{PCCU Discretization of the Nonstiff Terms}
The nonstiff part of \eref{IMEX1} is given by
$\bm{{\cal R}}^n:=((\bm{{\cal R}}^{\bv,n})^\top,{\cal R}^{\phi,n},{\cal R}^{\theta,n},{\cal R}^{q,n})^\top$ in \eref{3.11}, and is also
presented in the vector form as $\widetilde B(\bV^n)\bV_x^n+\widetilde C(\bV^n)\bV_y^n+\widetilde{\bQ}^n(\bV^n)$; see
\eref{3.4}--\eref{3.6}. As one can see, most of the terms there contain nonconservative products and thus require proper treatment. We
discretize these terms using the second-order PCCU approach introduced in \cite{Castro_2019} (for the 2-D version of the PCCU scheme, we
refer the reader to \cite{Chu_2022}):
\begin{equation}
\begin{aligned}
\bm{{\cal R}}_{j,k}^n:=\,&\frac{1}{\dx}\Bigg(\widetilde{\bm{{\cal D}}}_{\jR,k}^n-\widetilde{\bm{{\cal D}}}_{\jL,k}^n+ 
\widetilde{\bm B}_{j,k}^{\,n}+\frac{s^{+,n}_{\jL,k}\widetilde{\bm B}_{\bm\Psi,\jL,k}^{\,n}}{s^{+,n}_{\jL,k}-s^{-,n}_{\jL,k}}
-\frac{s^{-,n}_{\jR,k}\,\widetilde{\bm B}_{\bm\Psi,\jR,k}^{\,n}}{s^{+,n}_{\jR,k}-s^{-,n}_{\jR,k}}\Bigg)\\ 
+\,&\frac{1}{\dy}\Bigg(\widetilde{\bm{{\cal D}}}_{j,\kR}^n-\widetilde{\bm{{\cal D}}}_{j,\kL}^n+\widetilde{\bm C}_{j,k}^{\,n}
+\frac{s^{+,n}_{j,\kL}\,\widetilde{\bm C}_{\bm\Psi,j,\kL}^{\,n}}{s^{+,n}_{j,\kL}-s^{-,n}_{j,\kL}}-\frac{s^{-,n}_{j,\kR}
\widetilde{\bm C}_{\bm\Psi,j,\kR}^{\,n}}{s^{+,n}_{j,\kR}-s^{-,n}_{j,\kR}}\Bigg)+\,\xbar{\widetilde{\bQ}}_{j,k}^{\,n}.
\end{aligned}
\label{3.34}
\end{equation}	
Here,
\begin{equation*}
\begin{aligned}
\widetilde{\bm{{\cal D}}}_{\jR,k}^n&=\frac{s^{+,n}_{\jR,k}s^{-,n}_{\jR,k}}{s^{+,n}_{\jR,k}-s^{-,n}_{\jR,k}}
\Big(\bV^{+,n}_{\jR,k}-\bV^{-,n}_{\jR,k}-\delta\bV_{\jR,k}^n\Big),\\
\widetilde{\bm{{\cal D}}}_{j,\kR}^n&=\frac{s^{+,n}_{j,\kR}s^{-,n}_{j,\kR}}{s^{+,n}_{j,\kR}-s^{-,n}_{j,\kR}}
\Big(\bV^{+,n}_{j,\kR}-\bV^{-,n}_{j,\kR}-\delta\bV_{j,\kR}^n\Big),
\end{aligned}
\end{equation*}
are the numerical diffusion terms with
\begin{equation*}
\begin{aligned}
&\delta\bV_{\jR,k}^n:={\rm minmod}\Big(\bV^{+,n}_{\jR,k}-\bV^{*,n}_{\jR,k},\,\bV^{*,n}_{\jR,k}-\bV^{-,n}_{\jR,k}\Big),\\
&\delta\bV_{j,\kR}^n:={\rm minmod}\Big(\bV^{+,n}_{j,\kR}-\bV^{*,n}_{j,\kR},\,\bV^{*,n}_{j,\kR}-\bV^{-,n}_{j,\kR}\Big),
\end{aligned}
\end{equation*}
where
\begin{equation*}		
\bV^{*,n}_{\jR,k}=\frac{s^{+,n}_{\jR,k}\bV^{+,n}_{\jR,k}-s^{-,n}_{\jR,k}\bV^{-,n}_{\jR,k}}{s^{+,n}_{\jR,k}-s^{-,n}_{\jR,k}},\quad
\bV^{*,n}_{j,\kR}=\frac{s^{+,n}_{j,\kR}\bV^{+,n}_{j,\kR}-s^{-,n}_{j,\kR}\bV^{-,n}_{j,\kR}}{s^{+,n}_{j,\kR}-s^{-,n}_{j,\kR}}.
\end{equation*}
In addition, $s_{\jR,k}^{\pm,n}$ and $s_{j,\kR}^{\pm,n}$ denote the one-sided local propagation speeds in the $x$- and $y$-directions, 
respectively. They are estimated using the smallest and largest eigenvalues of the matrices $\widetilde B(\bV^n)$ and $\widetilde C(\bV^n)$,
which are given in \eref{3.9}:
\begin{equation}
\begin{aligned}
&s^{+,n}_{\jR,k}=\max\left\{u^{-,n}_{\jR,k}+\Lambda^{-,n}_{\jR,k},\,u^{+,n}_{\jR,k}+\Lambda^{+,n}_{\jR,k},\,0\right\},\\
&s^{-,n}_{\jR,k}=\min\left\{u^{-,n}_{\jR,k}-\Lambda^{-,n}_{\jR,k},\,u^{+,n}_{\jR,k}-\Lambda^{+,n}_{\jR,k},\,0\right\},\\
&s^{+,n}_{j,\kR}=\max\left\{v^{-,n}_{j,\kR}+\Lambda^{-,n}_{j,\kR},\,v^{+,n}_{j,\kR}+\Lambda^{+,n}_{j,\kR},\,0\right\},\\
&s^{-,n}_{j,\kR}=\min\left\{v^{-,n}_{j,\kR}-\Lambda^{-,n}_{j,\kR},\,v^{+,n}_{j,\kR}-\Lambda^{+,n}_{j,\kR},\,0\right\},
\end{aligned}
\label{3.35f}
\end{equation}
where
\begin{equation}
\Lambda^{\pm,n}_{\jR,k}=\frac{1}{\eps}\sqrt{\nu\big(h^{\pm,n}_{\jR,k}-a^n\big)\big(\Theta^{\pm,n}_{\jR,k}-b^n\big)},\quad
\Lambda^{\pm,n}_{j,\kR}=\frac{1}{\eps}\sqrt{\nu\big(h^{\pm,n}_{j,\kR}-a^n\big)\big(\Theta^{\pm,n}_{j,\kR}-b^n\big)},
\label{3.35}
\end{equation}
and the parameters $a^n$ and $b^n$ are computed as follows (see \eref{3.10f}):
\begin{equation}
\begin{aligned}
&a^n=(1-\eps)\min\limits_{j,k}\Big\{\min\Big(h^{+,n}_{\jR,k},\,h^{-,n}_{\jR,k},\,h^{+,n}_{j,\kR},\,h^{-,n}_{j,\kR}\Big)\Big\},\\
&b^n=(1-\eps)\min\limits_{j,k}\Big\{\min\Big(\Theta^{+,n}_{\jR,k},\,\Theta^{-,n}_{\jR,k},\,\Theta^{+,n}_{j,\kR},\,\Theta^{-,n}_{j,\kR}\Big)
\Big\}.
\end{aligned}
\label{3.37}
\end{equation}
Next, the cell averages of $\widetilde{\bQ}$ are approximated using the midpoint rule, which results in 
\begin{equation*}
\xbar{\widetilde{\bQ}}_{j,k}^{\,n}=\Big(-\Big(\frac{1-b^n}{\eps}+\bar{\beta}y_k\Big)\,\xbar v_{j,k}^{\,n},\,
\Big(\frac{1-b^n}{\eps}+\bar{\beta}y_k\Big)\,\xbar u_{j,k}^{\,n},\,0,\,0,\,
-\frac{1}{\nu}\big[\,\xbar{\phi}_{j,k}^{\,n},\,\xbar{\theta}_{j,k}^{\,n}\big]\Big)^\top.
\end{equation*}
The 2-D Jacobian operator $[\phi,\theta]$ is approximated using the second-order central differences as follows:
\begin{equation*}
\big[\,\xbar{\phi}_{j,k}^{\,n},\,\xbar{\theta}_{j,k}^{\,n}\big]=\frac{1}{4\dx\dy}
\Big(\big(\,\xbar{\phi}_{j+1,k}^{\,n}-\,\xbar{\phi}_{j-1,k}^{\,n}\big)\big(\,\xbar{\theta}_{j,k+1}^{\,n}-\,\bar{\theta}_{j,k-1}^{\,n}\big)+
\big(\,\xbar{\phi}_{j,k+1}^{\,n}-\,\bar{\phi}_{j,k-1}^{\,n}\big)\big(\,\xbar{\theta}_{j+1,k}^{\,n}-\,\xbar{\theta}_{j-1,k}^{\,n}\big)\Big).
\end{equation*}
	
Note that in order to numerically solve the linear elliptic equation \eref{3.13}--\eref{3.14} in Step 2 of the time update, the terms 
$\bnabla\!\cdot\!\bv^n$ and $\bnabla\!\cdot\!\bm{{\cal R}}^{\bv,n}$ are to be discretized. First, $\bnabla\!\cdot\!\bv^n$ is approximated
using the second-order central differences:
\begin{equation}
\bnabla\!\cdot\xbar{\bv}_{j,k}^{\,n}=\frac{\xbar u_{j+1,k}^{\,n}-\,\xbar u_{j-1,k}^{\,n}}{2\dx}+
\frac{\xbar v_{j,k+1}^{\,n}-\,\xbar v_{j,k-1}^{\,n}}{2\dy}.
\label{3.38}
\end{equation}
Second, the term $\bnabla\!\cdot\!\bm{{\cal R}}^{\bv,n}=\bnabla\!\cdot\!{\cal R}^{1,n}+\bnabla\!\cdot\!{\cal R}^{2,n}+
\bnabla\!\cdot\!{\cal R}^{3,n}$ is rewritten as a sum of the three terms with
\begin{equation}
\begin{aligned}
&\bnabla\!\cdot\!{\cal R}^{1,n}:=\bar{\beta}u^n-\frac{1+\eps\bar{\beta}y-b^n}{\eps}\omega^n=\bar{\beta}u^n-
\frac{1+\eps\bar{\beta}y-b^n}{\eps}\Big(q^n-\bar{\beta}y+\frac{\phi^n}{\nu}\Big),\\	
&\bnabla\!\cdot\!{\cal R}^{2,n}:=\bnabla\!\cdot\!\big(\bv^n\!\cdot\!\nabla\bv^n\big)=\hf\Big((u^n_{xx})^2+(v^n_{yy})^2\Big)-
[u^n,v^n]+(u^nv^n)_{xy},\\
&\bnabla\!\cdot\!{\cal R}^{3,n}:=\bnabla\!\cdot\!\bigg(\frac{\Theta^n-b^n}{\eps}\nabla\phi^n\bigg)+
\bnabla\!\cdot\!\bigg(\frac{h^n-b^n}{\eps}\nabla\theta^n\bigg).
\end{aligned}
\label{3.39}
\end{equation}
The first term in \eref{3.39} can be approximated by
\begin{equation}
\bnabla\!\cdot\!{\cal R}^{1,n}_{j,k}=\bar{\beta}\,\xbar u_{j,k}^{\,n}-\frac{1+\eps\bar{\beta}y_k-b^n}{\eps}\bigg(\,\xbar q_{j,k}^{\,n}-
\bar{\beta}y_k+\frac{\xbar{\phi}_{j,k}^{\,n}}{\nu}\bigg).
\label{3.40}
\end{equation}
The second term in \eref{3.39} is approximated using the second-order central differences, which give
\begin{equation}
\begin{aligned}
\bnabla\!\cdot\!{\cal R}_{j,k}^{2,n}&=\frac{\big(\,\xbar u^{\,n}_{j-1,k}\big)^2-2\big(\,\xbar u^{\,n}_{j,k}\big)^2+
\big(\,\xbar u^{\,n}_{j+1,k}\big)^2}{2(\dx)^2}+\frac{\big(\,\xbar v^{\,n}_{j,k-1}\big)^2-2\big(\,\xbar v^{\,n}_{j,k}\big)^2+
\big(\,\xbar v^{\,n}_{j,k+1}\big)^2}{2(\dy)^2}\\
&\hspace*{-0.9cm}-\frac{1}{4\dx\dy}\Big(\big(\,\xbar u_{j+1,k}^{\,n}-\,\xbar u_{j-1,k}^{\,n}\big)
\big(\,\xbar v_{j,k+1}^{\,n}-\,\xbar v_{j,k-1}^{\,n}\big)+
\big(\,\xbar u_{j,k+1}^{\,n}-\,\xbar u_{j,k-1}^{\,n}\big)\big(\,\xbar v_{j+1,k}^{\,n}-\,\xbar v_{j-1,k}^{\,n}\big)\Big)\\
&\hspace*{-0.9cm}+\frac{1}{4\dx\dy}
\left(\,\xbar u_{j+1,k+1}^{\,n}\,\xbar v_{j+1,k+1}^{\,n}-\,\xbar u_{j-1,k+1}^{\,n}\,\xbar v_{j-1,k+1}^{\,n}
-\,\xbar u_{j+1,k-1}^{\,n}\,\xbar v_{j+1,k-1}^{\,n}+\,\xbar u_{j-1,k-1}^{\,n}\,\xbar v_{j-1,k-1}^{\,n}\right).
\end{aligned}
\label{3.41}
\end{equation}
We then discretize the third term in \eref{3.39} using a compact central difference to obtain
\begin{equation}
\begin{aligned}
\bnabla\!\cdot\!{\cal R}_{j,k}^{3,n}&=\frac{1}{2\eps(\dx)^2}\Big(\big(\,\xbar\Theta_{j,k}^{\,n}+\,\xbar\Theta_{j+1,k}^{\,n}\big)
\big(\,\xbar{\phi}_{j+1,k}^{\,n}-\,\xbar{\phi}_{j,k}^{\,n}\big)-\big(\,\xbar\Theta_{j-1,k}^{\,n}+\,\xbar\Theta_{j,k}^{\,n}\big)
\big(\,\xbar{\phi}_{j,k}^{\,n}-\,\xbar{\phi}_{j-1,k}^{\,n}\big)\Big)\\
&+\frac{1}{2\eps(\dy)^2}\Big(\big(\,\xbar\Theta_{j,k}^{\,n}+\,\xbar\Theta_{j,k+1}^{\,n}\big)\big(\,\xbar{\phi}_{j,k+1}^{\,n}
-\,\xbar{\phi}_{j,k}^{\,n}\big)-\big(\,\xbar\Theta_{j,k-1}^{\,n}+\,\xbar\Theta_{j,k}^{\,n}\big)\big(\,\xbar{\phi}_{j,k}^{\,n}
-\,\xbar{\phi}_{j,k-1}^{\,n}\big)\Big)\\
&-\frac{b^n}{\eps}\bigg(\frac{\,\xbar{\phi}^{\,n}_{j-1,k}-2\,\xbar{\phi}^{\,n}_{j,k}+\,\xbar{\phi}^{\,n}_{j+1,k}}{(\dx)^2}+
\frac{\,\xbar{\phi}^{\,n}_{j,k-1}-2\,\xbar{\phi}^{\,n}_{j,k}+\,\xbar{\phi}^{\,n}_{j,k+1}}{(\dy)^2}\bigg)\\
&+\frac{1}{2\eps(\dx)^2}\Big(\big(\,\bar h_{j,k}^{\,n}+\,\xbar h_{j+1,k}^{\,n}\big)\big(\,\xbar\theta_{j+1,k}^{\,n}-
\,\xbar\theta_{j,k}^{\,n}\big)-\big(\,\xbar h_{j-1,k}^{\,n}+\,\xbar h_{j,k}^{\,n}\big)
\big(\,\xbar\theta_{j,k}^{\,n}-\,\xbar\theta_{j-1,k}^{\,n}\big)\Big)\\
&+\frac{1}{2\eps(\dy)^2}\Big(\big(\,\xbar h_{j,k}^{\,n}+\,\xbar h_{j,k+1}^{\,n}\big)\big(\,\xbar\theta_{j,k+1}^{\,n}-
\,\xbar\theta_{j,k}^{\,n}\big)-\big(\,\xbar h_{j,k-1}^{\,n}+\,\xbar h_{j,k}^{\,n}\big)\big(\,\xbar\theta_{j,k}^{\,n}-
\,\xbar\theta_{j,k-1}^{\,n}\big)\Big)\\
&-\frac{b^n}{\eps}\bigg(\frac{\,\xbar\theta^{\,n}_{j-1,k}-2\,\xbar\theta^{\,n}_{j,k}+\,\xbar\theta^{\,n}_{j+1,k}}{(\dx)^2}+
\frac{\,\xbar\theta^{\,n}_{j,k-1}-2\,\xbar\theta^{\,n}_{j,k}+\,\xbar\theta^{\,n}_{j,k+1}}{(\dy)^2}\bigg).	
\end{aligned}
\label{3.42}
\end{equation}
Finally, the terms $\widetilde{\bm B}_{j,k}^{\,n}$, $\widetilde{\bm B}_{\bm\Psi,j\pm\hf,k}^{\,n}$, $\widetilde{\bm C}_{j,k}^{\,n}$, and 
$\widetilde{\bm C}_{\bm\Psi,j,k\pm\hf}^{\,n}$ in \eref{3.34} reflect the contributions of the nonconservative products and they are 
computed using the linear paths to obtain (see \cite{Castro_2019, Chu_2022} for details)
\begin{equation*}
\begin{aligned}
\widetilde{\bm B}_{j,k}^{\,n}&=\hf\Big(\widetilde B\big(\bV^{-,n}_{\jR,k}\big)+\widetilde B\big(\bV^{+,n}_{\jL,k}\big)\Big)
\big(\bV^{-,n}_{\jR,k}-\bV^{+,n}_{\jL,k}\big),\\
\widetilde{\bm C}_{j,k}^{\,n}&=\hf\Big(\widetilde C\big(\bV^{-,n}_{j,\kR}\big)+\widetilde C\big(\bV^{+,n}_{j,\kL}\big)\Big)
\big(\bV^{-,n}_{j,\kR}-\bV^{+,n}_{j,\kL}\big),\\
\widetilde{\bm B}_{\bm\Psi,j\pm\hf,k}^{\,n}&=\hf\Big(\widetilde B\big(\bV^{-,n}_{j\pm\hf,k}\big)
+\widetilde B\big(\bV^{+,n}_{j\pm\hf,k}\big)\Big)\big(\bV^{+,n}_{j\pm\hf,k}-\bV^{-,n}_{j\pm\hf,k}\big),\\
\widetilde{\bm C}_{\bm\Psi,j,k\pm\hf}^{\,n}&=\hf\Big(\widetilde C\big(\bV^{-,n}_{j,k\pm\hf}\big)
+\widetilde C\big(\bV^{+,n}_{j,k\pm\hf}\big)\Big)\big(\bV^{+,n}_{j,k\pm\hf}-\bV^{-,n}_{j,k\pm\hf}\big),
\end{aligned}
\end{equation*}
where the matrices $\widetilde B$ and $\widetilde C$ are defined in \eref{3.5}.

\subsubsection{Central Differencing for the Stiff Terms}
The stiff terms in \eref{3.13} and \eref{3.16}, $\Delta\psi^{n+1}$ and $\nabla\psi^{n+1}$, are discretized using the second-order central 
differences:
\begin{subequations}
\begin{align}
&\Delta\xbar{\psi}^{\,n+1}_{j,k}=\frac{\,\xbar{\psi}^{\,n+1}_{j-1,k}-2\,\xbar{\psi}^{\,n+1}_{j,k}+\,\xbar{\psi}^{\,n+1}_{j+1,k}}{(\dx)^2}
+\frac{\xbar{\psi}^{\,n+1}_{j,k-1}-2\,\xbar{\psi}^{\,n+1}_{j,k}+\,\xbar{\psi}^{\,n+1}_{j,k+1}}{(\dy)^2}\label{3.43a},\\
&\nabla\xbar{\psi}^{\,n+1}_{j,k}=\bigg(\frac{\,\xbar{\psi}^{\,n+1}_{j+1,k}-\xbar{\psi}^{\,n+1}_{j-1,k}}{2\dx},\,
\frac{\xbar{\psi}^{\,n+1}_{j,k+1}-\xbar{\psi}^{\,n+1}_{j,k-1}}{2\dy}\bigg)^\top\label{3.43b}.
\end{align}
\end{subequations}
\begin{rmk}
We emphasize the importance of using central differences to approximate $\nabla\psi$, as this helps ensure that no numerical diffusion terms
of size ${\cal O}(\eps^{-1})$ are introduced.
\end{rmk}
	
\subsection{First-Order Semi-Implicit Fully Discrete AP Scheme}\label{sec34}
Summarizing the temporal (\S\ref{3.2}) and spatial (\S\ref{3.3}) discretizations, the first-order SI fully discrete AP scheme reads as
\begin{equation}
\begin{aligned}
&\frac{\xbar{\bv}^{\,n+1}_{j,k}-\,\xbar{\bv}^{\,n}_{j,k}}{\dt}+\bm{{\cal R}}^{\bv,n}_{j,k}+\frac{b^n}{\eps}\nabla\xbar{\psi}^{\,n+1}_{j,k} 
+\frac{b^n}{\eps}\big(\,\xbar{\bv}^{\,n+1}_{j,k}\big)^\perp=\bm0,\\
&\frac{\xbar{\phi}^{\,n+1}_{j,k}-\,\xbar{\phi}^{\,n}_{j,k}}{\dt}+{\cal R}^{\phi,n}_{j,k}+
\nu\,\frac{a^n}{\eps}\bnabla\!\cdot\xbar{\bv}^{\,n+1}_{j,k}=0,\\
&\frac{\xbar{\theta}^{\,n+1}_{j,k}-\,\xbar{\theta}^{\,n}_{j,k}}{\dt}+{\cal R}^{\theta,n}_{j,k}=0,\\
&\frac{\xbar q^{\,n+1}_{j,k}-\,\xbar q^{\,n}_{j,k}}{\dt}+{\cal R}^{q,n}_{j,k}=0.
\end{aligned}
\label{3.44}
\end{equation}
The update of this fully discrete system follows directly the three-step procedure introduced in \S\ref{3.2}. 

\subsection{Second-Order Semi-Implicit Fully Discrete AP Scheme}\label{sec35}
Note that the scheme \eref{3.44} is second-order accurate in space but only first-order accurate in time. To achieve second-order temporal
accuracy, we extend this scheme by employing a two-stage, globally stiffly accurate IMEX Runge-Kutta method, namely ARS(2,2,2), introduced
in \cite{Ascher_1997}, and present a second-order SI fully discrete AP scheme. The ARS(2,2,2) method is characterized by the following
double Butcher tableau:
\begin{equation}
\begin{split}
&\qquad~\textbf{Explicit:}\\&
\begin{array}{c|ccc}
0&0&0&0\\
\gamma&\gamma&0&0\\
1&1-\frac{1}{2\gamma}&\frac{1}{2\gamma}&0\\
\hline
&1-\frac{1}{2\gamma}&\frac{1}{2\gamma}&0
\end{array}
\end{split}
\begin{split}
\qquad\qquad\qquad&\qquad\textbf{Implicit:}\\&
\begin{array}{c|ccc} 
0&0&0&0\\
\gamma&0&\gamma&0\\ 
1&0&1-\gamma&\gamma\\
\hline
&0&1-\gamma&\gamma
\end{array}	
\end{split}
\label{3.45}
\end{equation}
with $\gamma=1-1/\sqrt{2}$. In what follows, the superscript $(\cdot)^*$ denotes quantities evaluated at the first stage of the ARS(2,2,2)
method.
	
According to the Butcher tableau \eref{3.45}, the first stage of the second-order SI fully discrete scheme is
\allowdisplaybreaks
\begin{subequations}
\begin{align}
&\frac{\xbar{\bv}^{\,*}_{j,k}-\,\xbar{\bv}^{\,n}_{j,k}}{\dt}+\gamma\bm{{\cal R}}^{\bv,n}_{j,k}
+\gamma\frac{b^n}{\eps}\nabla\xbar{\psi}^{\,*}_{j,k}+\gamma\frac{b^n}{\eps}\big(\bar{\bv}^{\,*}_{j,k}\big)^\perp=\bm0,\label{3.46a}\\
&\frac{\xbar{\phi}^{\,*}_{j,k}-\,\xbar{\phi}^{\,n}_{j,k}}{\dt}+\gamma{\cal R}^{\phi,n}_{j,k}+
\gamma\nu\frac{a^n}{\eps}\bnabla\!\cdot\xbar{\bv}^{\,*}_{j,k}=0,\label{3.46b}\\
&\frac{\xbar{\theta}^{\,*}_{j,k}-\,\xbar{\theta}^{\,n}_{j,k}}{\dt}+\gamma{\cal R}^{\theta,n}_{j,k}=0,\label{3.46c}\\
&\frac{\xbar q^{\,*}_{j,k}-\,\xbar q^{\,n}_{j,k}}{\dt}+\gamma{\cal R}^{q,n}_{j,k}=0.\label{3.46d}
\end{align}
\label{3.46}
\end{subequations}
As in the update procedure for the first-order SI scheme described in \S\ref{sec32}, we proceed in the following stages.

\smallskip
\noindent
$\bullet\,$ \textbf{Step 1 (Compute $\xbar\theta^{\,*}_{j,k}$ and $\,\xbar q^{\,*}_{j,k}$ explicitly).} We solve equations
\eref{3.46c}--\eref{3.46d} to obtain
\begin{equation*}
\xbar\theta^{\,*}_{j,k}=\,\xbar\theta^{\,n}_{j,k}-\gamma\dt{\cal R}^{\theta,n}_{j,k},\qquad
\xbar q^{\,*}_{j,k}=\,\xbar q^{\,n}_{j,k}-\gamma\dt{\cal R}^{q,n}_{j,k}.
\end{equation*}

\smallskip
\noindent
$\bullet\,$ \textbf{Step 2 (Solve the linear elliptic equation for $\,\xbar{\psi}^{\,*}_{j,k}$ and compute $\,\xbar{\phi}^{\,*}_{j,k}$).} We
discretize the linear elliptic equation \eref{3.13}--\eref{3.14} to obtain the following linear system of algebraic equations for
$\,\xbar{\psi}^{\,*}_{j,k}$:
\begin{equation}
\begin{aligned}
\big(\eps^2+a^nb^n(\gamma\dt)^2\big)\,\xbar{\psi}^{\,*}_{j,k}=&-a^nb^n(\gamma\dt)^2\big(\nu\,\xbar q^{\,*}_{j,k}-\nu\bar{\beta}y_k-
\,\xbar\theta^{\,*}_{j,k}\big)+\nu a^nb^n(\gamma\dt)^2\Delta\xbar{\psi}^{\,*}_{j,k}\\
&-\eps\nu a^n\gamma\dt\big(\bnabla\!\cdot\xbar{\bv}^{\,n}_{j,k}-\gamma\dt\bnabla\!\cdot\!\bm{{\cal R}}^{\bv,n}_{j,k}\big)
+\eps^2\big(\,\xbar\psi^{\,n}_{j,k}-\gamma\dt{\cal R}^{\psi,n}_{j,k}\big),
\end{aligned}
\label{3.47}
\end{equation}
where the discrete Laplacian $\Delta\xbar{\psi}^{\,*}_{j,k}$ is defined as in \eref{3.43a}. After solving \eref{3.47}, we obtain
$\xbar{\phi}^{\,*}_{j,k}=\,\xbar\psi^{\,*}_{j,k}-\,\xbar\theta^{\,*}_{j,k}$.

\smallskip
\noindent
$\bullet\,$ \textbf{Step 3 (Compute $\,\xbar{\bv}^{\,*}_{j,k}$).} Once $\,\xbar\psi^{\,*}_{j,k}$ is available, we compute	
\begin{equation*}
\begin{aligned}
\xbar{\bv}^{\,*}_{j,k}=\frac{1}{\eps^2+(b^n\gamma\dt)^2}\Big(\!&-\eps^2\gamma\dt\bm{{\cal R}}^{\bv,n}_{j,k}
+\eps b^n(\gamma\dt)^2\big(\bm{{\cal R}}^{\bv,n}_{j,k}\big)^\perp+\eps^2\,\xbar{\bv}^{\,n}_{j,k}\\
&-\eps b^n\gamma\dt\big((\,\xbar{\bv}^{\,n}_{j,k})^\perp+\nabla\xbar\psi^{\,*}_{j,k}\big)+
(b^n\gamma\dt)^2\nabla^\perp\xbar\psi^{\,*}_{j,k}\Big),
\end{aligned}
\end{equation*}
where the discrete gradient operator $\nabla\xbar{\psi}^{\,*}_{j,k}$ is defined as in \eref{3.43b}.
	
Equipped with $\,\xbar{\bV}^{\,*}_{j,k}=\big((\,\xbar{\bv}^{\,*}_{j,k})^\top,\,\xbar{\phi}^{\,*}_{j,k},\,\xbar\theta^{\,*}_{j,k},
\,\xbar q^{\,*}_{j,k}\big)^\top$, we evaluate $\bm{{\cal R}}_{j,k}^*$ in the same way as in \eref{3.34}, with the time-level superscript $n$
replaced by $*$. The parameters $a^*$ and $b^*$ are computed analogously to \eref{3.37}, again by replacing the time-level superscript $n$
with $*$. We then proceed to the second stage of \eref{3.45}, which can be written as
\begin{subequations}
\begin{align}
\frac{\xbar{\bv}^{\,n+1}_{j,k}-\,\xbar{\bv}^{\,n}_{j,k}}{\dt}+\left(1-\frac{1}{2\gamma}\right)\bm{{\cal R}}^{\bv,n}_{j,k}
+\frac{1}{2\gamma}\bm{{\cal R}}^{\bv,*}_{j,k}&+(1-\gamma)\frac{b^n}{\eps}\nabla\xbar{\psi}^{\,*}_{j,k}+
(1-\gamma)\frac{b^n}{\eps}(\xbar{\bv}^{\,*}_{j,k})^\perp\notag\\
&+\gamma\frac{b^*}{\eps}\nabla\xbar\psi^{\,n+1}_{j,k}+\gamma\frac{b^*}{\eps}(\xbar{\bv}^{\,n+1}_{j,k})^\perp=\bm0,\label{3.48a}\\
\frac{\xbar{\phi}^{\,n+1}_{j,k}-\,\xbar{\phi}^{\,n}_{j,k}}{\dt}+\left(1-\frac{1}{2\gamma}\right){\cal R}^{\phi,n}_{j,k}
+\frac{1}{2\gamma}{\cal R}^{\phi,*}_{j,k}&+(1-\gamma)\nu\frac{a^n}{\eps}\bnabla\!\cdot\xbar{\bv}^{\,*}_{j,k}
+\gamma\nu\frac{a^*}{\eps}\bnabla\!\cdot\xbar{\bv}^{\,n+1}_{j,k}=0,\label{3.48b}\\
\frac{\xbar\theta^{\,n+1}_{j,k}-\,\xbar\theta^{\,n}_{j,k}}{\dt}+\left(1-\frac{1}{2\gamma}\right){\cal R}^{\theta,n}_{j,k}
+\frac{1}{2\gamma}{\cal R}^{\theta,*}_{j,k}&=0,\label{3.48c}\\
\frac{\xbar q^{\,n+1}_{j,k}-\,\xbar q^{\,n}_{j,k}}{\dt}+\left(1-\frac{1}{2\gamma}\right){\cal R}^{q,n}_{j,k}
+\frac{1}{2\gamma}{\cal R}^{q,*}_{j,k}&=0,\label{3.48d}
\end{align}
\label{3.48}
\end{subequations}
and is realized in Steps~4--6.

\smallskip
\noindent
$\bullet\,$ \textbf{Step 4 (Compute $\xbar\theta^{\,n+1}_{j,k}$ and $\,\xbar q^{\,n+1}_{j,k}$ explicitly).} We solve equations
\eref{3.48c}--\eref{3.48d} directly to obtain
\begin{equation*}
\begin{aligned}
&\xbar\theta^{\,n+1}_{j,k}=\,\xbar\theta^{\,n}_{j,k}-\left(1-\frac{1}{2\gamma}\right)\dt{\cal R}^{\theta,n}_{j,k}-
\frac{1}{2\gamma}\dt{\cal R}^{\theta,*}_{j,k},\qquad
\xbar q^{\,n+1}_{j,k}=\,\xbar q^{\,n}_{j,k}-\left(1-\frac{1}{2\gamma}\right)\dt{\cal R}^{q,n}_{j,k}-
\frac{1}{2\gamma}\dt{\cal R}^{q,*}_{j,k}.
\end{aligned}
\end{equation*}

\smallskip
\noindent
$\bullet\,$ \textbf{Step 5 (Solve the linear elliptic equation for $\,\xbar\psi^{\,n+1}_{j,k}$ and compute $\,\xbar{\phi}^{\,n+1}_{j,k}$).}
We discretize the linear elliptic equation \eref{3.13}--\eref{3.14} to obtain the following the following linear system of algebraic
equations for $\,\xbar\psi^{\,n+1}_{j,k}$:
\begin{equation}
\begin{aligned}
\big(\eps^2+a^*b^*(\gamma\dt)^2\big)\,\xbar\psi^{\,n+1}_{j,k}=&-a^*b^*(\gamma\dt)^2\big(\nu\,\xbar q^{\,n+1}_{j,k}-\nu\bar{\beta}y_k
-\,\xbar\theta^{\,n+1}_{j,k}\big)+\nu a^*b^*(\gamma\dt)^2\Delta\xbar\psi^{\,n+1}_{j,k}\\
&-\eps\nu a^*\gamma\dt\left(\bnabla\!\cdot\xbar{\bv}^{\,n}_{j,k}-
\left(1-\frac{1}{2\gamma}\right)\dt\bnabla\!\cdot\!\bm{{\cal R}}^{\bv,n}_{j,k}\right)\\
&-\eps\nu\Big(a^n(1-\gamma)\dt\bnabla\!\cdot\xbar{\bv}^{\,*}_{j,k}-
a^*(\gamma\dt)\frac{1}{2\gamma}\dt\bnabla\!\cdot\!\bm{{\cal R}}^{\bv,*}_{j,k}\Big)\\
&+\eps^2\left(\xbar\psi^{\,n}_{j,k}-\left(1-\frac{1}{2\gamma}\right)\dt{\cal R}^{\psi,n}_{j,k}-\frac{1}{2\gamma}\dt{\cal R}^{\psi,*}_{j,k}
\right)\\
&-a^*b^n\gamma(1-\gamma)(\dt)^2\Big(\nu\,\xbar q^{\,*}_{j,k}-\nu\bar{\beta}y_k+\,\xbar{\phi}^{\,*}_{j,k}-\nu\Delta\xbar\psi^{\,*}_{j,k}
\Big),
\end{aligned}
\label{3.49}
\end{equation}
where the discrete Laplacian $\Delta\xbar\psi^{\,*}_{j,k}$ is defined as in \eref{3.43a}. Similarly, the discrete divergence operators
$\bnabla\!\cdot\xbar{\bv}^{\,*}_{j,k}$ and $\bnabla\!\cdot\!\bm{{\cal R}}^{\bv,*}_{j,k}$ are defined as in \eref{3.38} and
\eref{3.39}--\eref{3.42}, respectively. After solving \eref{3.49}, we obtain
$\,\xbar{\phi}^{\,n+1}_{j,k}=\,\xbar\psi^{\,n+1}_{j,k}-\,\xbar\theta^{\,n+1}_{j,k}$.

\smallskip
\noindent
$\bullet\,$ \textbf{Step 6 (Compute $\,\xbar{\bv}^{\,n+1}_{j,k}$).} Once $\,\xbar\psi^{\,n+1}_{j,k}$ is available, we compute
\begin{equation*}
\begin{aligned}
\xbar{\bv}^{\,n+1}_{j,k}&=\frac{1}{\eps^2+(b^*\gamma\dt)^2}\Big[\!-\eps\frac{1}{2\gamma}\dt\left(\eps\bm{{\cal R}}^{\bv,*}_{j,k}
-b^*\gamma\dt\big(\bm{{\cal R}}^{\bv,*}_{j,k}\big)^\perp\right)\\
&-\eps\bigg(1-\frac{1}{2\gamma}\bigg)\dt\Big(\eps\bm{{\cal R}}^{\bv,n}_{j,k}-b^*\gamma\dt\big(\bm{{\cal R}}^{\bv,n}_{j,k}\big)^\perp\Big)
-b^nb^*\gamma(1-\gamma)(\dt)^2\big(\,\xbar{\bv}^{\,*}_{j,k}-\nabla^\perp\xbar\psi^{\,*}_{j,k}\big)\\
&+\eps^2\,\xbar{\bv}^{\,n}_{j,k}-\eps b^n(1-\gamma)\dt\big((\,\xbar{\bv}^{\,*}_{j,k})^\perp+\nabla\xbar\psi^{\,*}_{j,k}\big)
-\eps b^*\gamma\dt\big((\,\xbar{\bv}^{\,n}_{j,k})^\perp+\nabla\xbar\psi^{\,n+1}_{j,k}\big)+
(b^*\gamma\dt)^2\nabla^\perp\xbar\psi^{\,n+1}_{j,k}\Big],
\end{aligned}
\end{equation*}
where the discrete gradient operator $\nabla\xbar\psi^{\,*}_{j,k}$ is defined as in \eref{3.43b}.
\begin{rmk}
For brevity, we omit the intermediate algebraic manipulations and present only the final formulae in Steps 5 and 6.
\end{rmk}
\begin{rmk}
It is worth noting that, when $\gamma=1$ in \eref{3.45}, Steps 1--3 in the second-order update reduce exactly to the update procedure of 
the first-order SI, fully discrete AP scheme \eref{3.44}.
\end{rmk}

\section{All-Rossby-Number Dual Formulation Finite-Volume Method}\label{sec4}
We recall that the primary objective of this work is to develop an efficient numerical method for the 2-D TRSW equations that remains valid
across all-Rossby-number regimes. Our approach is based on a dual formulation framework recently introduced in
\cite{Abgrall_2025,CFKM,CKKM}. Specifically, we solve both the conservative and nonconservative (primitive) forms of the TRSW equations
simultaneously, exploiting their complementary strengths. In high-Rossby-number regimes, accurate shock capturing requires the conservative
formulation, whereas in low-Rossby-number regimes, the nonconservative formulation is better suited for enforcing the divergence-free
constraint and for recovering the correct TQG asymptotics, while both formulations converge to the same physically relevant TQG limit. To
this end, we employ the AP scheme described in \S\ref{sec3} to compute numerical solutions of the augmented primitive system \eref{2.9},
together with a finite-volume, Riemann-problem-solver-free CU scheme, presented in \S\ref{sec41}, for the conservative TRSW system
\eref{2.20}. Having obtained two sets of numerical solutions---one conservative and one primitive---we introduce in \S\ref{sec43} a
post-processing procedure that blends them to produce a DF-FV method valid for all-Rossby-number regimes. This post-processing ensures that,
in a high-Rossby-number regime, the solution is effectively governed by the CU scheme applied to the conservative system, while in a
low-Rossby-number regime, it is primarily dictated by the AP solution of the primitive system.

\subsection{Semi-Discrete CU Scheme}\label{sec41}
In this section, we describe the semi-discrete CU scheme for the conservative formulation of the TRSW equations \eref{2.20}. To this end, we
write \eref{2.20} in the vector form
\begin{equation}
\bU_t+\bF(\bU)_x+\bG(\bU)_y=\bS(\bU),
\label{4.1}
\end{equation}
where $\bU:=(h,hu,hv,h\Theta)^\top$ is the vector of conservative variables, and $\bF,\bG$ and $\bS$ are the fluxes and the source,
respectively, given by
\begin{equation}
\bF(\bU)=\begin{pmatrix}hu\\[0.5ex]hu^2+\dfrac{\nu}{2\eps^2}\Theta h^2\\[1ex]huv\\hu\Theta\end{pmatrix},\quad
\bG(\bU)=\begin{pmatrix}hv\\huv\\[0.5ex]hv^2+\dfrac{\nu}{2\eps^2}\Theta h^2\\[1ex]hv\Theta\end{pmatrix},\quad
\bS(\bU)=\begin{pmatrix}0\\[0.5ex]\dfrac{1+\eps\bar{\beta}y}{\eps}hv\\[1ex]\dfrac{1+\eps\bar{\beta}y}{-\eps}hu\\[0.5ex]0\end{pmatrix}.
\label{4.2}
\end{equation}
	
We assume that the cell averages $\xbar{\bU}_{j,k}(t)\approx\frac{1}{\dx\dy}\iint_{I_{j,k}}\bU(x,y,t)\,\d x\,\d y$ are available at a
certain time level $t$ and are evolved in time according to the following system of ODEs:
\begin{equation}
\frac{\d}{\d t}\,\xbar{\bU}_{j,k}=\bm{{\cal L}}_{j,k}\quad\mbox{with}\quad\bm{{\cal L}}_{j,k}
:=-\frac{\bm{{\cal F}}_{\jR,k}-\bm{{\cal F}}_{\jL,k}}{\dx}-\frac{\bm{{\cal G}}_{j,\kR}-\bm{{\cal G}}_{j,\kL}}{\dy}+\xbar{\bS}_{j,k}.
\label{4.3}
\end{equation}
Here,
\begin{equation}
\begin{aligned}
&\bm{{\cal F}}_{\jR,k}=\frac{\sigma^+_{\jR,k}\bF\big(\bU^-_{\jR,k}\big)-\sigma^-_{\jR,k}\bF\big(\bU^+_{\jR,k}\big)}{\sigma^+_{\jR,k}-
\sigma^-_{\jR,k}}+\frac{\sigma^+_{\jR,k}\,\sigma^-_{\jR,k}}{\sigma^+_{\jR,k}-\sigma^-_{\jR,k}}\Big(\bU^+_{\jR,k}-\bU^-_{\jR,k}\Big),\\
&\bm{{\cal G}}_{j,\kR}=\frac{\sigma^+_{j,\kR}\bG\big(\bU^-_{j,\kR}\big)-\sigma^-_{j,\kR}\bG\big(\bU^+_{j,\kR}\big)}{\sigma^+_{j,\kR}-
\sigma^-_{j,\kR}}+\frac{\sigma^+_{j,\kR}\,\sigma^-_{j,\kR}}{\sigma^+_{j,\kR}-\sigma^-_{j,\kR}}\Big(\bU^+_{j,\kR}-\bU^-_{j,\kR}\Big),
\end{aligned}
\label{4.4}
\end{equation}
are the CU numerical fluxes from \cite{Kurganov_2001}. The interface values $\bU_{\jR,\,k}^\pm:=\bU\big(\bV_{\jR,\,k}^\pm\big)$ and 
$\bU_{j,\,\kR}^\pm:=\bU\big(\bV_{j,\,\kR}^\pm\big)$ are computed from the reconstructed primitive variables $\bV_{\jR,\,k}^\pm$ and 
$\bV_{j,\,\kR}^\pm$ (see \S\ref{sec33}) via a straightforward transformation $\bU(\bV)$ from $\bV$ to $\bU$. In addition, the quantities
$\sigma_{\jR,\,k}^\pm$ and $\sigma_{j,\,\kR}^\pm$ denote the one-sided local speeds of propagation for the conservative system
\eref{4.1}--\eref{4.2} in the $x$- and $y$-directions, respectively. They are estimated using the largest and smallest eigenvalues of the
corresponding flux Jacobians as follows:
\begin{equation}
\begin{aligned}
&\sigma^+_{\jR,k}=\max\left\{u^-_{\jR,k}+\frac{1}{\eps}\sqrt{\nu(h\Theta)^-_{\jR,k}},\,
u^+_{\jR,k}+\frac{1}{\eps}\sqrt{\nu(h\Theta)^+_{\jR,k}},\,0\right\},\\
&\sigma^-_{\jR,k}=\min\,\left\{u^-_{\jR,k}-\frac{1}{\eps}\sqrt{\nu(h\Theta)^-_{\jR,k}},\,
u^+_{\jR,k}-\frac{1}{\eps}\sqrt{\nu(h\Theta)^+_{\jR,k}},\,0\right\},\\
&\sigma^+_{j,\kR}=\max\left\{v^-_{j,\kR}+\frac{1}{\eps}\sqrt{\nu(h\Theta)^-_{j,\kR}},\,
v^+_{j,\kR}+\frac{1}{\eps}\sqrt{\nu(h\Theta)^+_{j,\kR}},\,0\right\},\\
&\sigma^-_{j,\kR}=\min\,\left\{v^-_{j,\kR}-\frac{1}{\eps}\sqrt{\nu(h\Theta)^-_{j,\kR}},\,
v^+_{j,\,\kR}-\frac{1}{\eps}\sqrt{\nu(h\Theta)^+_{j,\kR}},\,0\right\}.
\end{aligned}
\label{4.5}
\end{equation}
Finally, the cell averages of $\bS$ are approximated using the midpoint rule, which results in
\begin{equation}
\xbar{\bS}_{j,k}=\bigg(0,\,\frac{1+\eps\bar{\beta}y_k}{\eps}\left(\xbar{hv}\right)_{j,k},\,
-\frac{1+\eps\bar{\beta}y_k}{\eps}\left(\xbar{hu}\right)_{j,k},\,0\bigg)^\top.
\label{4.6}
\end{equation}
We note that most of the indexed quantities in \eref{4.3}--\eref{4.6} are time-dependent, but we have omitted this dependence for the sake 
of brevity.
	
\subsection{Second-Order Fully Discrete CU Scheme}
As it was mentioned above, in the DF-FV approach, the $\bU$- and $\bV$-solutions are evolved in time simultaneously. We therefore discretize
the ODE system \eref{4.3} using the explicit part of the ARS(2,2,2) method. The second-order fully discrete CU scheme for \eref{2.20} then
reads as
\begin{equation*}
\begin{aligned}
&\xbar{\bU}^{\,*}_{j,k}=\,\xbar{\bU}_{j,k}^{\,n}+\gamma\dt\bm{{\cal L}}_{j,k}^{\,n},\\
&\xbar{\bU}_{j,k}^{\,n+1}=\,\xbar{\bU}_{j,k}^{\,n}+\Big(1-\frac{1}{2\gamma}\Big)\dt\bm{{\cal L}}_{j,k}^{\,n}+
\frac{1}{2\gamma}\dt\bm{{\cal L}}_{j,k}^{\,*}.
\end{aligned}
\end{equation*}
where $\,\xbar{\bU}^{\,*}_{j,k}$ and $\bm{{\cal L}}_{j,k}^*$ denote the cell averages and right-hand side \eref{4.3} computed after the 
first stage of the ARS(2,2,2) method, and $\,\xbar{\bU}^{\,n}_{j,k}:=\,\xbar{\bU}_{j,k}(t^n)$ and
$\,\xbar{\bU}^{\,n+1}_{j,k}:=\,\xbar{\bU}_{j,k}(t^{n+1})$.

\subsection{Post-Processing}\label{sec43}
As the result of evolving both conservative $\bm U$ and primitive $\bm V$ solutions in time according to \eref{4.3} and \eref{3.46},
\eref{3.48}, respectively, each time step will produce two numerical solutions: one that is AP but nonconservative, and another one that is 
conservative but non-AP. This naturally raises a key question: how can one construct an oscillation-free AP scheme that combines the
advantages of the primitive and conservative solutions within a unified framework? 

To address this issue, we follow \cite{CKKM} and introduce a post-processing that couples the two solutions and yields the DF-FV method, 
capable of accurately handling all-Rossby-number regimes. Specifically, upon completion of every stage of the ARS(2,2,2) time evolution 
method, we modify the obtained $\bV$-solution by replacing $\xbar{\bV}^{\,*}_{j,k}$ with 
$\big(1-\varphi(\eps)\big)\bV\big(\xbar{\bU}^{\,*}_{j,k}\big)+\varphi(\eps)\xbar{\bV}^{\,*}_{j,k}$ (after the first stage) and 
$\xbar{\bV}^{\,n+1}_{j,k}$ with $\big(1-\varphi(\eps)\big)\bV\big(\xbar{\bU}_{j,k}^{\,n+1}\big)+\varphi(\eps)\xbar{\bV}^{\,n+1}_{j,k}$
(after the second stage), where $\bV(\bU)$ is a transformation function from the conservative to primitive variables.

Here, $\varphi(\eps)$ is a switching function, which is supposed to be $\sim1$ for $\eps\sim0$ and $\sim0$ for $\eps\sim1$. In general, when
$\eps>0.1$, the flow is in a high-Rossby-number regime, while for smaller values of $\eps$, it corresponds to a low-Rossby-number regime.
Therefore, the post-processing function should be almost zero for $\eps>0.1$. One can design many switching functions that satisfy these
requirements. In the numerical experiments reported in \S\ref{sec5}, we have used $\varphi(\eps)=\exp(-2000\eps^6)$, which seems to be a
robust choice.

\section{Numerical Results}\label{sec5}
In this section, we conduct several numerical experiments to test the developed AP DF-FV method. We also compare its results with those
obtained using a second-order CU scheme as a spatial discretization for the conservative formulation of the TRSW equations described in
\S\ref{sec41}, combined with an explicit second-order strong stability preserving Runge–Kutta time discretization \cite{Gottlieb_2001},
hereafter referred to as the Explicit scheme. For all numerical experiments, the minmod parameter is set to $\mu=1.3$, and the CFL number is
fixed at $0.25$. The time step for the AP DF-FV method is asymptotically independent of $\eps$ and restricted by 
\begin{equation*}
\dt_{\rm AP}={\rm CFL}\cdot\min\left\{\frac{\dx}{\max\limits_{j,k}\Big\{s_{\jR,k}^+,-s_{\jR,k}^-\Big\}},\,
\frac{\dy}{\max\limits_{j,k}\Big\{s_{j,\kR}^+,-s_{j,\kR}^-\Big\}}\right\},
\end{equation*}
where $s_{\jR,k}^{\pm}$ and $s_{j,\kR}^{\pm}$ are defined in \eref{3.35f}--\eref{3.37}, while the Explicit scheme is stable under a
different time-step restriction:
\begin{equation*}
\dt_{\rm EX}={\rm CFL}\cdot\min\left\{\frac{\dx}{\max\limits_{j,k}\Big\{\sigma_{\jR,k}^+,-\sigma_{\jR,k}^-\Big\}},\,
\frac{\dy}{\max\limits_{j,k}\Big\{\sigma_{j,\kR}^+,-\sigma_{j,\kR}^-\Big\}}\right\},
\end{equation*}
where $\eps$-dependent one-sided speeds $\sigma_{\jR,k}^\pm$ and $\sigma_{j,\kR}^\pm$ are defined in \eref{4.5}.

In Examples 5.3--5.5, the test data are taken from the existing studies in which the governing equations are given in the dimensional form.
Physical quantities, reported below using their standard unit symbols, are rescaled to match the nondimensional TRSW equations \eref{2.20}
(or \eref{3.1}) solved by the proposed AP DF–FV method.

\begin{example}({\bf Accuracy Test})
We consider the nondimensional TRSW system \eref{2.20} subject to the initial data
\begin{equation*}
\begin{aligned}
&h(x,y,0)=1+0.9\eps^2\cos(2\pi(x+y)),&&u(x,y,0)=\pi\sin(2\pi x)\cos(2\pi y),\\
&\Theta(x,y,0)=1+0.9\eps\sin(2\pi x)\sin(2\pi y),&&v(x,y,0)=\pi\cos(2\pi x)\sin(2\pi y),
\end{aligned}
\end{equation*}
which are prescribed in the computational domain $[0,1]\times[0,1]$ subject to the periodic boundary conditions. We take $\bar\beta=0$,
$\nu=1$, and compute the numerical solution until the final time $t=0.01$ by the proposed AP DF-FV method on a sequence of uniform meshes
with $\dx=\dy=1/16$, $1/32$, $1/64$, $1/128$, and $1/256$ for various values of $\eps$. The $L^1$-errors and experimental order of
convergence (EOC) for the computed $h$, $hu$, and $h\Theta$ are documented in Tables \ref{tab51}, \ref{tab52}, and \ref{tab53},
respectively, demonstrating that the second order of accuracy is achieved independently of $\eps$.
\begin{table}[ht!]
\centering
\begin{tabular}{|c|cc|cc|cc|cc|}
\hline
\multirow{2}{*}{$\dx=\dy$}&\multicolumn{2}{c|}{$\eps=1$}&\multicolumn{2}{c|}{$\eps=10^{-2}$}&\multicolumn{2}{c|}{$\eps=10^{-4}$}
&\multicolumn{2}{c|}{$\eps=10^{-6}$}\\
\cline{2-9}
&Error&EOC&Error&EOC&Error&EOC&Error&EOC\\
\hline  		         
1/16 &9.87e-03& -- &3.90e-04& -- &2.94e-07& -- &2.93e-09& -- \\
1/32 &3.02e-03&1.71&2.57e-04&0.60&7.56e-08&1.96&7.54e-10&1.96\\        
1/64 &8.44e-04&1.84&4.11e-05&2.65&1.91e-08&1.98&1.91e-10&1.98\\          
1/128&2.27e-04&1.90&4.84e-06&3.09&4.23e-09&2.18&4.21e-11&2.18\\
1/256&4.86e-05&2.22&7.63e-07&2.66&8.23e-10&2.36&8.07e-12&2.38\\
\hline
\end{tabular}
\caption{\sf Example 5.1: The $L^1$-errors and EOC for $h$ and for different values of $\eps$.\label{tab51}}
\end{table}
\begin{table}[ht!]
\centering
\begin{tabular}{|c|cc|cc|cc|cc|}
\hline
\multirow{2}{*}{$\dx=\dy$}&\multicolumn{2}{c|}{$\eps=1$}&\multicolumn{2}{c|}{$\eps=10^{-2}$}&\multicolumn{2}{c|}{$\eps=10^{-4}$}
& \multicolumn{2}{c|}{$\eps = 10^{-6}$} \\ 
\cline{2-9}
&Error&EOC&Error&EOC&Error&EOC&Error&EOC\\
\hline  		           
1/16 &2.93e-02& -- &1.39e-02& -- &5.28e-02& -- &1.30e-01& -- \\   
1/32 &7.90e-03&1.89&1.33e-02&0.06&1.29e-02&2.04&3.22e-02&2.04\\          
1/64 &2.12e-03&1.89&8.26e-03&0.69&3.10e-03&2.05&7.67e-03&2.05\\          
1/128&5.29e-04&2.01&2.27e-03&1.86&7.25e-04&2.09&1.79e-03&2.10\\
1/256&1.08e-04&2.29&4.68e-04&2.28&1.44e-04&2.33&3.54e-04&2.33\\
\hline
\end{tabular}
\caption{\sf Example 5.1: The same as in Table \ref{tab51}, but for $hu$.\label{tab52}}
\end{table}
\begin{table}[ht!]
\centering
\begin{tabular}{|c|cc|cc|cc|cc|}
\hline
\multirow{2}{*}{$\dx=\dy$}&\multicolumn{2}{c|}{$\eps=1$}&\multicolumn{2}{c|}{$\eps=10^{-2}$}&\multicolumn{2}{c|}{$\eps=10^{-4}$}
&\multicolumn{2}{c|}{$\eps=10^{-6}$} \\ 
\cline{2-9}
&Error&EOC&Error&EOC&Error&EOC&Error&EOC\\
\hline 		           
1/16 &1.07e-02& -- &3.92e-04& -- &3.24e-07& -- &3.19e-09& -- \\   
1/32 &3.34e-03&1.68&2.57e-04&0.61&7.36e-08&2.14&7.27e-10&2.13\\          
1/64 &9.84e-04&1.76&4.11e-05&2.65&1.86e-08&1.98&1.84e-10&1.98\\          
1/128&2.52e-04&1.97&4.85e-06&3.08&4.46e-09&2.06&4.38e-11&2.07\\
1/256&5.19e-05&2.28&7.60e-07&2.67&9.29e-10&2.26&8.52e-12&2.36\\
\hline
\end{tabular}
\caption{\sf Example 5.1: The same as in Table \ref{tab51}, but for $h\Theta$.\label{tab53}}
\end{table}	
\end{example}

\begin{example}({\bf Freely Decaying Wavetrain}) 
In this example, we consider a freely decaying wavetrain in the strongly rotating regime, in which gravity waves are strongly influenced by 
rotation. We begin by considering the RSW model governed by \eref{2.20} with $\Theta\equiv1$ and use the initial data setting from
\cite[Example 5.2]{Buhler_2000}. The initial conditions consist of a zonally symmetric train of gravity waves with a Gaussian envelope in
the meridional direction, defined as follows:
\begin{equation}
\begin{aligned}
&h(x,y,0)=1+0.2\sin x\left(1+\frac{1}{d^2}\Big(2-\frac{4}{d^2}y_c^2\Big)\right){\rm e}^{-y_c^2/d^2},\\
&u(x,y,0)=0.2\sin x\Big(\sqrt{2}-\frac{2}{d^2}y_c\Big){\rm e}^{-y_c^2/d^2},\quad
v(x,y,0)=0.2\cos x\bigg(\frac{2\sqrt{2}}{d^2}y_c-1\bigg){\rm e}^{-y_c^2/d^2},
\end{aligned}
\label{5.1}
\end{equation}
where $d=40$ and $y_c=y-160$. These data are prescribed in the computational domain $[0,2\pi]\times[0,320]$ subject to the periodic boundary
conditions in the $x$-direction and free boundary conditions in the $y$-direction. We take $\bar\beta=0$, $\eps=\nu=1$, and compute the
solution by the proposed AP DF-FV method on a uniform mesh of $126\times162$ cells until the final time $t=20\pi$. Figure \ref{Fig1} (left)
presents one-dimensional (1-D) slices of the fluid thickness $h$ along $y=160$ at times $t=1.4\pi$, $2.8\pi$, and $20\pi$. The reference
solution is obtained by the Explicit scheme on a finer $420\times540$ uniform mesh. We observe that the wave initially steepens and then,
after a brief phase of dissipative decay, adjusts into a steadily propagating nonlinear equilibrium profile in which nonlinearity and
dispersion are balanced. The AP DF-FV results are consistent with those reported in \cite{Buhler_2000} and agree well with the reference
solutions. 
	
Next, we consider the TRSW model and augment the initial data \eref{5.1} with
\begin{equation*}
\Theta(x,y,0)=1+0.5\cos x\left(1+\frac{1}{d^2}\Big(2-\frac{4}{d^2}y_c^2\Big)\right){\rm e}^{-y_c^2/d^2}.
\end{equation*}
We then compute the solution of the TRSW system on the same grids as before and report the obtained results in Figure \ref{Fig1} (right). In
contrast to the RSW model, the TRSW solution exhibits a more pronounced variation in the fluid thickness $h$, reflecting the additional
coupling introduced by the thermodynamic variable $\Theta$. Moreover, the nonlinear adjustment process is noticeably slower: even at
$t=20\pi$, the solution remains in the decay phase and has not yet reached a steady propagating state. 
\begin{figure}[ht!]
\centering
\includegraphics[height=0.23\textheight]{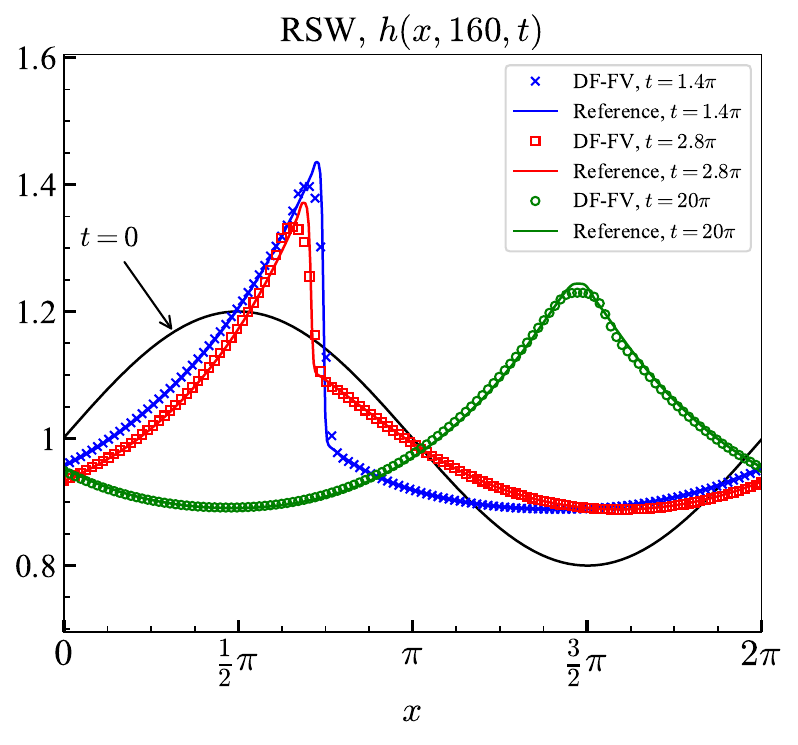}\hspace{1cm}
\includegraphics[height=0.23\textheight]{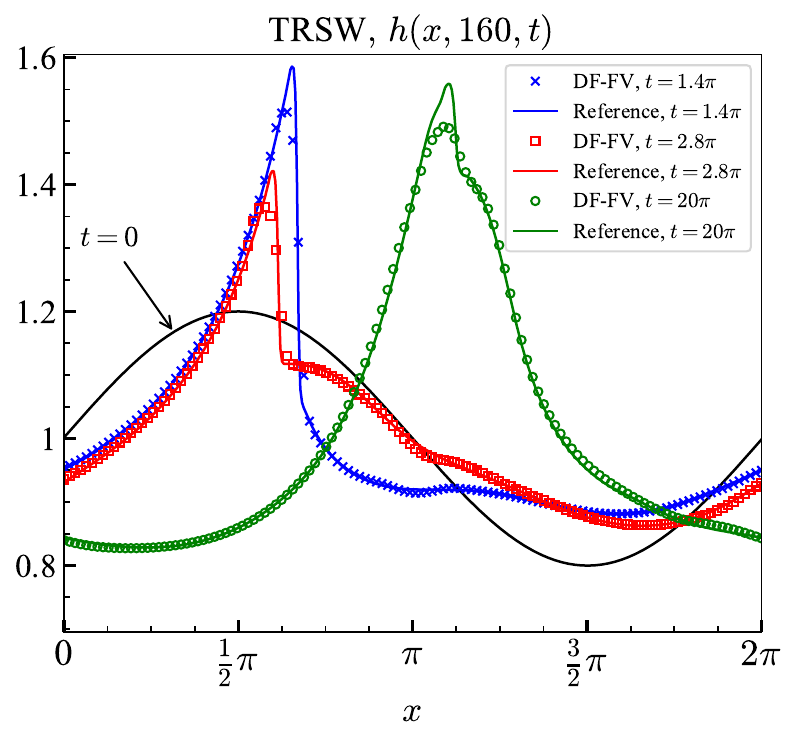}
\caption{\sf Example 5.2: 1-D slices of the fluid thickness $h$ along $y=160$, showing the initial sinusoidal shape, shock formation, decay,
and adjustment for the RSW (left) and TRSW (right) equations.\label{Fig1}}
\end{figure}

This example highlights the qualitative differences between the RSW and TRSW models in the strongly rotating regime and further demonstrates
the capability of the proposed AP DF–FV method to accurately and robustly resolve the flow dynamics---including nonlinear wave adjustment
and the formation of shock structures---for both the RSW and TRSW systems.
\end{example}
  
\begin{example}({\bf Vortex Pair Interaction})
In this example, we examine the evolution of the vortex pair interaction problem to demonstrate the effectiveness of the proposed AP DF-FV
method in the intermediate regime. The initial data taken from \cite{Eldred_2019,Karasozen_2022,Tambyah_2025} are given by 
\begin{equation*}
\begin{aligned}
&h(x,y,0)=H_0-\Phi_0\left({\rm e}^{-\frac{\tilde x_1^2+\tilde y_1^2}{2}}+{\rm e}^{-\frac{\tilde x_2^2+\tilde y_2^2}{2}}-
\frac{9\pi}{400}\right),\quad\Theta(x,y,0)=\g\left(1-0.05\sin\!\Big(2\pi\frac{x}{L_x}\Big)\right),\\
&u(x,y,0)=-\frac{40\g\,\Phi_0}{3f_0L_y}\left(\tilde{\tilde{y}}_1{\rm e}^{-\frac{\tilde x_1^2+\tilde y_1^2}{2}}
+\tilde{\tilde{y}}_2\,{\rm e}^{-\frac{\tilde x_2^2+\tilde y_2^2}{2}}\right),\quad
v(x,y,0)=\frac{40\g\,\Phi_0}{3f_0L_x}\left(\tilde{\tilde{x}}_1{\rm e}^{-\frac{\tilde x_1^2+\tilde y_1^2}{2}}
+\tilde{\tilde{x}}_2\,{\rm e}^{-\frac{\tilde x_2^2+\tilde y_2^2}{2}}\right),
\end{aligned}
\end{equation*}
and defined on the computational domain $[0,L_x]\times[0,L_y]$ with $L_x=L_y=5000\,{\rm km}$ with the periodic boundary conditions. Here,
$f_0=6.147\times10^{-5}\,{\rm s^{-1}}$ is the Coriolis parameter (that is, $\beta\equiv0$), $\g=9.80616\,{\rm m/s^2}$ is gravitational
acceleration, $H_0=750\,{\rm m}$ is the mean depth of the fluid layer, $\Phi_0=75\,{\rm m}$ is the mean characteristic buoyancy scale, and 
\begin{equation*}
\begin{aligned}
&\tilde x_{1,2}=\frac{40}{3\pi}\sin\!\Big(\frac{\pi}{L_x}\big(x-x_{c_{1,2}}\big)\Big),&&
\tilde{\tilde x}_{1,2}=\frac{20}{3\pi}\sin\!\Big(\frac{2\pi}{L_x}\big(x-x_{c_{1,2}}\big)\Big),&&
x_{c_1}=\frac{2}{5}L_x,~x_{c_2}=\frac{3}{5}L_x,\\
&\tilde y_{1,2}=\frac{40}{3\pi}\sin\!\Big(\frac{\pi}{L_y}\big(y-y_{c_{1,2}}\big)\Big),&&
\tilde{\tilde y}_{1,2}=\frac{20}{3\pi}\sin\!\Big(\frac{2\pi}{L_y}\big(y-y_{c_{1,2}}\big)\Big),&&
y_{c_1}=\frac{2}{5}L_y,~y_{c_2}=\frac{3}{5}L_y.
\end{aligned}
\end{equation*}
To express the problem in the nondimensional form, we select the reference values $\Theta_0=\g$, $L_0=3(L_x+L_y)/20$, and $V_0=\Theta_0\Phi_0/(L_0f_0)$. Then, according to \eref{2.5}, the dimensionless parameters $\Ro$, 
$\bar\beta$, and $\Bu$ become
\begin{equation*}
\Ro=\eps=\frac{V_0}{L_0f_0}\approx0.087,\quad\bar\beta=0,\quad\mbox{and}\quad\Bu=\nu=\frac{\Theta_0H_0}{L_0^2f_0^2}\approx0.865,
\end{equation*}
indicating that the flow is in an intermediate regime.

We compute the numerical solution of \eref{2.20} (\eref{3.1}) by the proposed AP DF-FV method on three meshes consisting of $300\times300$,
$400\times400$, and $600\times600$ uniform cells. The buoyancy field $\Theta$ obtained at time $t=20\,{\rm h}$ is shown in Figure \ref{Fig2}
(top row). As one can observe, the conducted mesh-refinement study indicates convergence and shows that the computed solutions are
oscillation-free despite having rather large gradients.
\begin{figure}[ht!]
\centering
\includegraphics[height=0.18\textheight]{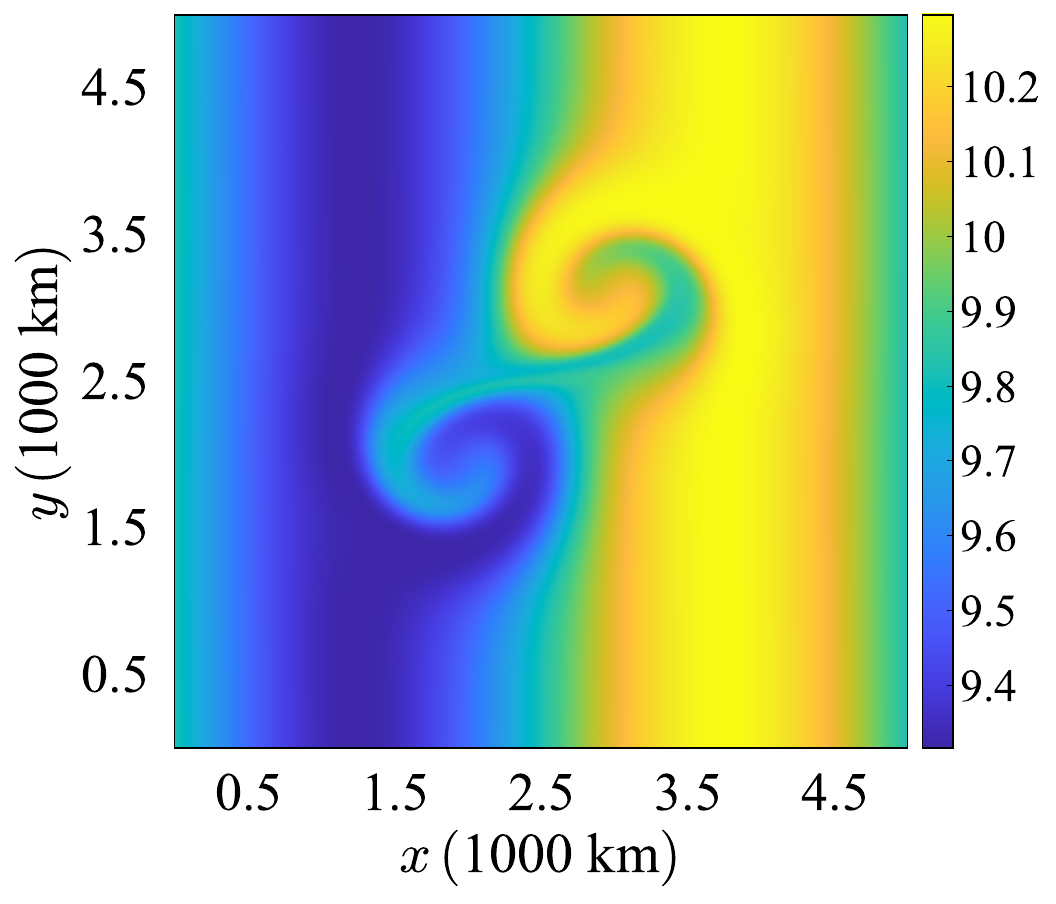}\hspace{1mm}
\includegraphics[height=0.18\textheight]{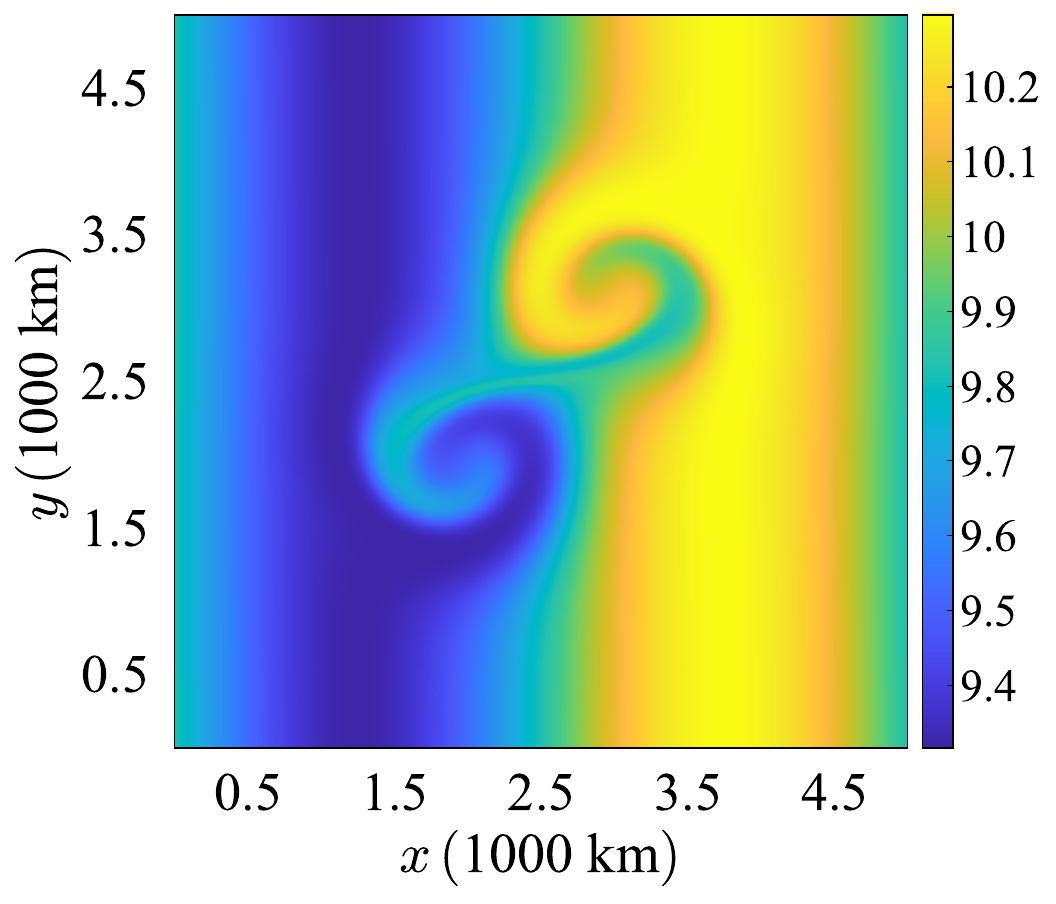}\hspace{1mm}
\includegraphics[height=0.18\textheight]{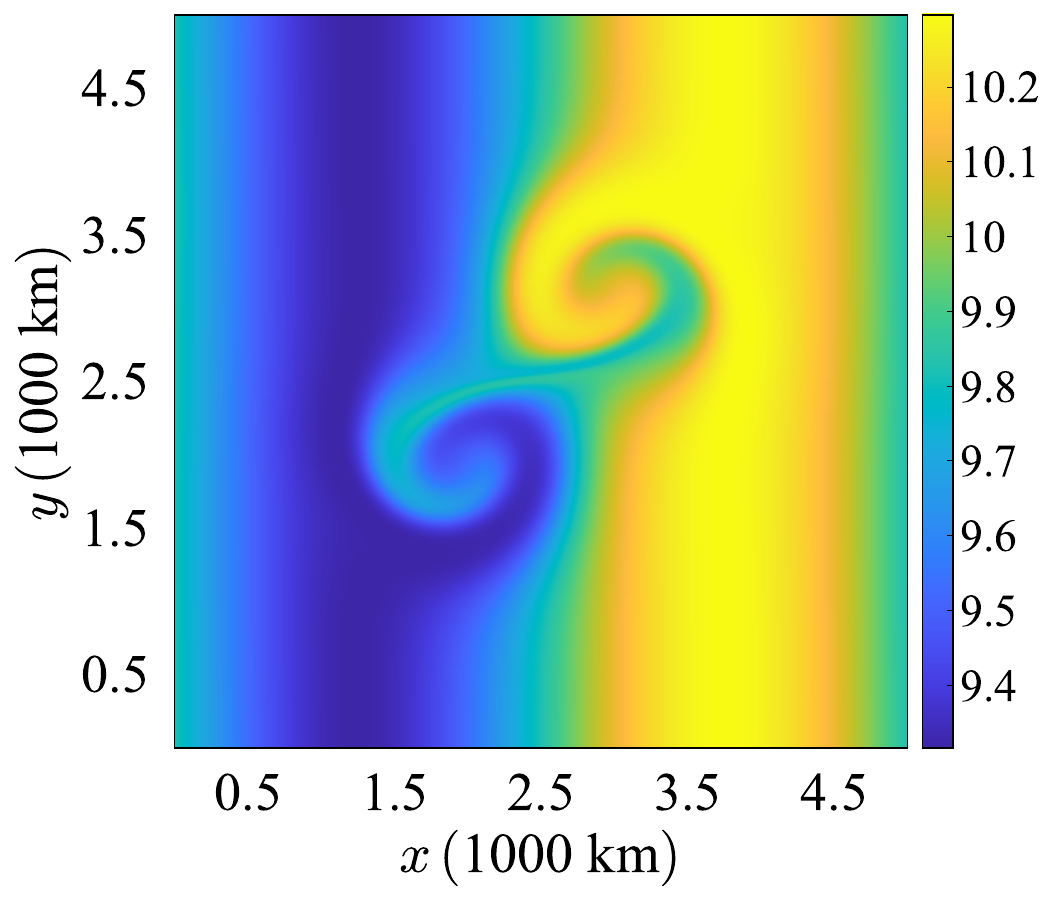}\\[1.5ex]
\includegraphics[height=0.18\textheight]{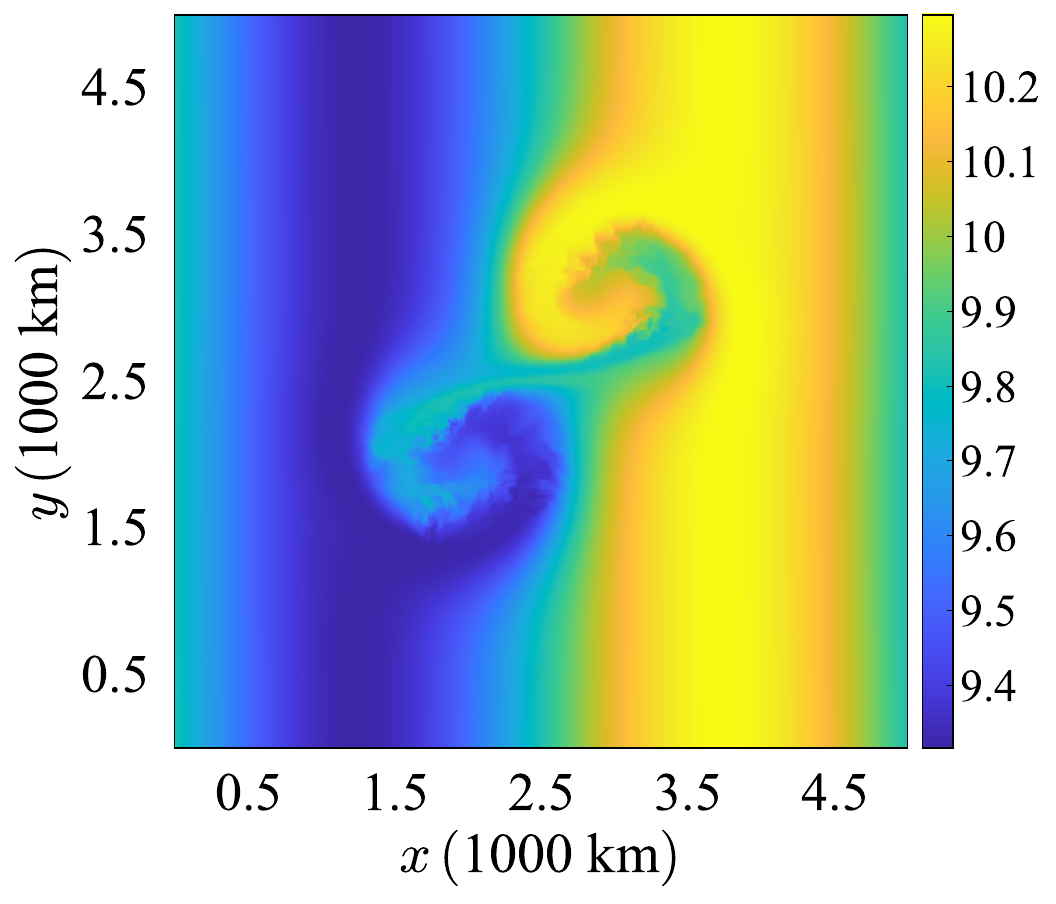}\hspace{1mm}
\includegraphics[height=0.18\textheight]{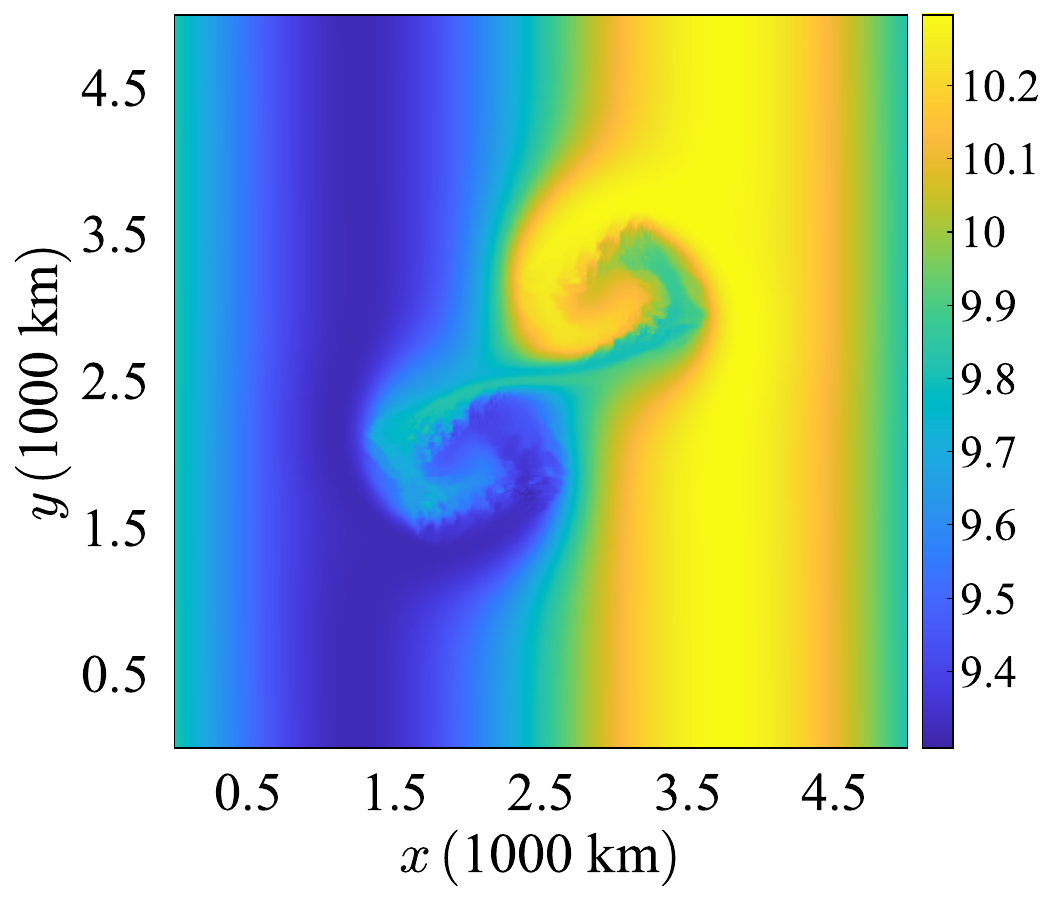}\hspace{1mm}
\includegraphics[height=0.18\textheight]{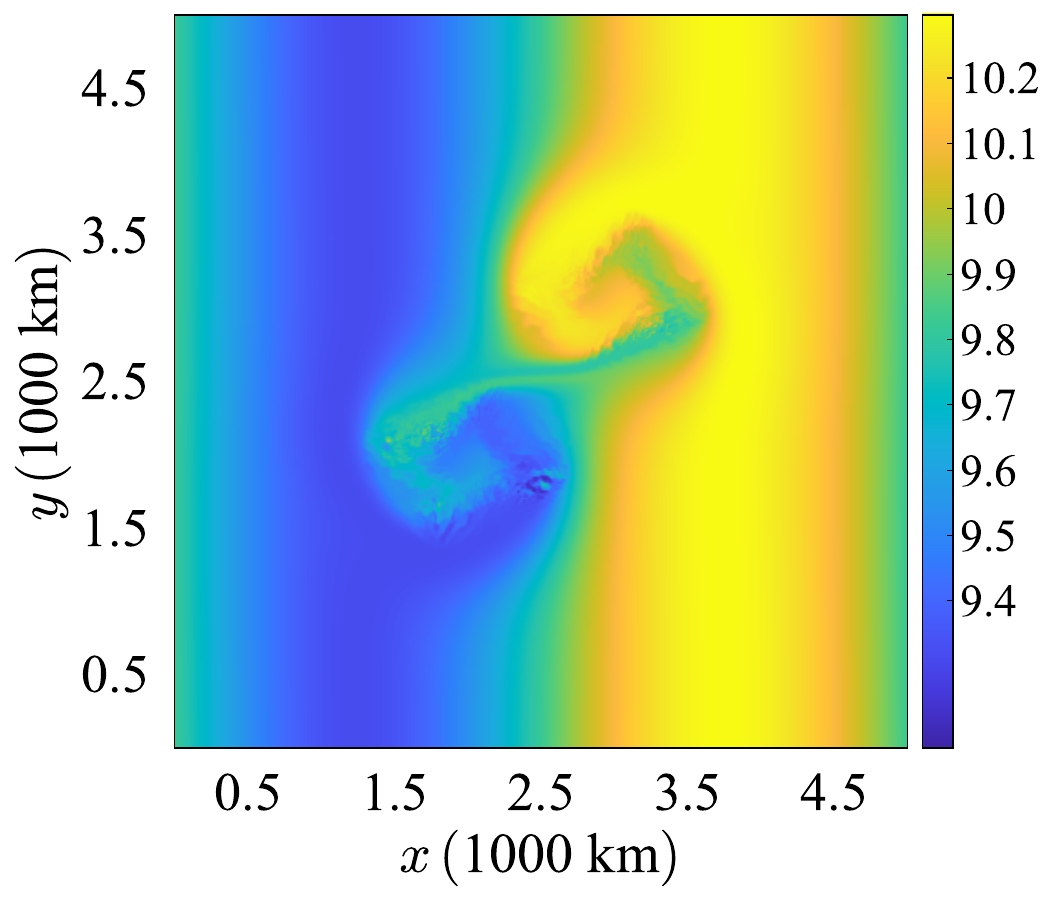}
\caption{\sf Example 5.3: Snapshots of the buoyancy field $\Theta$ computed by the AP DF-FV (top row) and simplified DF-FV (bottom row)
methods on $300\times300$ (left column), $400\times400$ (middle column), and $600\times600$ (right column) uniform meshes at time
$t=20\,{\rm h}$.\label{Fig2}}
\end{figure}
	
When working within the DF-FV framework, it is essential to demonstrate the importance of a proper treatment of the nonconservative terms 
$\widetilde B(\bV)\bV_x$ and $\widetilde C(\bV)\bV_y$ that appear in the nonstiff part of the nonconservative system \eref{3.4}. If a
simplified version of the PCCU discretization of \eref{3.34}, in which the terms that accounting for solution jumps across cell interfaces 
are omitted, that is, if we set $\widetilde{\bm B}_{\bm\Psi,\jR,k}=\widetilde{\bm C}_{\bm\Psi,j,\kR}\equiv\bm0~\forall j,k$, then the
resulting simplified DF-FV method fails to accurately capture the solution. This is illustrated in Figure \ref{Fig2} (bottom row), where one
can observe spurious oscillations that persist even as the mesh is refined. Moreover, when the simulations are continued with this
simplified approach, the numerical solution blows up at times $t=52.76\,{\rm h}$, $28.96\,{\rm h}$, and $23.75\,{\rm h}$ on $300\times300$,
$400\times400$, and $600\times600$ meshes, respectively.
	
Nevertheless, the proposed AP DF-FV method remains stable over long-term simulations. In Figure \ref{Fig3} (top row), we present the
buoyancy field $\Theta$ computed using the AP DF-FV method on a $600\times600$ uniform mesh at larger times $t=33\,{\rm h}\,45\,{\rm min}$,
$67\,{\rm h}\,30\,{\rm min}$, and $101\,{\rm h}\,15\,{\rm min}$. To verify that the obtained solution is sufficiently accurate, we also run
the Explicit scheme on the same uniform mesh and plot the corresponding results in Figure \ref{Fig3} (bottom row). We observe that both
methods capture the evolution of vortex structures and generate numerous smaller vortices over time. Moreover, the obtained AP DF-FV and
Explicit results are in good agreement. To further demonstrate this, we compute the solution by the AP DF-FV method on finer $900\times900$
uniform mesh and plot the 1-D slices of $\Theta$ along $x=2500\,{\rm km}$ at $t=67\,{\rm h}\,30\,{\rm min}$ and
$101\,{\rm h}\,15\,{\rm min}$ in Figure \ref{Fig4}. As one can see, the proposed AP DF-FV method demonstrates strong convergence and the
ability to accurately capture large gradient solution structures, confirming that the method performs well in the intermediate regime.
\begin{figure}[ht!]
\centering
\includegraphics[height=0.18\textheight]{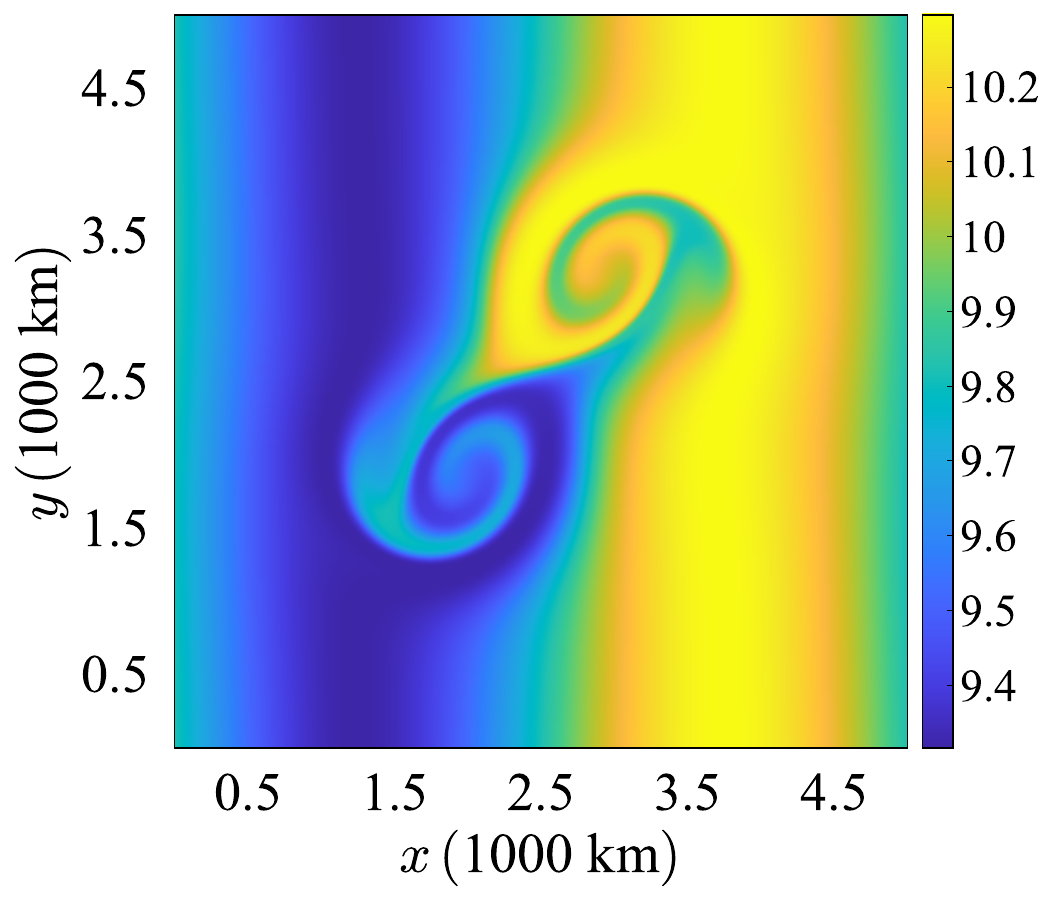}\hspace{1mm}
\includegraphics[height=0.18\textheight]{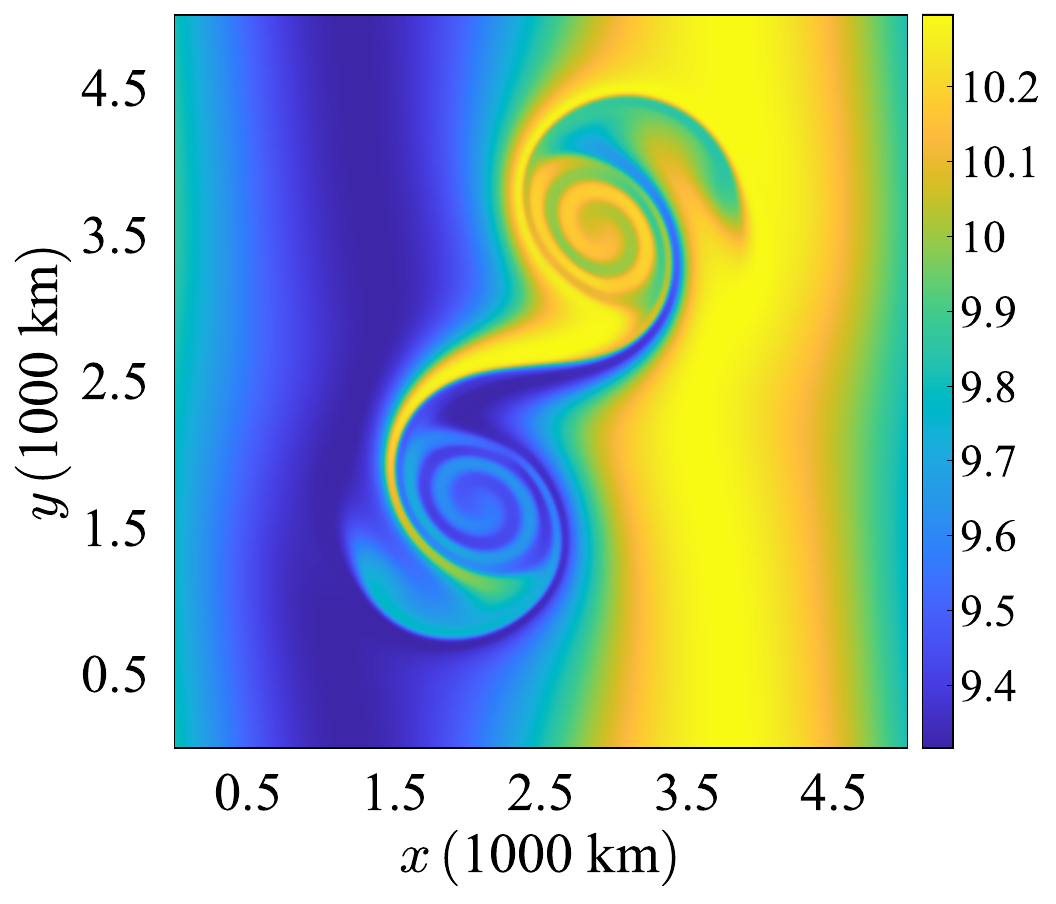}\hspace{1mm}
\includegraphics[height=0.18\textheight]{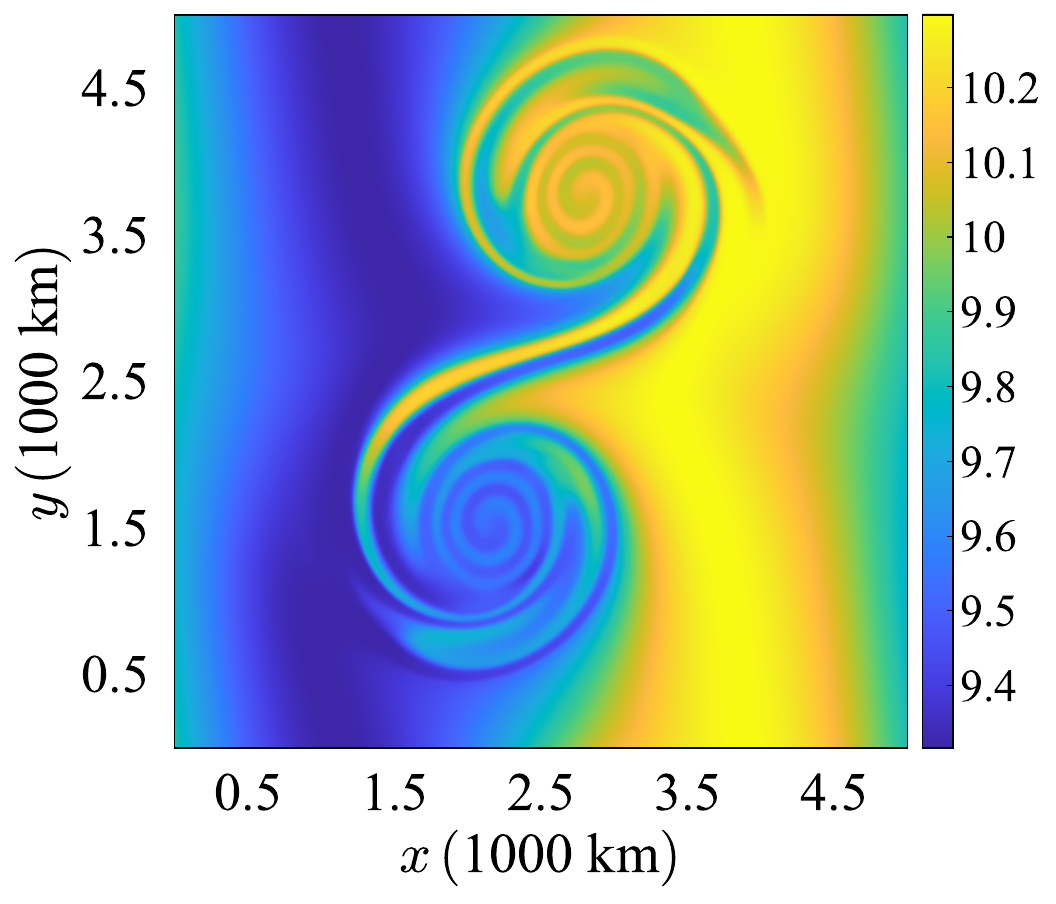}\\[1.5ex]
\includegraphics[height=0.18\textheight]{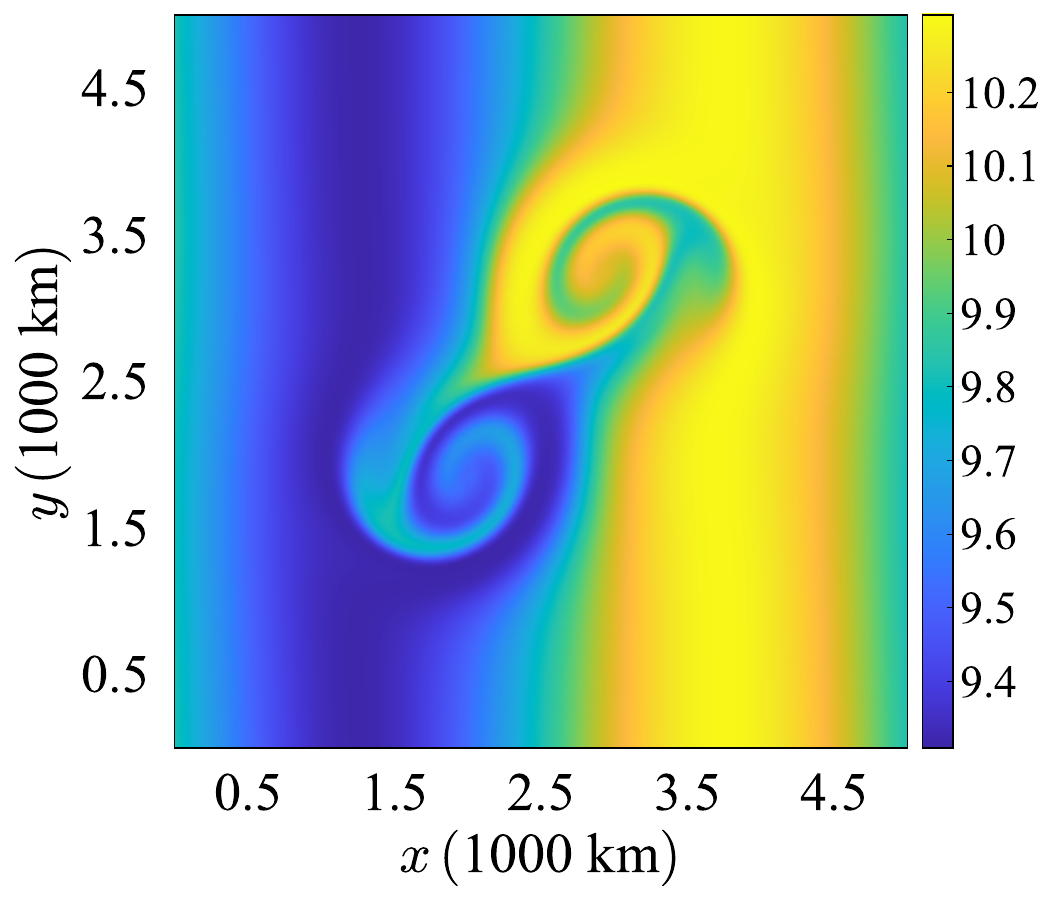}\hspace{1mm}
\includegraphics[height=0.18\textheight]{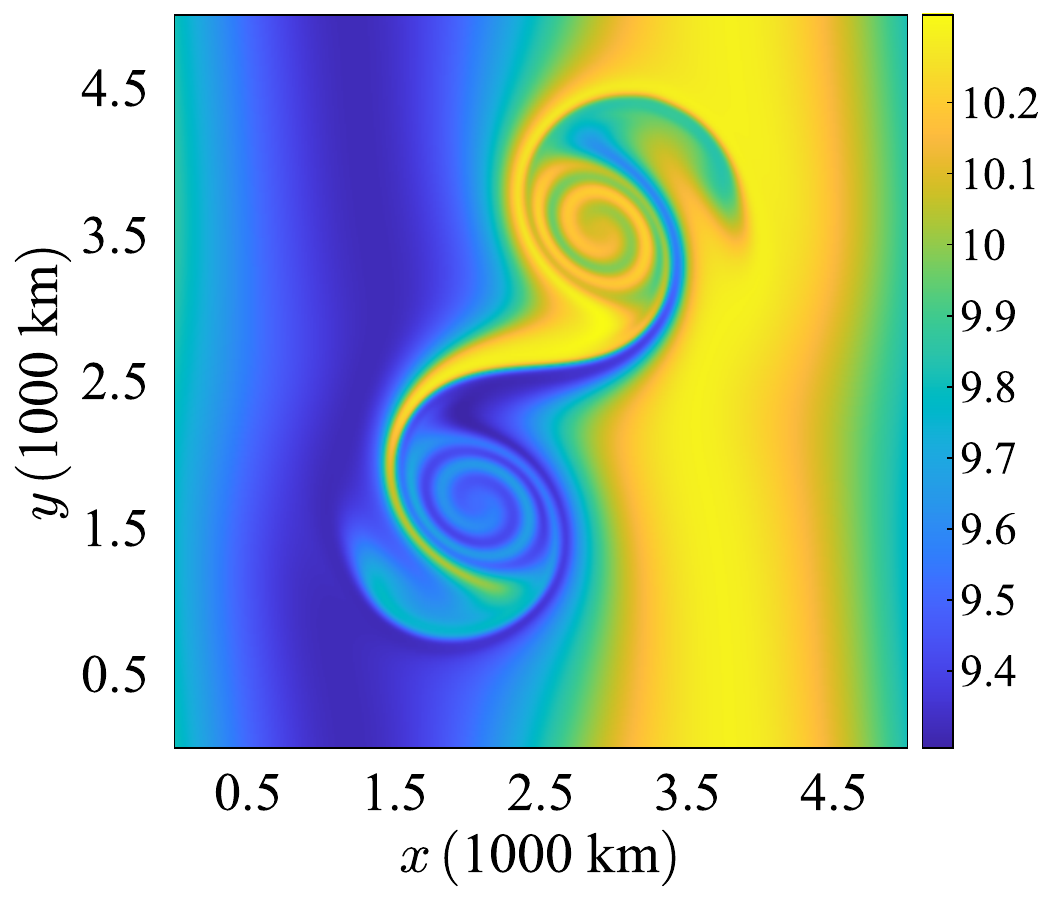}\hspace{1mm}
\includegraphics[height=0.18\textheight]{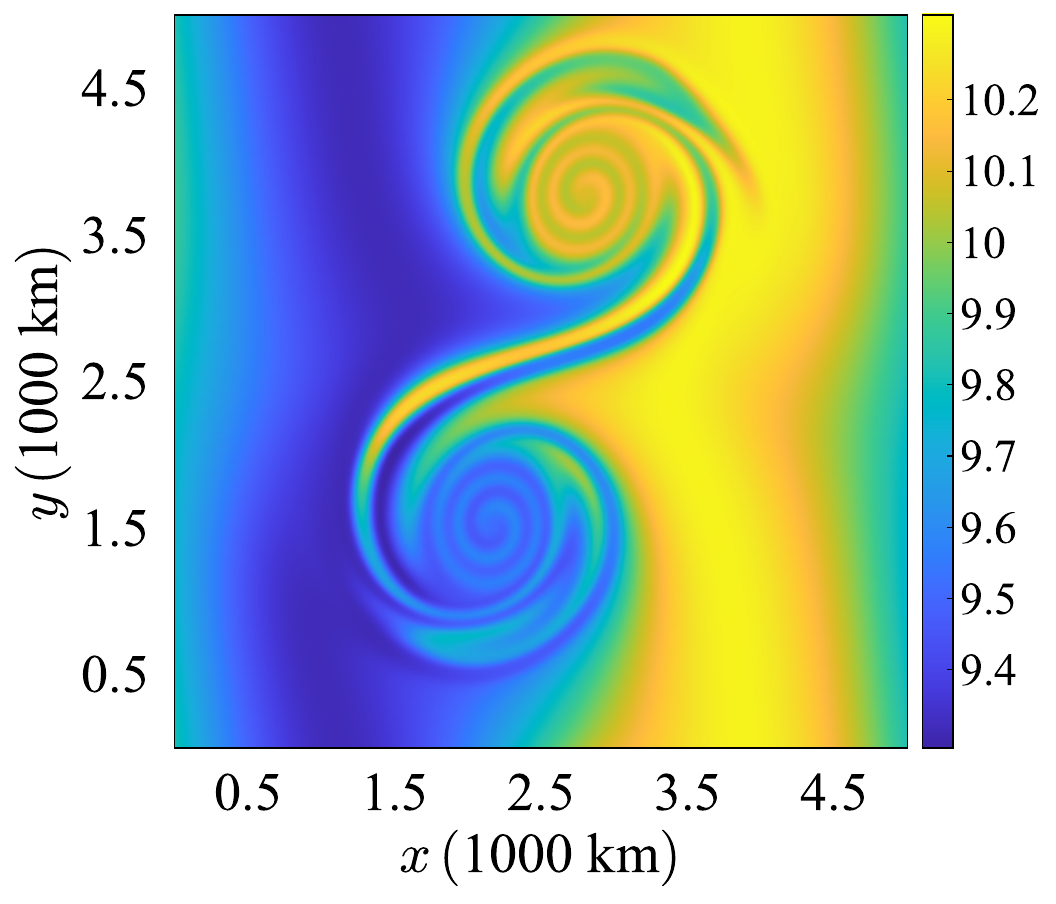}
\caption{\sf Example 5.3: Time snapshots of the buoyancy field $\Theta$ computed by the AP DF-FV (top row) and Explicit (bottom row) methods
at times $t=33\,{\rm h}\,45\,{\rm min}$ (left column), $67\,{\rm h}\,30\,{\rm min}$ (middle column), and $101\,{\rm h}\,15\,{\rm min}$
(right column).\label{Fig3}}
\end{figure}
\begin{figure}[ht!]
\centering
\includegraphics[height=0.22\textheight]{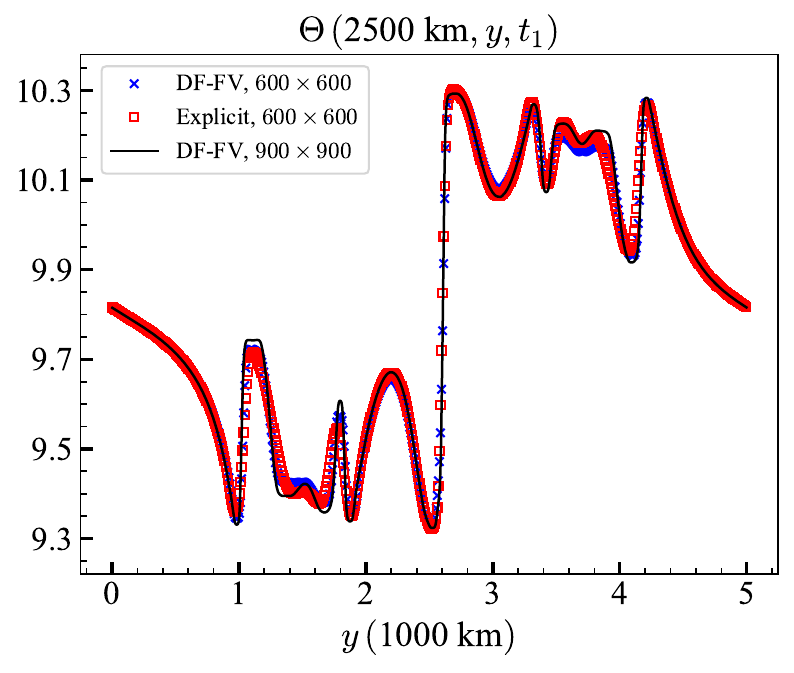}\hspace{1cm}
\includegraphics[height=0.22\textheight]{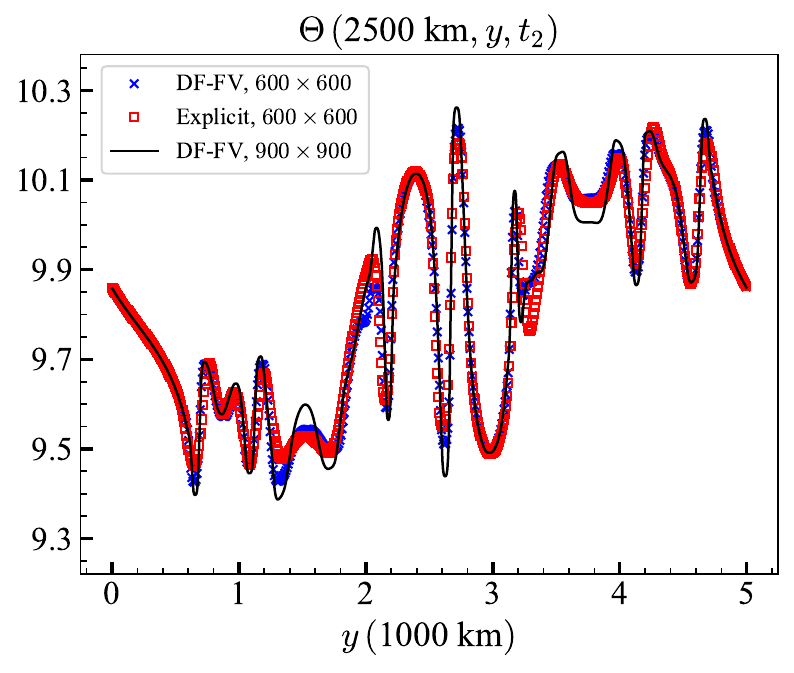}
\caption{\sf Example 5.3: 1-D slices of $\Theta$ along $x=2500\,{\rm km}$ at times $t_1=67\,{\rm h}\,30\,{\rm min}$ (left) and
$t_2=101\,{\rm h}\,15\,{\rm min}$ (right).\label{Fig4}}
\end{figure}
\end{example}
 
\begin{example}({\bf Shear Flow Evolution})
In this example, we study the evolution of shear flow, as described in \cite{Bauer_2019,Zhang_2024}, where the dimensional model is
considered with the Coriolis parameter $f(y)\equiv f_0=6.147\times10^{-5}\,{\rm s^{-1}}$ (that is, $\beta\equiv0$), the gravitational
acceleration $\g=9.80616\,{\rm m/s^2}$, the mean depth of the fluid layer $H_0=1076\,{\rm m}$, and the mean characteristic buoyancy scale is
$\Phi_0=30\,{\rm m}$. The initial data, prescribed in the computational domain $[0,L]\times[0,L]$ with the periodic boundary conditions, are
\begin{equation}
\begin{aligned}
&h(x,y,0)=H_0+\frac{6\Phi_0}{\pi}\Big(1+\frac{1}{10}\sin(4\pi\tilde x)\Big)\sin(2\pi\tilde y)
\exp\!\left(\hf-\frac{72}{\pi^2}\cos^2(\pi\tilde y)\right),\\
&u(x,y,0)=-\frac{12\g\,\Phi_0}{fL}\Big(1+\frac{1}{10}\sin(4\pi\tilde x)\Big)
\Big(\cos(2\pi\tilde y)+\frac{36}{\pi^2}\sin^2(2\pi\tilde y)\Big)\exp\!\left(\hf-\frac{72}{\pi^2}\cos^2(\pi\tilde y)\right),\\
&v(x,y,0)=\frac{12\g\,\Phi_0}{5fL}\cos(4\pi\tilde x)\sin(2\pi\tilde y)\exp\!\left(\hf-\frac{72}{\pi^2}\cos^2(\pi\tilde y)\right),\\
&\Theta(x,y,0)=\g\Big(1+\frac{1}{20}\cos(2\pi\tilde x)\sin(2\pi\tilde y)\Big),
\end{aligned}
\label{5.2}
\end{equation}
where $\tilde x=x/L$, $\tilde y=y/L$, and $L=5000\,{\rm km}$. We take $\Theta_0=\g$, $L_0=L/3$, and $V_0=\Theta_0\Phi_0/(L_0f_0)$ as the
reference values and reformulate the problem in the nondimensional form \eref{2.20} (or \eref{3.1}). Then, the dimensionless parameters
$\Ro$, $\bar\beta$, and $\Bu$ become
\begin{equation*}
\Ro=\eps=\frac{V_0}{L_0f_0}\approx0.028,\quad\bar\beta=0,\quad\mbox{and}\quad\Bu=\nu=\frac{\Theta_0H_0}{(L_0f_0)^2}\approx1.005,
\end{equation*}
indicating that the flow is in the TQG regime. In this case, superimposed perturbations on the initial zonal jet in the $x$-direction evolve
into two pairs of counter-rotating vortices. 
	
It should be observed that in \cite{Bauer_2019,Zhang_2024} a simpler RSW system was considered, and therefore, we first replace the initial
$\Theta$ in \eref{5.2} with $\Theta(x,y,0)\equiv\g$, for which the TRSW system reduces to the RSW one, and compute its solution by the AP
DF-FV and Explicit methods until the final time $t=10\,{\rm d}$ on a $300\times300$ uniform mesh. The obtained vorticities $\omega$ are
shown in Figure \ref{Fig5} (left and middle). To verify the accuracy of the obtained results, we compare them with a reference solution (see
the right panel of Figure \ref{Fig5}) obtained by the explicit fifth-order well-balanced weighted essentially non-oscillatory (WENO) scheme
introduced in \cite{Zhang_2024}. The numerical solution computed by the AP DF-FV method clearly shows a better agreement with the reference
solution, indicating higher numerical accuracy, while the Explicit scheme suffers from ${\cal O}(\eps^{-1})$ numerical dissipation, which
smears out small-scale vortical structures and limits its accuracy in low-Rossby-number regimes. 
\begin{figure}[ht!]
\centering
\includegraphics[height=0.18\textheight]{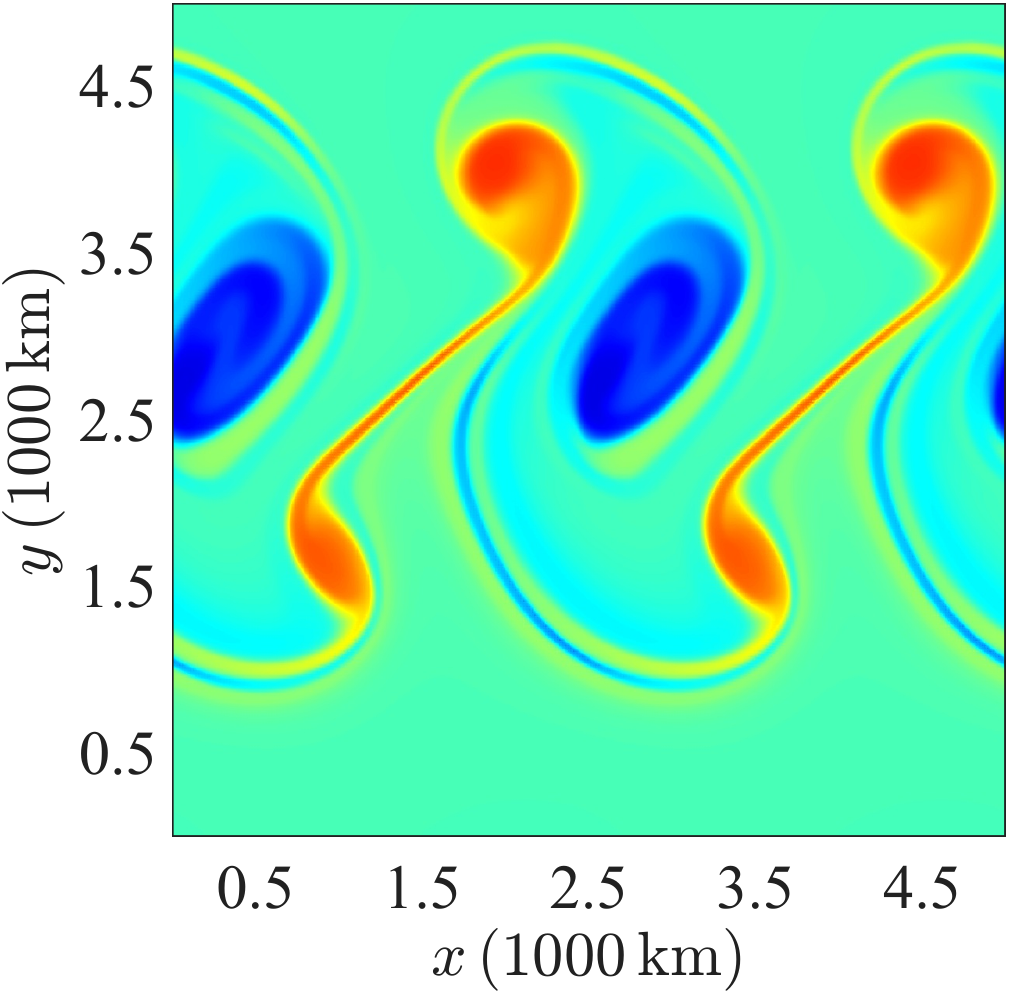}\hspace{5mm}
\includegraphics[height=0.18\textheight]{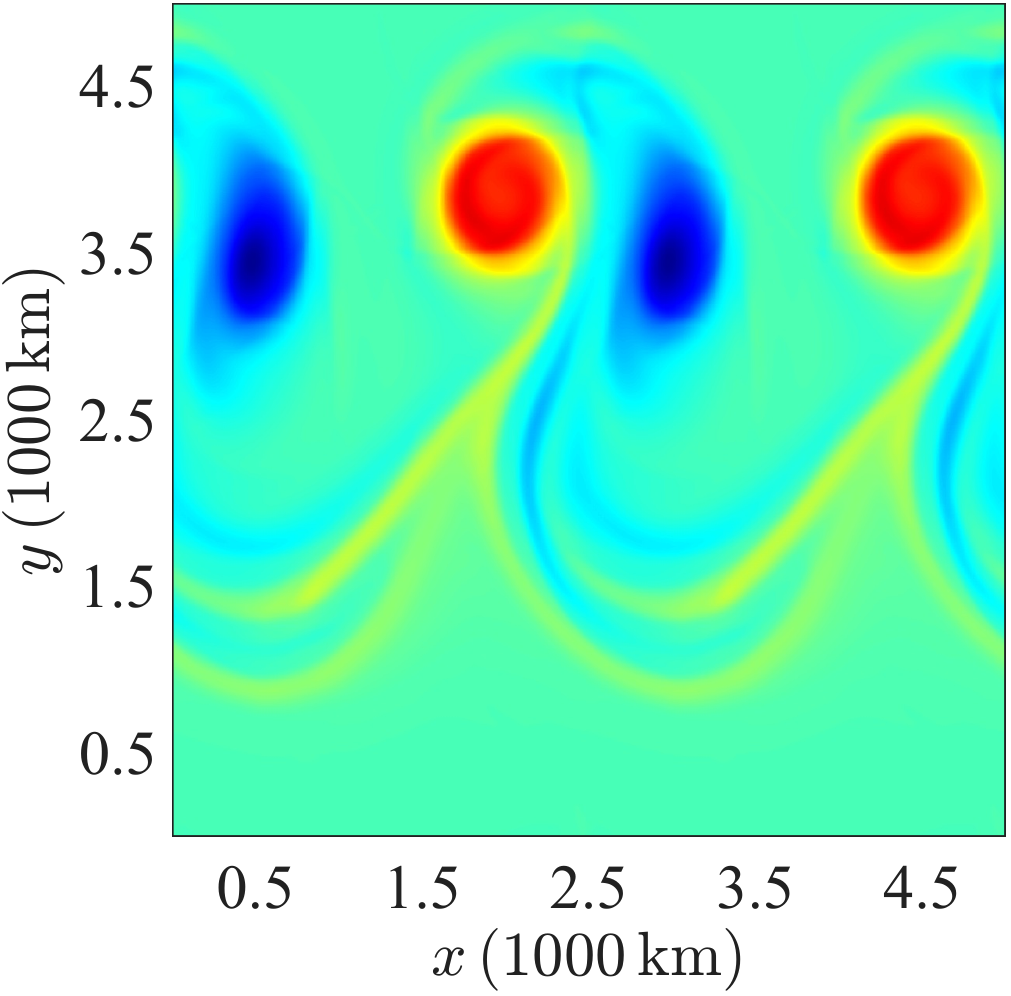}\hspace{5mm}
\includegraphics[height=0.18\textheight]{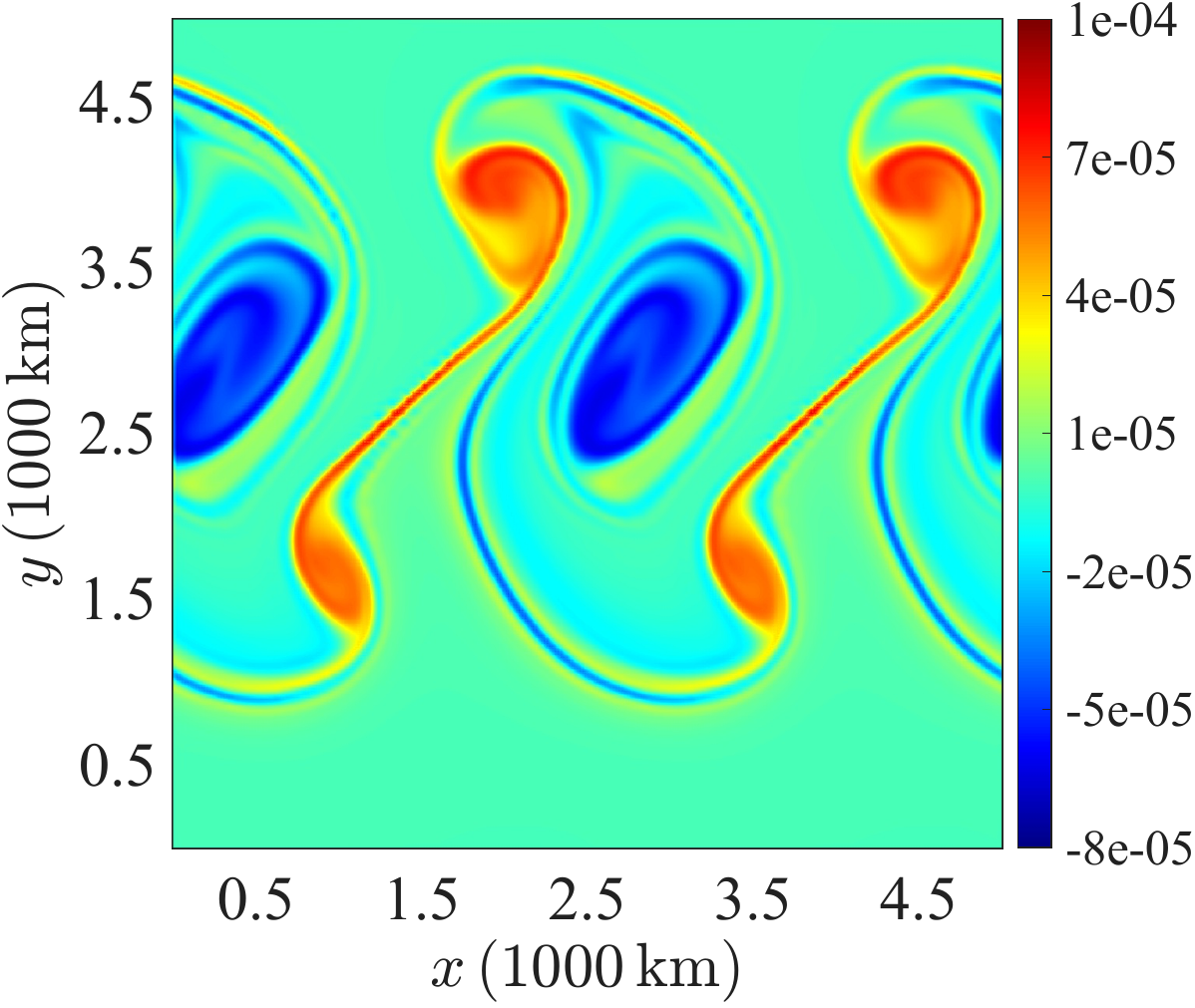}
\caption{\sf Example 5.4 (RSW system): Snapshots of the vorticity $\omega$ computed by the AP DF-FV (left), Explicit (middle), and
fifth-order WENO (right) schemes on a uniform $300\times300$ mesh.\label{Fig5}}
\end{figure}
	
We then run the same simulations for the TRSW system (that is, for the initial data \eref{5.2}) on finer meshes with $400\times400$ and 
$600\times600$ uniform cells, and plot the resulting vorticities in Figure \ref{Fig6} (left and middle). As one can see, the AP DF-FV method
exhibits excellent convergence, whereas the Explicit scheme yields results that vary with mesh resolution. To verify the accuracy of the AP
DF-FV method, a reference solution computed using the explicit fifth-order WENO scheme is shown in Figure \ref{Fig6} (right). Comparison
with this reference solution demonstrates that the AP DF-FV method clearly resolves the fine vortex structures. Similar to Example 5.3, the
vortices continue to rotate under the influence of the Coriolis force, generating an increasing number of smaller vortices over time. On the
finer $600\times600$ mesh, the AP DF-FV method further improves the achieved resolution. The same observations and conclusions apply to the
buoyancy field $\Theta$ shown in Figure \ref{Fig7}.
\begin{figure}[ht!]
\centering
\includegraphics[height=0.18\textheight]{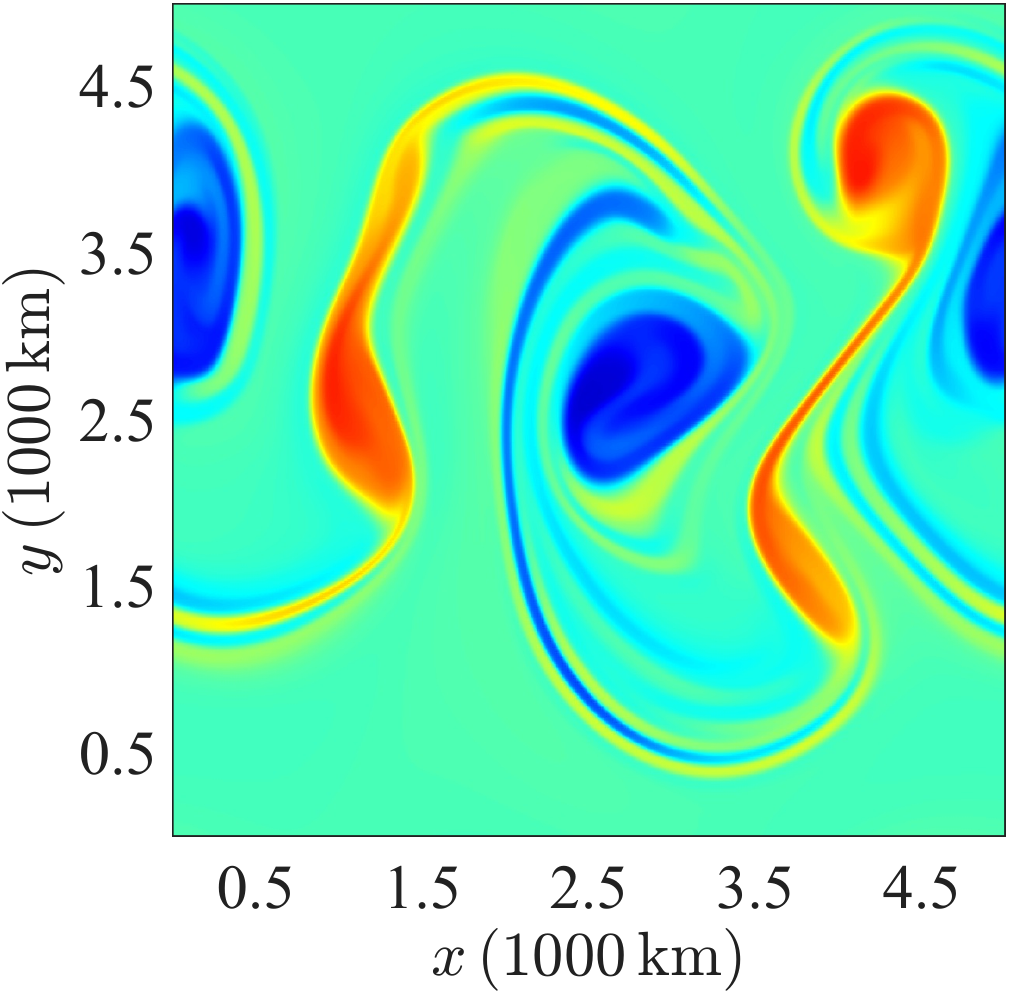}\hspace{5mm}
\includegraphics[height=0.18\textheight]{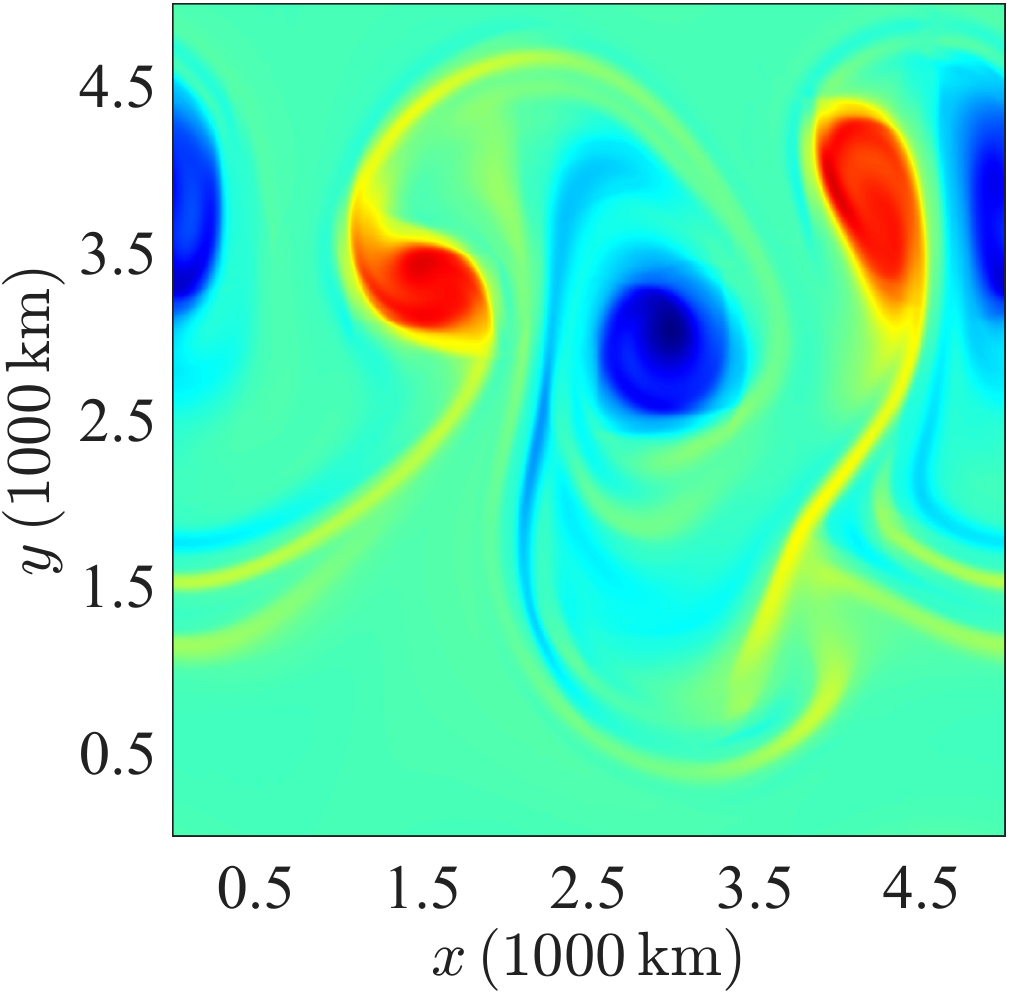}\hspace{5mm}
\includegraphics[height=0.18\textheight]{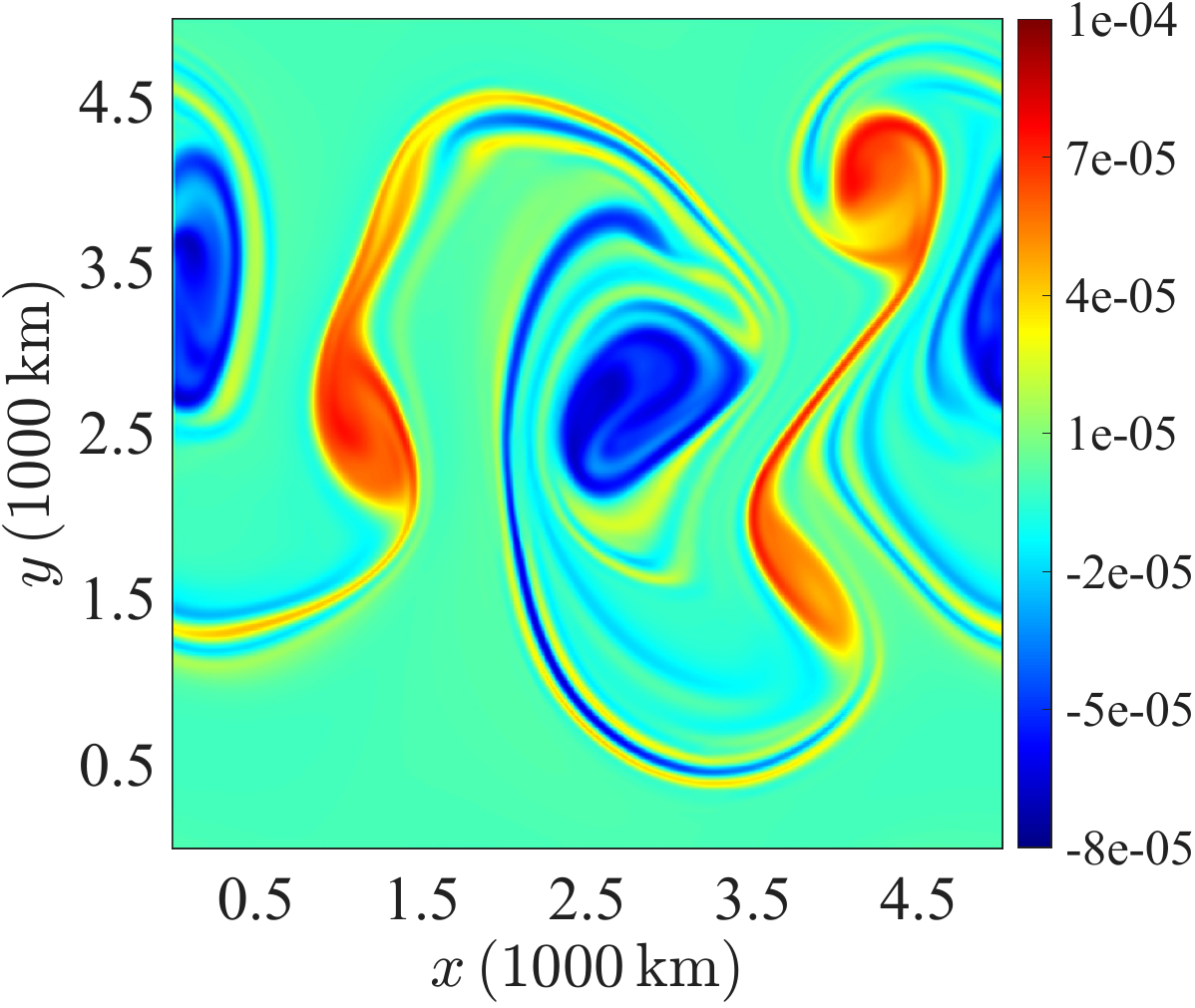}\\[1.5ex]
\includegraphics[height=0.18\textheight]{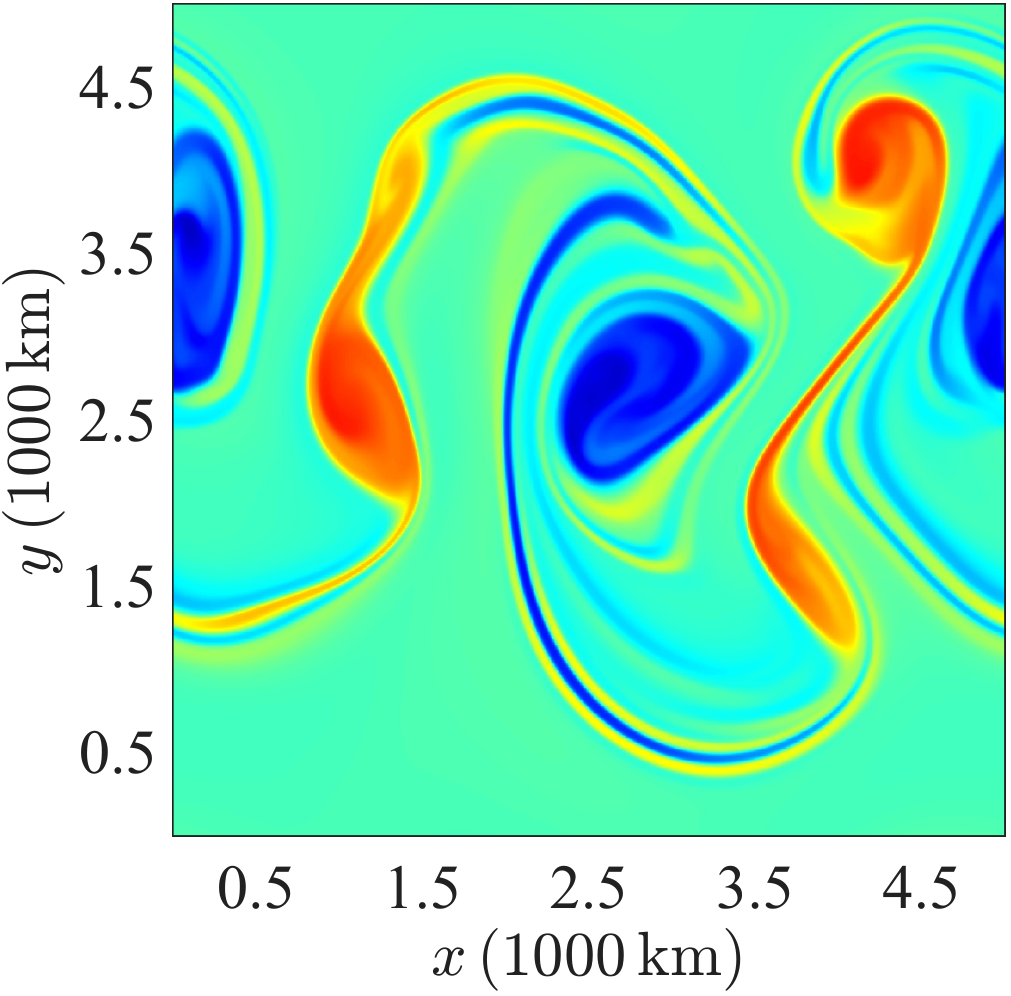}\hspace{5mm}
\includegraphics[height=0.18\textheight]{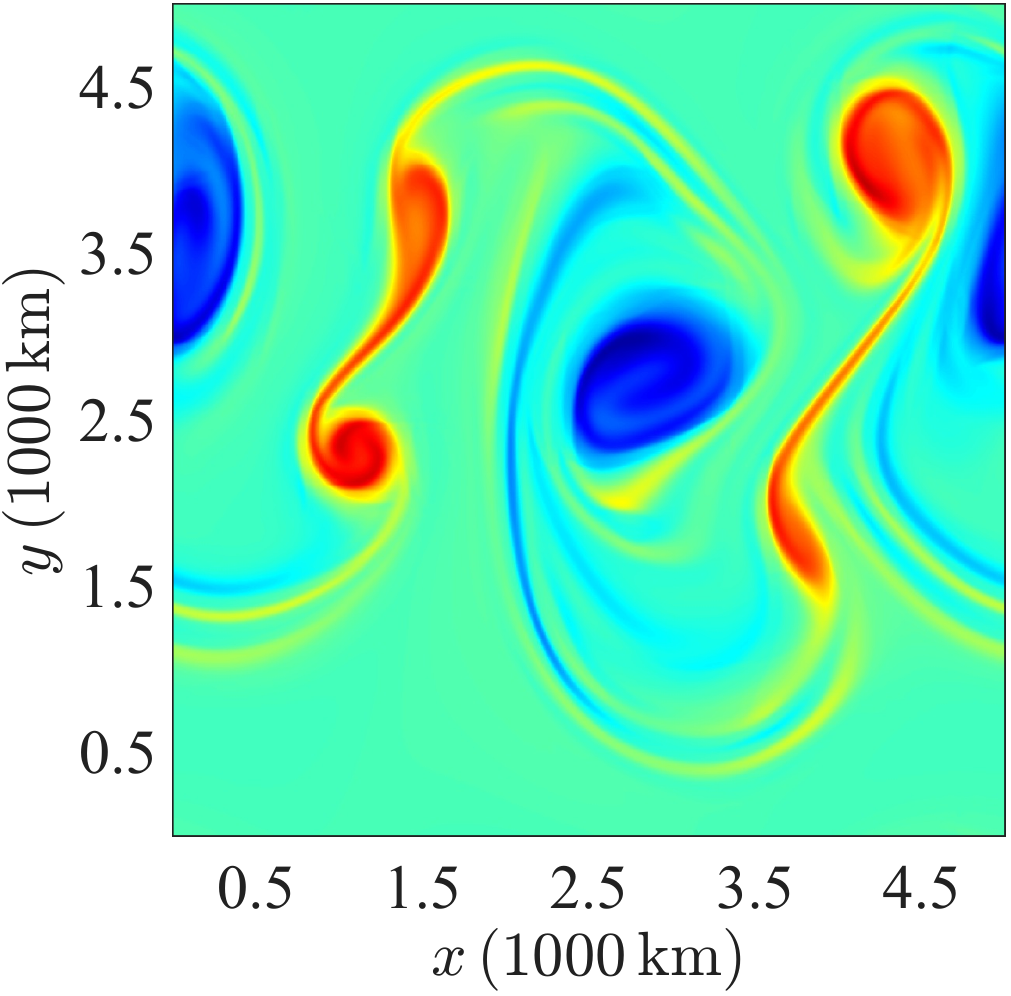}\hspace{5mm}
\includegraphics[height=0.18\textheight]{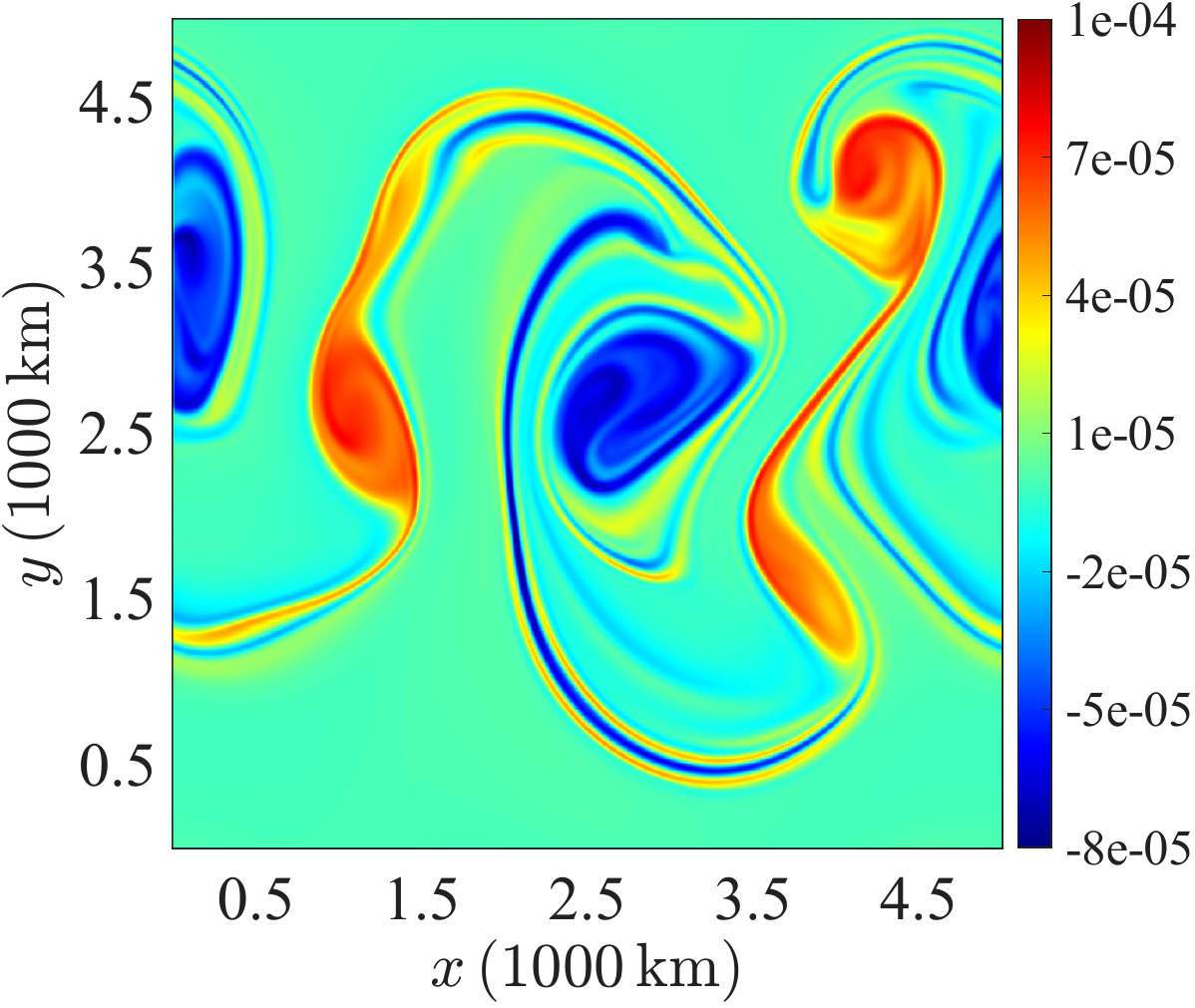}
\caption{\sf Example 5.4 (TRSW system): Snapshots of the vorticity $\omega$ computed by the AP DF-FV (left column), Explicit (middle
column), and fifth-order WENO (right column) schemes on $400\times400$ (top row) and $600\times600$ (bottom row) uniform meshes.
\label{Fig6}}
\end{figure}
\begin{figure}[ht!]
\centering
\includegraphics[height=0.18\textheight]{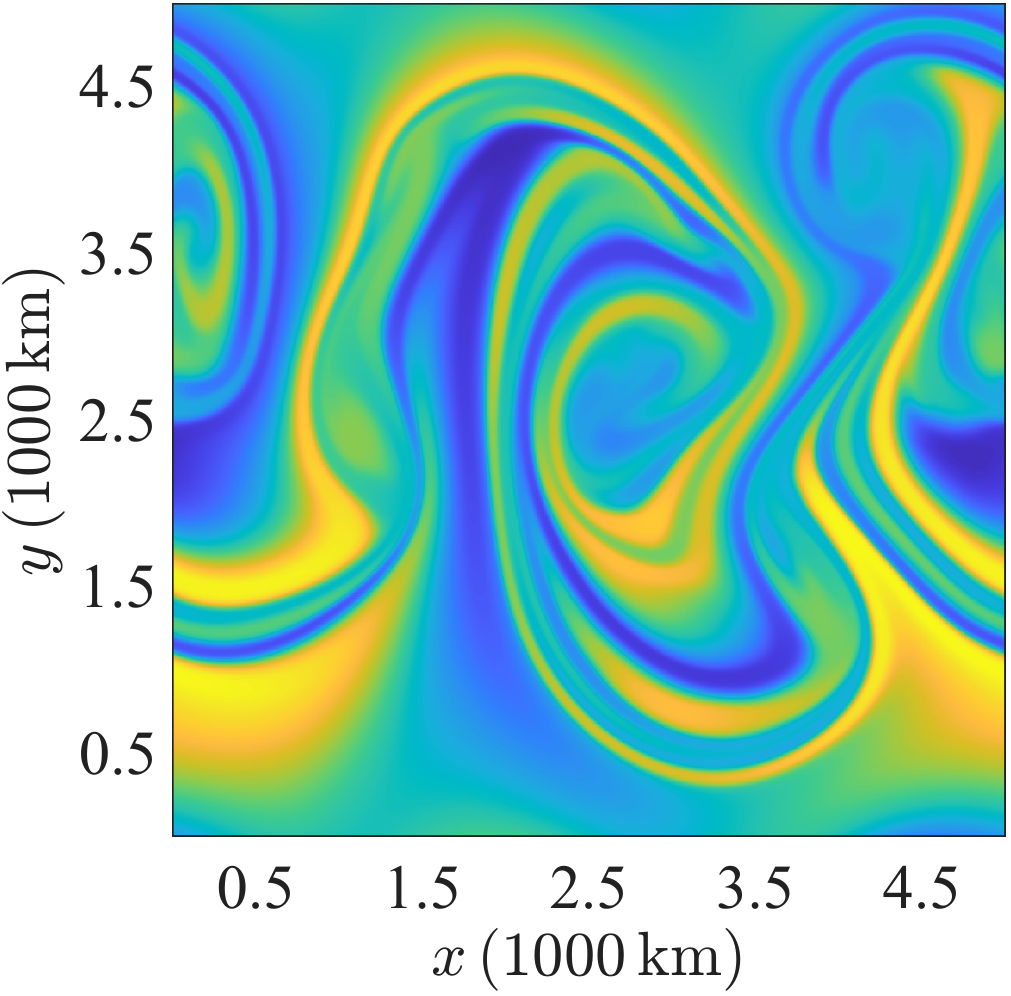}\hspace{5mm}
\includegraphics[height=0.18\textheight]{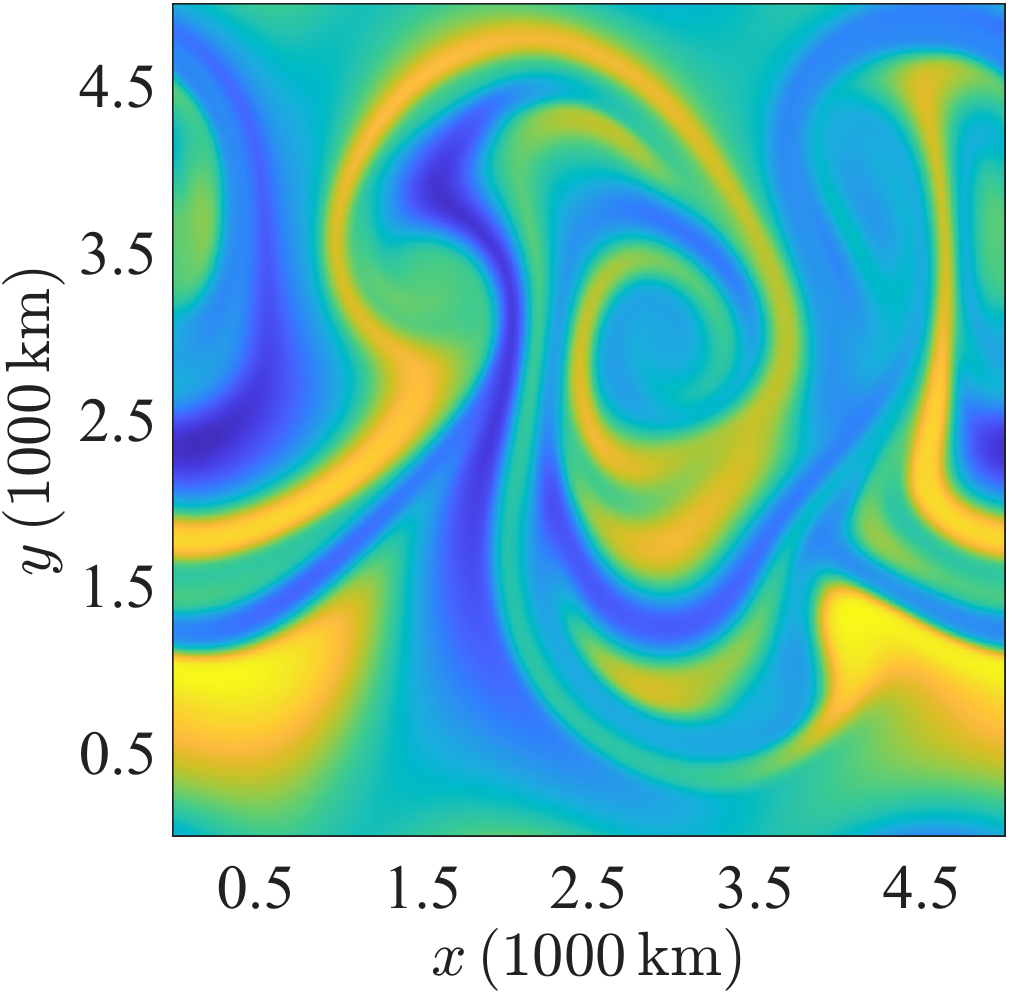}\hspace{5mm}
\includegraphics[height=0.18\textheight]{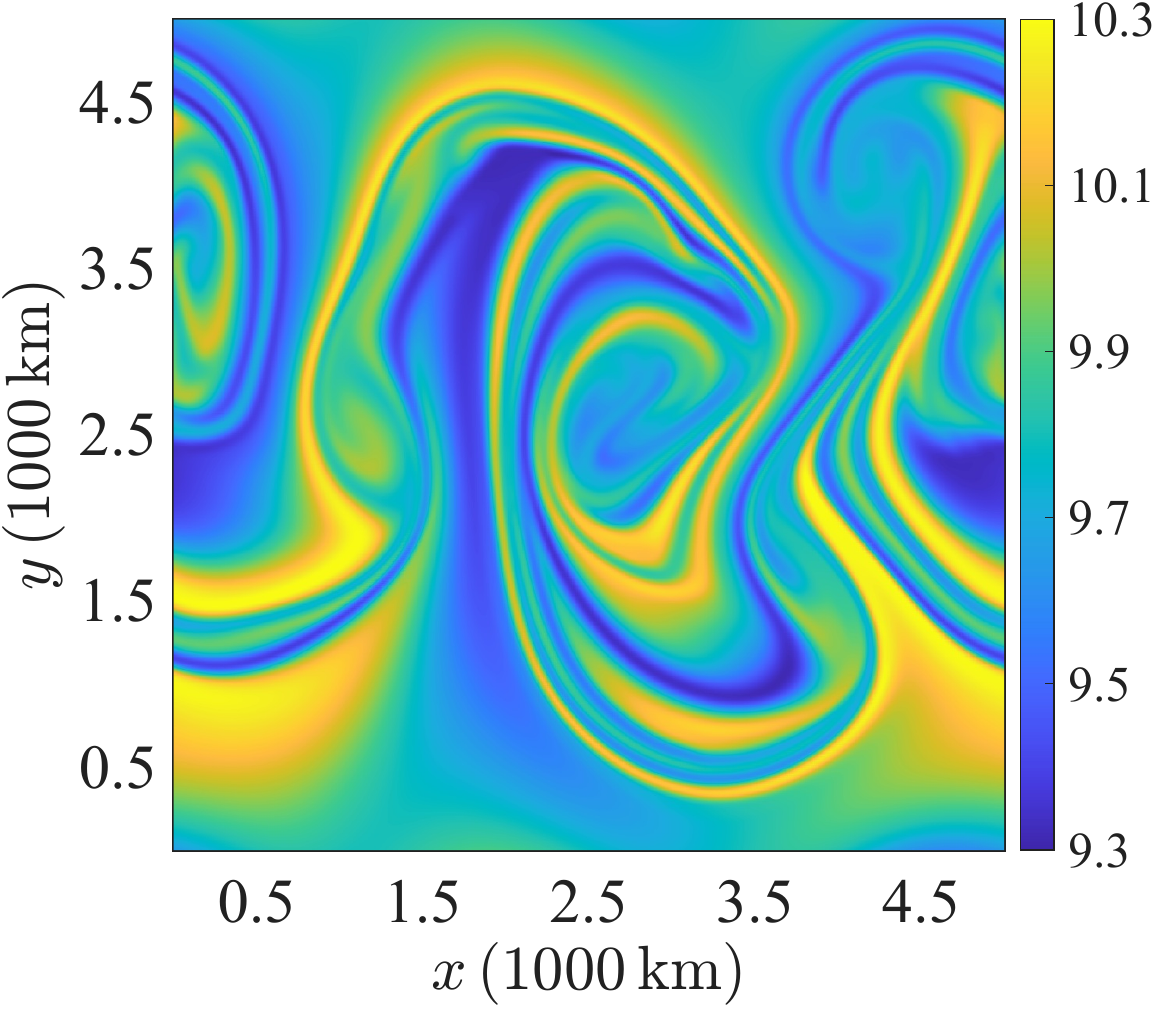}\\[1.5ex]
\includegraphics[height=0.18\textheight]{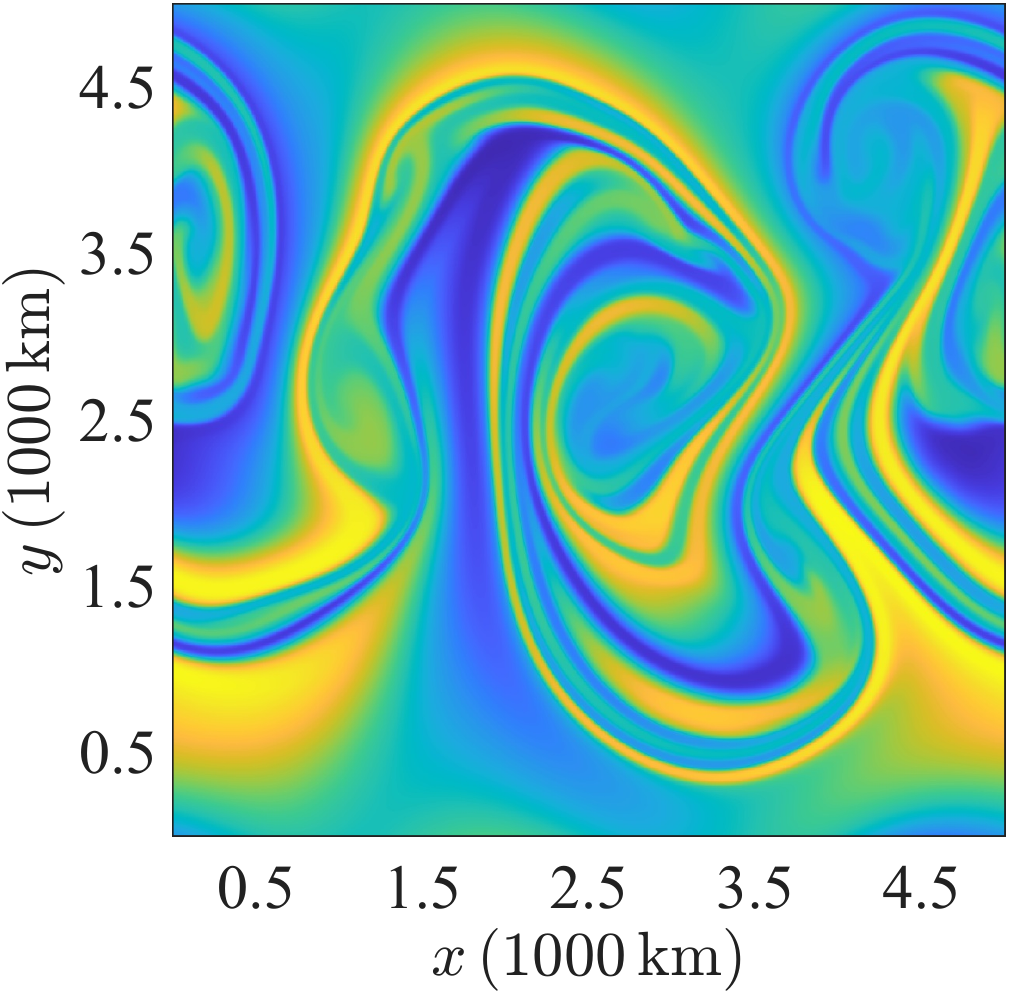}\hspace{5mm}
\includegraphics[height=0.18\textheight]{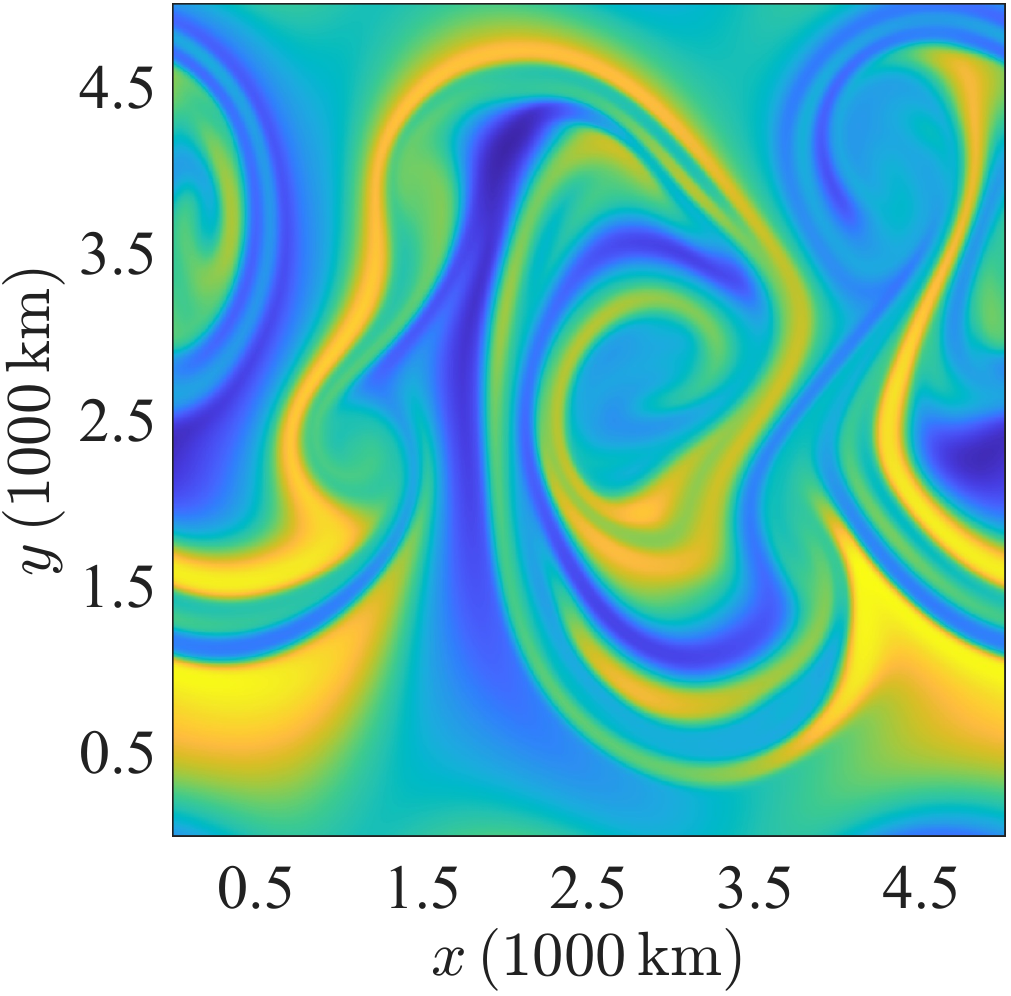}\hspace{5mm}
\includegraphics[height=0.18\textheight]{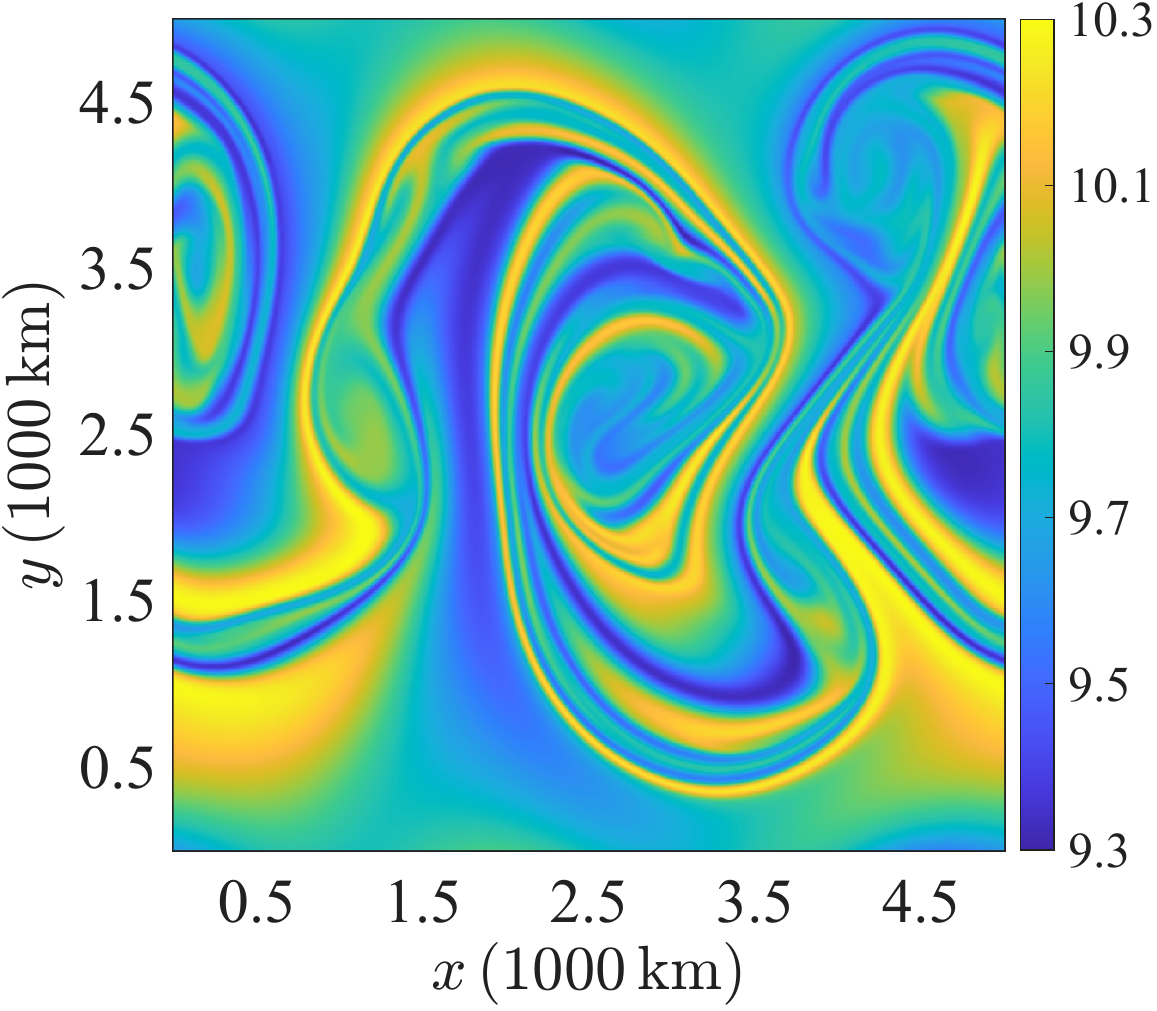}
\caption{\sf Example 5.4: The same as in Figure \ref{Fig6}, but for the buoyancy field $\Theta$.\label{Fig7}}
\end{figure}
	
In addition, we compare the efficiency of the AP DF-FV, Explicit, and WENO schemes. To this end, we measure the CPU times consumed by these 
schemes on the $600\times600$ mesh simulation and include the obtained results together with the total number of time steps and average
cost per time step in the first three rows of Table \ref{table4}. As one can see, the fifth-order WENO scheme is the most computationally 
expensive among the three studied schemes, whereas the proposed AP DF-FV method is the most efficient one. We also observe that the AP DF-FV
results are more accurate than the Explicit results, indicating that the AP DF-FV method clearly outperforms its Explicit counterpart.
\begin{table}[!ht]
\centering
\begin{tabular}{|c|c|c|c|c|}
\hline
method&mesh&number of time steps&average cost per time step (s) &CPU time (s)\\ 
\hline
AP DF-FV&$600\times600$&11290&0.764&8625\\ 
\hline
Explicit&$600\times600$&50707&0.234&11880\\ 
\hline
WENO&$600\times600$&50707&4.159&210898\\ 
\hline
Explicit&$540\times540$&45668&0.189&8644\\
\hline		
WENO&$240\times240$&20263&0.417&8454\\
\hline
\end{tabular}
\caption{\sf Example 5.4: The total number of time steps, average cost per time step, and CPU times consumed by the AP DF-FV, Explicit, and 
WENO methods on different meshes.\label{table4}}
\end{table}
	
In order to better compare the studied three schemes, we compute solutions using the Explicit and WENO schemes on coarser meshes with 
$540\times540$ and $240\times240$ cells, respectively. On these meshes, the CPU times consumed by these two explicit schemes are about the 
same as the CPU time consumed by the AP DF-FV method on the original $600\times600$ mesh (see rows $1$, $4$, and $5$ in Table \ref{table4}).
Notice that now the average cost per time step is the highest for the AP DF-FV method, but it is compensated by its less severe time-step
restriction compared with its explicit counterparts.
	
In Figure \ref{Fig8}, we plot the vorticity $\omega$ and buoyancy $\Theta$ computed by the AP DF-FV, Explicit, and WENO schemes on the
meshes indicated in rows $1$, $4$, and $5$ in Table \ref{table4}. As one can see, the AP DF-FV results are the sharpest, clearly
demonstrating the superiority of the proposed AP DF-FV method in the TQG regime.
\begin{figure}[ht!]
\centering
\includegraphics[height=0.18\textheight]{ShearFlow/DFFV-omega-600Cells-10d}\hspace{5mm}
\includegraphics[height=0.18\textheight]{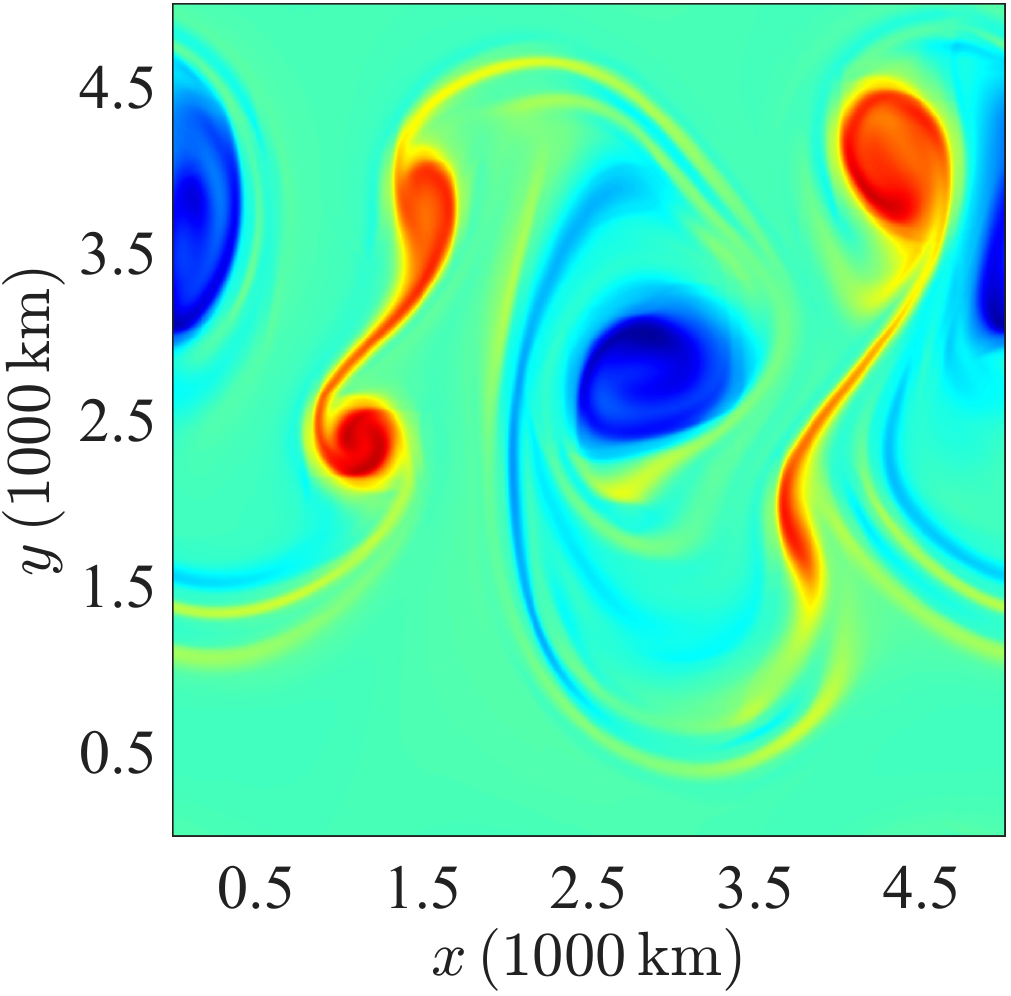}\hspace{5mm}
\includegraphics[height=0.18\textheight]{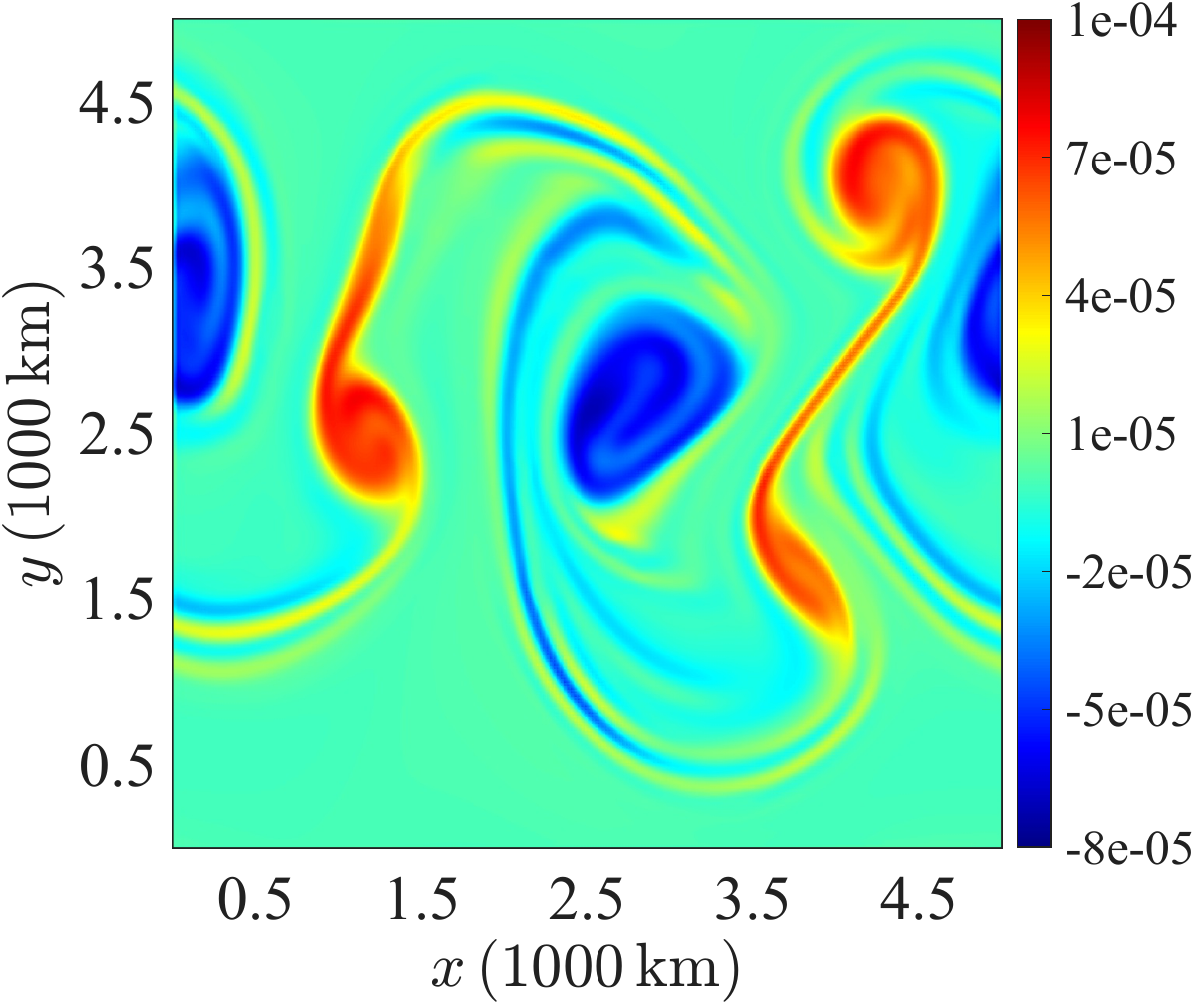}\\[2ex]
\hspace*{-0.25cm}\includegraphics[height=0.18\textheight]{ShearFlow/DFFV-Theta-600Cells-10d}\hspace{5mm}
\includegraphics[height=0.18\textheight]{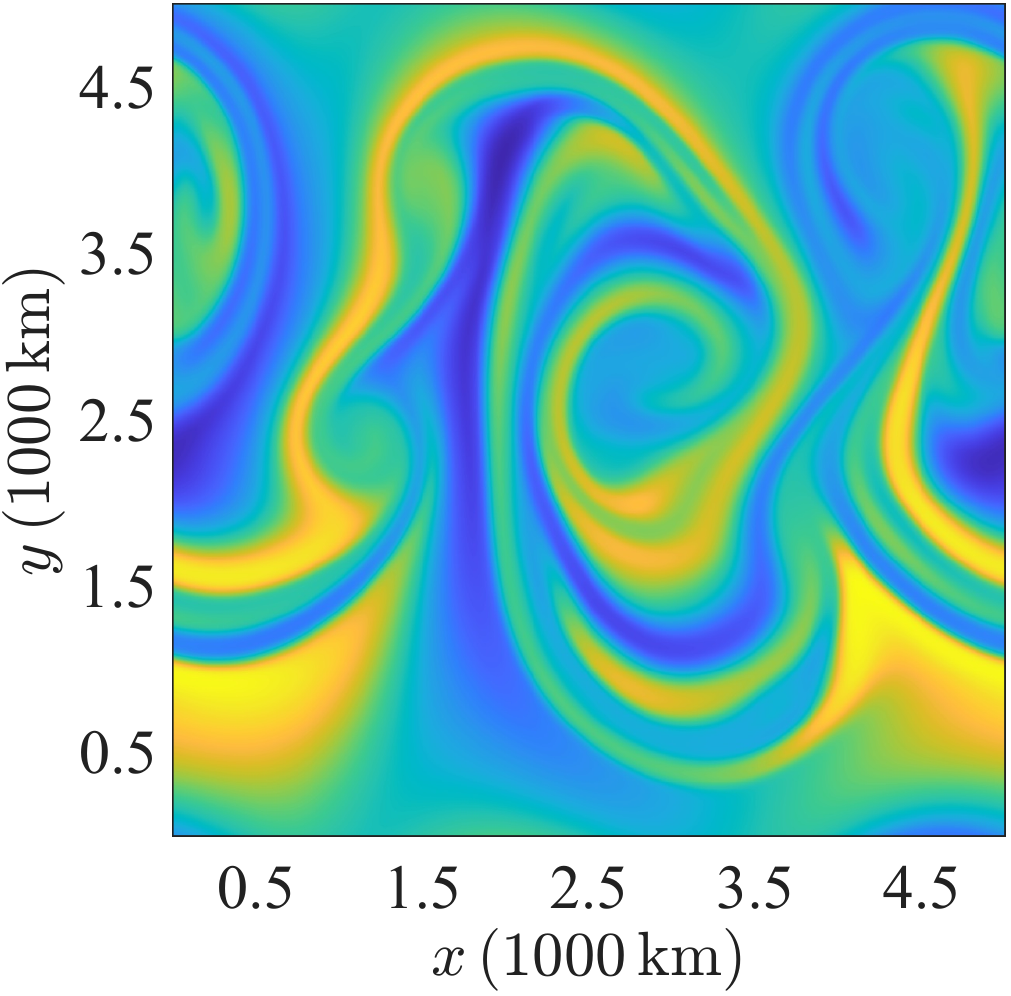}\hspace{5mm}
\includegraphics[height=0.18\textheight]{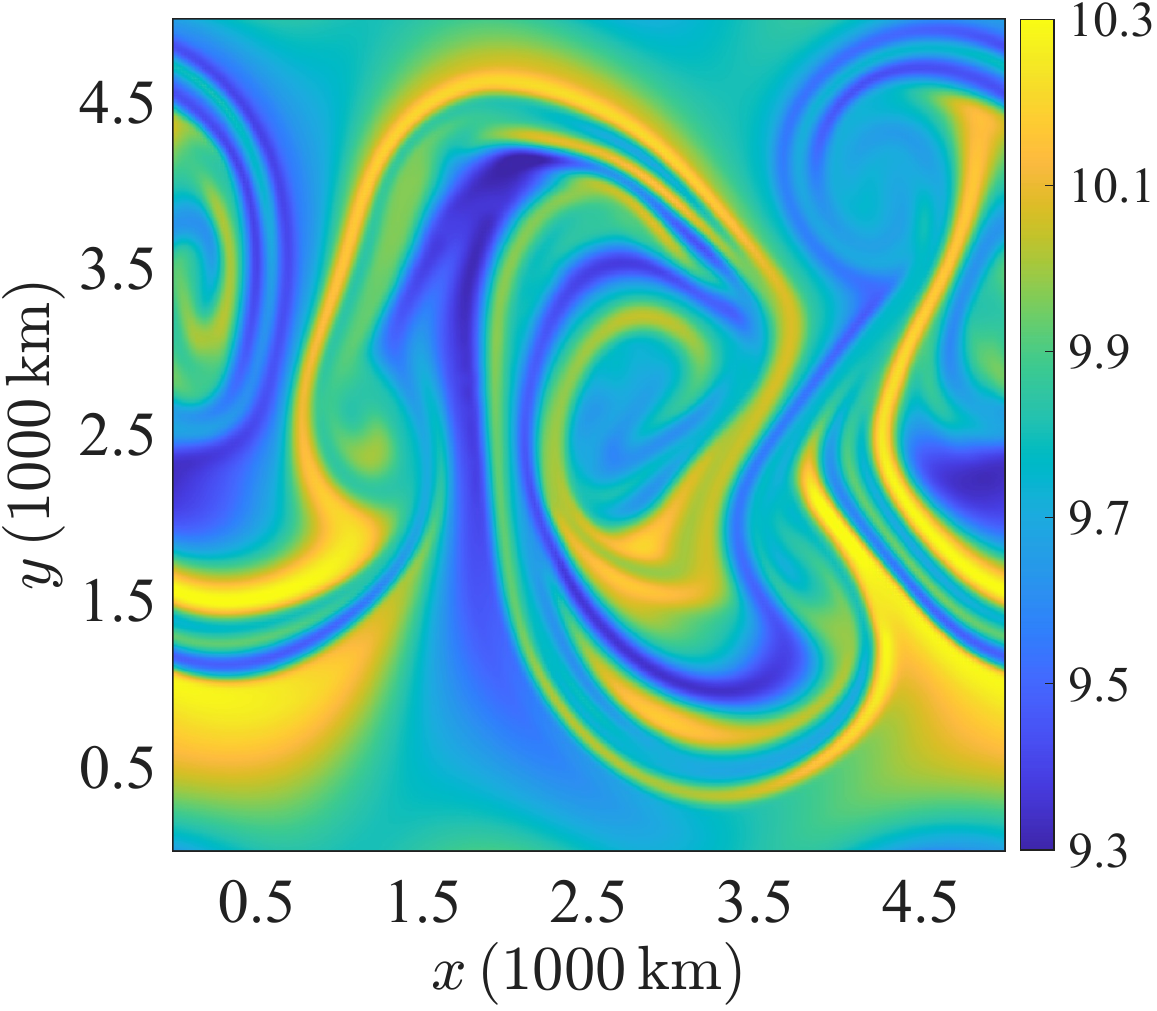}
\caption{\sf Example 5.4: Snapshots of the vorticity $\omega$ (top row) and buoyancy $\Theta$ (bottom row) computed by the AP DF-FV (left
column), Explicit (middle column), and fifth-order WENO (right column) methods on the meshes indicated in rows $1$, $4$, and $5$ in Table
\ref{table4}. On these meshes, all studied methods require approximately the same CPU time.\label{Fig8}}
\end{figure}
\end{example}
	
\begin{example}({\bf Anticyclonic Propagation in the $\beta$-Plane})
In the final example, we consider an initially symmetric vortex propagating westward in the $\beta$-plane with $f(y)=f_0+\beta y$ with
$f_0=6.1635\times10^{-5}\,{\rm s^{-1}}$ and $\beta=2.0746\times10^{-11}\,{\rm m^{-1}s^{-1}}$. The setup is based on the example studied in
\cite{LeRoux_2008,Zhang_2026} for the RSW model. The gravitational acceleration is $\g=9.81\,{\rm m/s^2}$. The mean depth of the fluid layer
is $H_0=163.1\,{\rm m}$. The initial data, prescribed in the computational domain $[-L_x,L_x]\times[-L_y,L_y]$ subject to free boundary
conditions, are
\begin{equation*}
\begin{aligned}
&h(x,y,0)=H_0+A\,{\rm e}^{-\frac{x^2+y^2}{D^2}},\quad&&u(x,y,0)=\frac{2A\g}{f_0+\beta y}\,\frac{y}{D^2}\,{\rm e}^{-\frac{x^2+y^2}{D^2}},\\
&\Theta(x,y,0)=\g\left(1-\frac{A}{H_0}\,{\rm e}^{-\frac{x^2+y^2}{D^2}}\right),\quad
&&v(x,y,0)=-\frac{2A\g}{f_0+\beta y}\,\frac{x}{D^2}\,{\rm e}^{-\frac{x^2+y^2}{D^2}},
\end{aligned}
\end{equation*}
where $A=0.95\,{\rm m}$, $D=130\,{\rm km}$, $L_x=1000\,{\rm km}$, and $L_y=600\,{\rm km}$. To nondimensionalize the problem, we take the
reference values $L_0=1000\,{\rm km}$, $\Theta_0=9.81\,{\rm m/s^2}$, and $V_0=1\,{\rm m/s}$. Using these scales, the corresponding
dimensionless parameters $\Ro$, $\bar\beta$, and $\Bu$ become
\begin{equation*}
\Ro=\eps=\frac{V_0}{L_0f_0}\approx0.016,\quad\bar\beta=20.746,\quad\Bu=\nu=\frac{\Theta_0H_0}{\big(L_0f_0\big)^2}\approx0.421,
\end{equation*}
indicating that the flow is in the TQG regime. We note that in the Northern Hemisphere ($f_0>0$), the secondary vorticity structure induced
by the $\beta$-effect causes the cyclonic vortex to follow a curved southwestward trajectory.
	
We compute the solution by the proposed AP DF-FV method on three uniform meshes with $400\times400$, $600\times600$, and $900\times900$
cells. The buoyancy field $\Theta$ obtained at times $t=20\,{\rm d}$ and $30\,{\rm d}$ are shown in Figure \ref{Fig9}. As one can observe,
the proposed AP DF-FV method robustly captures the characteristic features of the $\beta$-drift, including the secondary circulation, the
Rossby wave tail, and the southwestward migration of the vortex core across all mesh resolutions. Most notably, the mesh refinement study 
indicates strong convergence of the proposed AP DF-FV method. To further demonstrate this, we plot the 1-D slices of $\Theta$ along $x=0$ at
times $t=20\,{\rm d}$ and $30\,{\rm d}$ in Figure \ref{Fig10}. These results demonstrate the high accuracy, robustness, and mesh-consistent
convergence of the proposed AP DF-FV method in capturing complex geophysical vortex dynamics.
\begin{figure}[ht!]
\centering
\includegraphics[height=0.18\textheight]{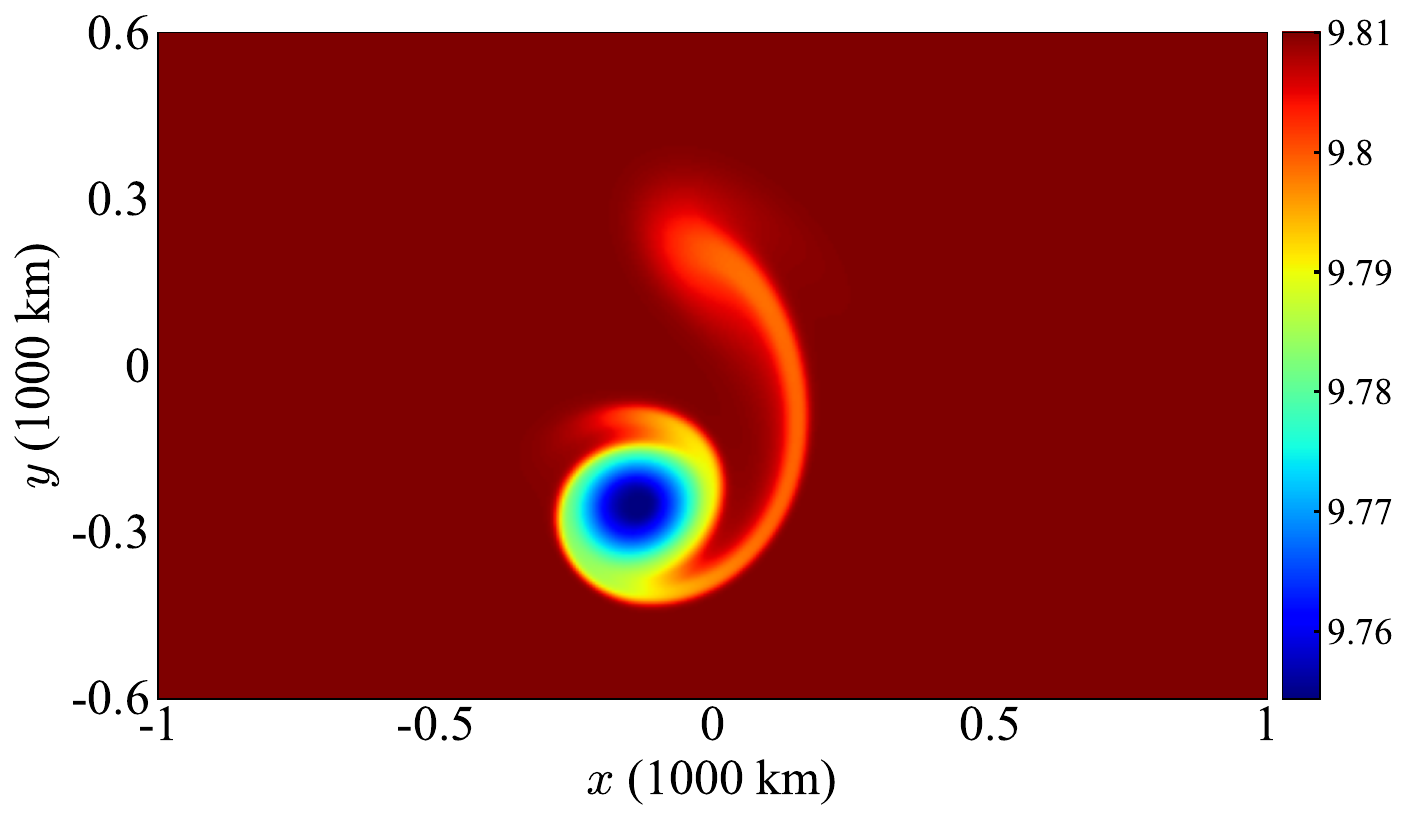}\hspace{5mm}
\includegraphics[height=0.18\textheight]{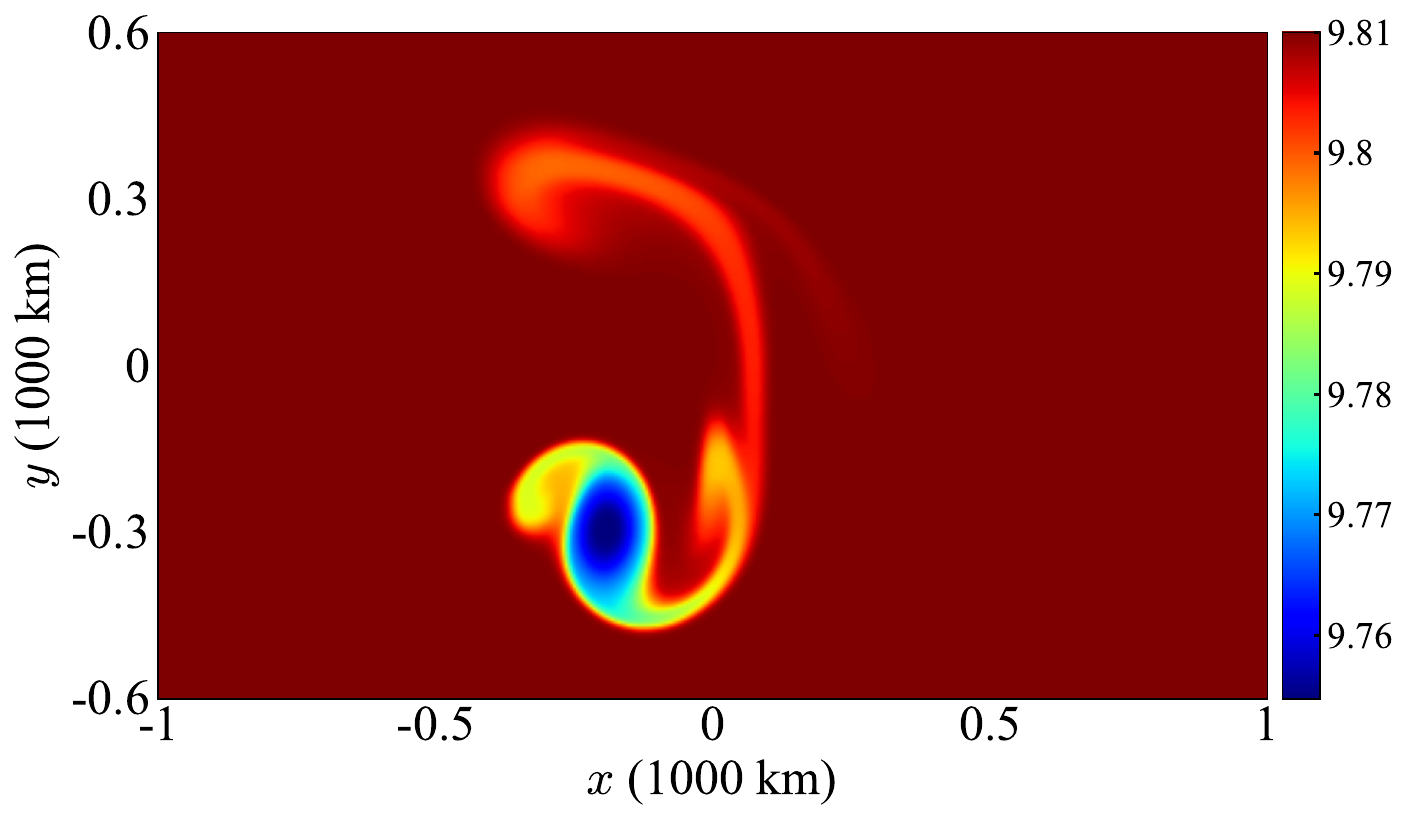}\\[1.2ex]
\includegraphics[height=0.18\textheight]{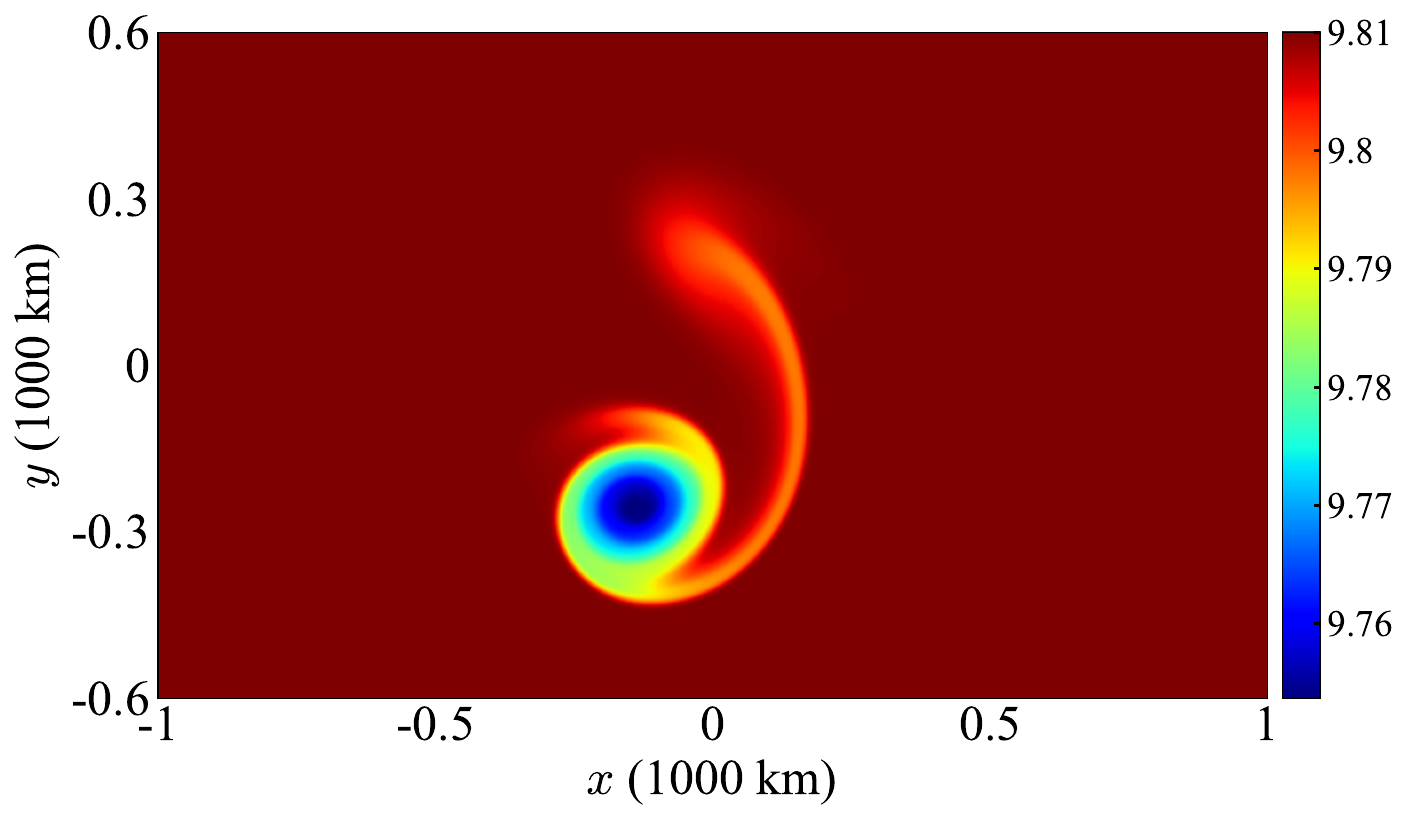}\hspace{5mm}
\includegraphics[height=0.18\textheight]{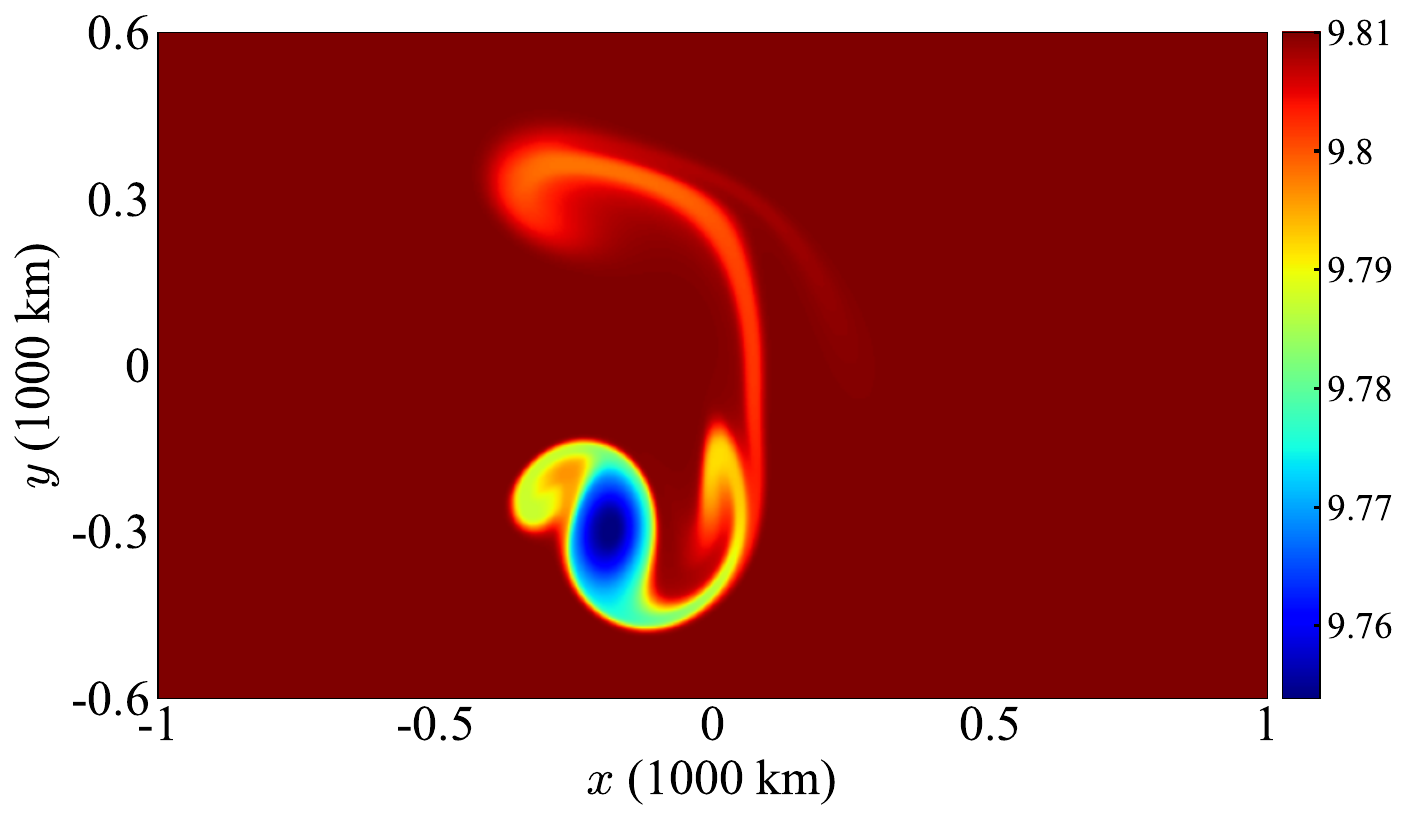}\\[1.2ex]
\includegraphics[height=0.18\textheight]{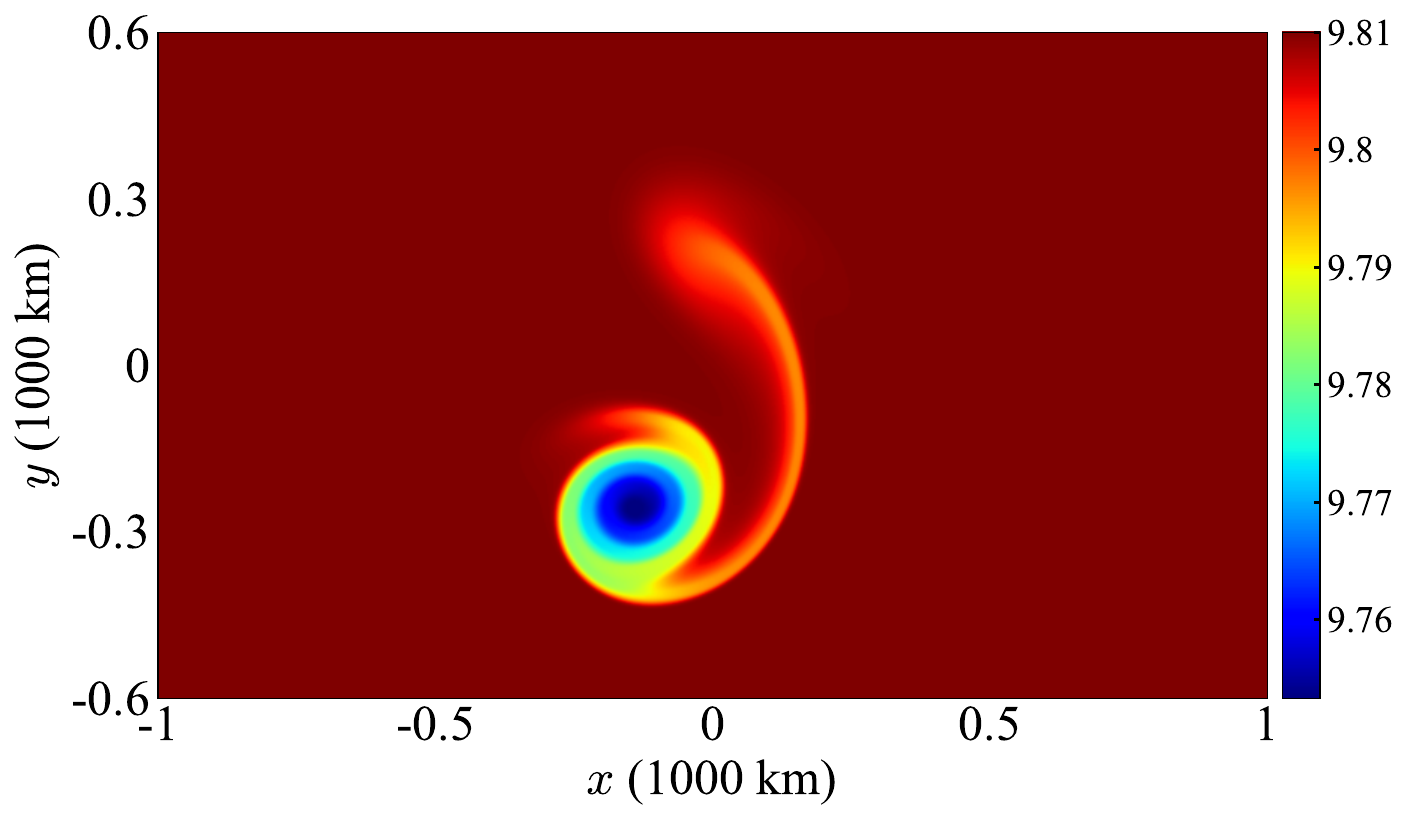}\hspace{5mm}
\includegraphics[height=0.18\textheight]{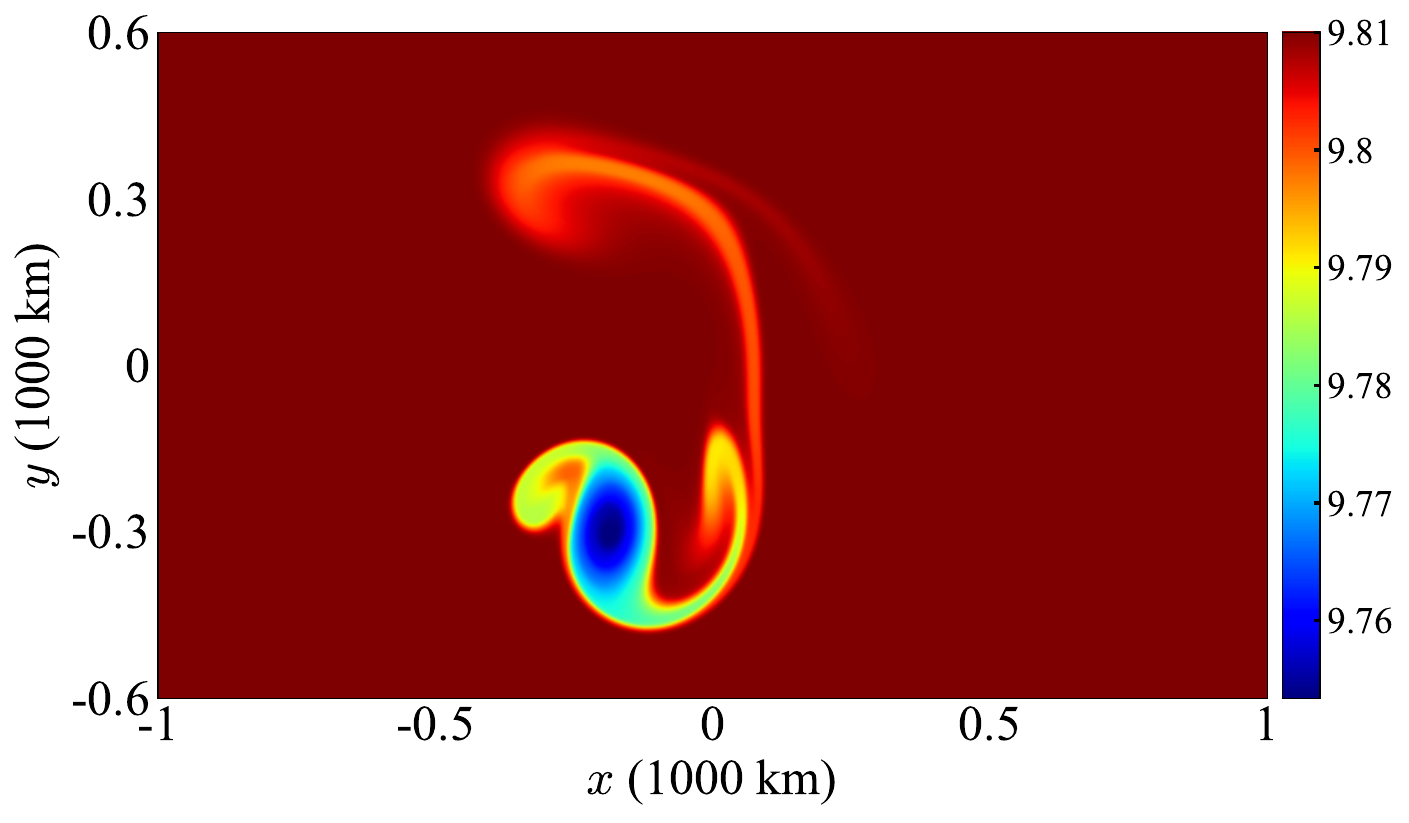}
\caption{\sf Example 5.5: Time snapshots of the buoyancy field $\Theta$ computed by the proposed AP DF-FV method on $400\times400$ (top
row), $600\times600$ (middle row), and $900\times900$ (bottom row) meshes at times $t=20\,{\rm d}$ (left column) and $30\,{\rm d}$ (right
column).\label{Fig9}}
\end{figure}
\begin{figure}[ht!]
\centering
\includegraphics[height=0.22\textheight]{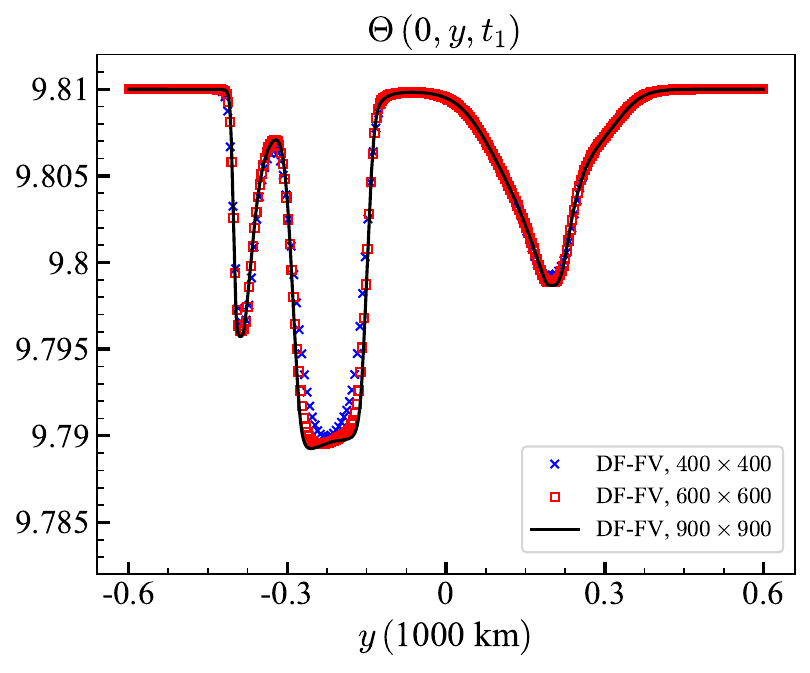}\hspace{1cm}
\includegraphics[height=0.22\textheight]{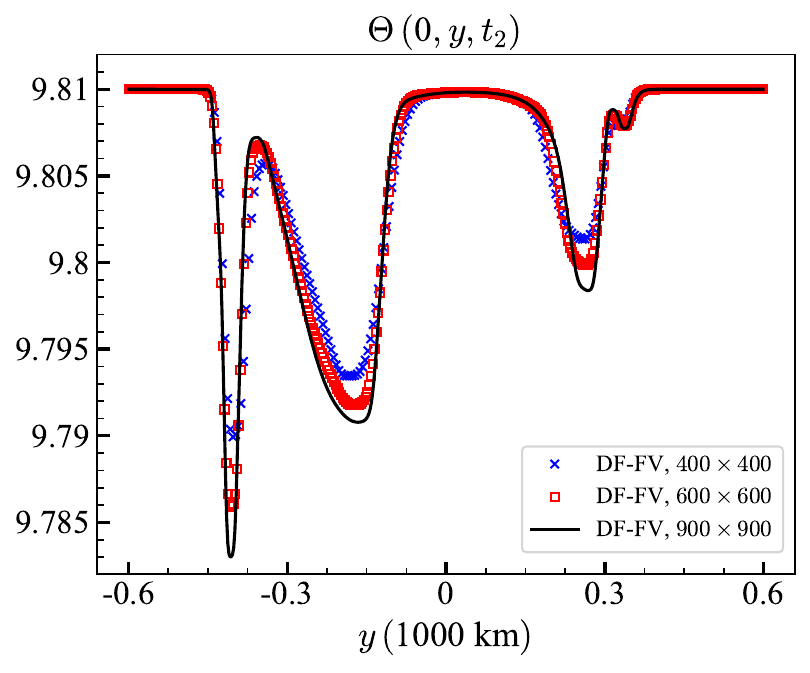}
\caption{\sf Example 5.5: 1-D slices of the $\Theta$ along $x=0$ at $t_1=20\,{\rm d}$ (left) and $t_2=30\,{\rm d}$ (right).\label{Fig10}}
\end{figure}
\end{example}

\section{Conclusions}\label{sec6}	 
This paper presents a new approach for solving the nondimensional TRSW equations using an AP DF-FV method that performs effectively across a
broad range of Rossby numbers. This method effectively combines the asymptotic consistency of the primitive formulation in low-Rossby-number
regimes with the shock-capturing robustness of the conservative formulation in high-Rossby-number regimes. To develop a formulation that is
both easily solvable and AP for the primitive system, the proposed AP DF-FV method augments the system with the nonstiff potential vorticity
equation. In addition, a post-processing is introduced to couple the solutions of the primitive and conservative systems, resulting in a
method that can accurately handle all-Rossby-number regimes.

The developed AP DF–FV method has been tested on several numerical examples illustrating that the proposed scheme is capable of accurately
resolving shock waves without spurious oscillations in high-Rossby-number regimes, achieving efficiency comparable to that of explicit
methods. At the same time, it effectively preserves the correct asymptotic limits and exhibits significantly improved computational
efficiency in low-Rossby-number regimes when compared with explicit schemes.

\begin{acknowledgment}
The work of A. Chertock was supported in part by NSF grant DMS-2208438. The work of A. Kurganov was supported in part by NSFC grant
W2431004. The work of L. Micalizzi was supported in part by the LeRoy B. Martin, Jr. Distinguished Professorship Foundation.
\end{acknowledgment}
	
\bibliographystyle{siamnodash}
\bibliography{TRSW}
	
\end{document}